\documentclass[12pt,a4paper, reqno]{amsart}
\usepackage{}
\usepackage{txfonts}
\usepackage{color}
\usepackage{amsmath,amssymb,extarrows}
\usepackage[backref=page]{hyperref}
\usepackage{tikz,enumerate}
\usepackage{diagbox}
\usepackage{appendix}
\usepackage{epic}
\usepackage{fourier}
\usepackage{float}
\usepackage{pifont}
\usepackage{bbm}
\usepackage{amsfonts}
\usepackage{amsxtra}
\usepackage{amssymb}
\usepackage{amscd}
\usepackage{amsthm}
\usepackage{mathrsfs}
\usepackage{hyperref}
\usepackage{leqno}
\usepackage{multirow}
\usepackage{tikz}
\usepackage{tikz-cd}
\usepackage{CJK}
\makeatletter

\numberwithin{equation}{section}
\newtheorem*{theorem*}{Theorem}
\newtheorem{lemma}{Lemma}[section]
\newtheorem{proposition}[lemma]{Proposition}
\newtheorem{remark}[lemma]{Remark}
\newtheorem{example}[lemma]{Example}

\newtheorem{theorem}[lemma]{Theorem}
\newtheorem{definition}[lemma]{Definition}

\newtheorem{corollary}[lemma]{Corollary}

\newtheorem*{question*}{Question}
\newtheorem*{assumption*}{Assumption}
\newtheorem*{axiom*}{Axiom}

\newtheorem*{theorem*1}{Theorem (\ref{theta1})}
\newtheorem*{theorem*2}{Theorem (\ref{theta2})}
\newtheorem*{theorem*3}{Theorem (\ref{theta3})}
\newtheorem*{theorem*4}{Theorem (\ref{theta4})}
\newtheorem*{proposition*5}{Proposition (\ref{theta5})}
\newtheorem*{proposition*6}{Proposition (\ref{theta6})}
\baselineskip 20pt \textwidth 18cm  \theoremstyle{plain}
\usepackage[all,cmtip]{xy}

\renewcommand{\Im}{\operatorname{Im}}

\newcommand{\Ind}{\operatorname{Ind}}

\newcommand{\Ha}{\operatorname{H}}
\newcommand{\Id}{\operatorname{Id}}

\newcommand{\C}{\mathbb C}
\newcommand{\F}{\mathbb F}

\newcommand{\R}{\mathbb R}
\newcommand{\Z}{\mathbb Z}

\newcommand{\GL}{\operatorname{GL}}

\newcommand{\GSp}{\operatorname{GSp}}

\newcommand{\Mp}{\operatorname{Mp}}

\newcommand{\Mm}{\operatorname{M}_m}
\newcommand{\Ma}{\operatorname{M}_{2m}}
\newcommand{\SL}{\operatorname{SL}}
\newcommand{\Oa}{\operatorname{O}}

\newcommand{\Sp}{\operatorname{Sp}}

\newcommand{\U}{\operatorname{U}}

\newcommand{\Span}{\operatorname{Span}}

\newcommand{\id}{\operatorname{Id}}

\newcommand{\supp}{\operatorname{supp}}

\newcommand{\sgn}{\operatorname{sgn}}

\begin{document}

\title{Siegel modular  forms associated to Weil representations}
\author{Chun-Hui Wang}
\address{School of Mathematics and Statistics\\Wuhan University \\Wuhan, 430072,
P.R. CHINA}
\email{cwang2014@whu.edu.cn}
\begin{abstract} 
We study some explicit  Siegel modular forms from Weil representations.   For  the classical theta group $\Gamma_m(1,2)$ with  $m > 1$, there are  some eighth roots of unity associated with  these modular forms, as noted in  the works of Andrianov,  Friedberg,  Maloletkin, Stark, Styer, Richter, and others. We apply $2$-cocycles    introduced by Rao, Kudla, Perrin,  Lion-Vergne, Satake-Takase to investigate  these unities. We extend our study  to the full Siegel group   $\Sp_{2m}(\Z)$ and   obtain two matrix-valued Siegel modular forms  from Weil representations; these  forms arise from  a finite-dimensional representation $\Ind_{\widetilde{\Gamma}'_m(1,2)}^{\widetilde{\Sp}'_{2m}(\Z)} (1_{\Gamma_m(1,2)} \cdot \Id_{\mu_8})^{-1}$, which is related to Igusa's quotient group $\tfrac{\Sp_{2m}(\Z)}{\Gamma_m(4,8)}$.
\end{abstract}
\maketitle

\setcounter{secnumdepth}{3}
\tableofcontents{}

\section{Introduction}\label{In}
 \subsection{Notations and conventions}
 Let $\R$ and $\C$ denote  the usual real numbers and complex numbers.  Let $(W, \langle, \rangle)$ be a symplectic vector space  of  dimension $2m$   over  $\R$.  Let $\Ha(W)=W\oplus \R$ denote the usual Heisenbeg group, defined  by
$$(w, t)\cdot(w',t')=(w+w', t+t'+\tfrac{\langle w,w'\rangle}{2}),
\quad\quad w, w'\in W, t,t'\in \R.$$
 Let   $\Sp(W)$ denote the corresponding symplectic group.  Let $\{e_1, \cdots, e_m; e_1^{\ast}, \cdots, e_m^{\ast}\}$ be a symplectic basis of $W$. Let $X=\Span_\R\{ e_1, \cdots, e_m\}$, $X^{\ast}=\Span_\R\{e_1^{\ast}, \cdots, e_m^{\ast}\}$.

 In $\Sp(W)$, let  $P_{X^{\ast}}(W)=\{ g\in \Sp(W) \mid X^{\ast}g=X^{\ast}\}=\{ \begin{pmatrix} a& b\\ 0& (a^{\ast})^{-1}\end{pmatrix}\in \Sp(W)\}$, $N_{X^{\ast}}(W)=\{ \begin{pmatrix} 1_m& b\\ 0& 1_m\end{pmatrix} \in \Sp(W)\}$, $M_{X^{\ast}}(W)=\{ \begin{pmatrix} a& 0\\ 0& (a^{\ast})^{-1}\end{pmatrix}\in \Sp(W)\}$, $N^{-}_{X^{\ast}}(W)=\{ \begin{pmatrix} 1_m& 0\\ c& 1_m\end{pmatrix} \in \Sp(W)\}$. For a subset $S\subseteq \{1, \cdots, m\}$, let us define an element   $\omega_S$ of  $ \Sp(W)$ as follows: $ (e_i)\omega_S=\left\{\begin{array}{lr}
-e_i^{\ast}& i\in S,\\
 e_i & i\notin S,
 \end{array}\right.$ and $ (e_i^{\ast})\omega_S=\left\{\begin{array}{lr}
e_i^{\ast}& i\notin S,\\
 e_i & i\in S.
 \end{array}\right.$  If  $S=\{1, \cdots, m\}$, we   write  $\omega$ for simplicity. For elements within  $\Sp(W)$, we denote by $u(b)=\begin{pmatrix}
  1_m&b\\
  0 & 1_m
\end{pmatrix}$,  $u_-(c)=\begin{pmatrix}
  1_m&0\\
  c & 1_m
\end{pmatrix}$, $h(a)=\begin{pmatrix}
  a&0\\
  0 & (a^{\ast})^{-1}
\end{pmatrix} $.  Let $S_m$ denote the permutation group on $m$ elements. For $s\in S_m$, let $\omega_s$ be an element of $\Sp(W)$ defined as $(e_i)\omega_s=e_{s(i)}$ and $(e_j^{\ast}) \omega_s=e_{s(j)}^{\ast}$. Let  $\mathfrak{W}_{P_{X^{\ast}}}=\{ \omega_{S_i} \mid S_0=\emptyset \textrm{  or } S_i= \{1, \cdots,i\} \}$ 
and   $\mathfrak{W}=\{ \omega_s\omega_S \mid S \subseteq   \{1, \cdots, m\}, s\in S_m\}$.  Given the basis $\{ e_1, \cdots, e_m; e_1^{\ast}, \cdots, e_m^{\ast}\}$ of $W$, we identity $\Sp(W)$ with $\Sp_{2m}(\R)$. We use the same notations for elements of these two groups and replace 
$W$ by $\R$ for the corresponding subgroups.   Let $N_0(\R)=\{ h(a)\in \Sp_{2m}(\R)\mid  a =\begin{pmatrix} 1& * & *\\ & \ddots & * \\ & & 1\end{pmatrix} \in \GL_m(\R)\}$.  Let $N^-_0(\R)= \omega^{-1} N_0(\R)\omega $,  $N(\R)=N_{X^{\ast}}(\R)N_0(\R)$ and $N^-(\R)=N_{X^{\ast}}^-(\R)N^-_0(\R)$.

For the elements in $\Mm(\R)$ or $\Ma(\R)$, we denote by  $\epsilon_{ij}(t)$  the $m\times m$-matrix where the entry in the $i$-th row and $j$-th column set to $t$ and all other entries set to $0$. We then define the matrices $u_{ii}(t), u^-_{ii}(t), u_{ij}(t), u_{ij}^-(t)$, $v_{ij}(t)$ ($i\neq j$) as follows: $u_{ii}(t)= \begin{pmatrix} 1_m& \epsilon_{ii}(t) \\0 & 1_m\end{pmatrix}$, $u^-_{ii}(t)= \begin{pmatrix} 1_m& 0 \\ \epsilon_{ii}(-t) & 1_m\end{pmatrix}$,  $u_{ij}(t)= \begin{pmatrix} 1_m& \epsilon_{ij}(t) +\epsilon_{ji}(t)\\0 & 1_m\end{pmatrix} $,and  $u_{ij}^{-}(t)= \begin{pmatrix} 1_m& 0\\\epsilon_{ij}(-t) +\epsilon_{ji}(-t) & 1_m\end{pmatrix} $, $v_{ij}(t)=\begin{pmatrix} 1_m +\epsilon_{ij}(t) & 0\\0 &  1_m +\epsilon_{ji}(-t)\end{pmatrix}$.  If $t=1$, we simply write them as  $u_{ii}$, $u_{ii}^-$, $u_{ij}$, $u_{ij}^-$, $v_{ij}$.

If $\psi$ is a non-trivial character of $\R$ and $a\in \R$,   we will  write   $\psi^a$ for the character: $t\longrightarrow \psi(at)$.  Let $i=\sqrt{-1}$. We will let  $\psi_0$ denote the fixed character of $\R$ defined as: $ t \longmapsto e^{2\pi it}$, for $t\in \R$. Let $\mu_n=\langle e^{\tfrac{2\pi i}{n}}\rangle$,  $ e^{\tfrac{2\pi i}{n}} \in \C^{\times}$. Let $T=\{ e^{i\theta}\mid \theta \in \R\}$. For a finite set  $G$, let $|G|$ denote the cardinality of $G$. On $\R^m$, let $du$ denote the Lebsegue measure. On $\C^m$, let  $dz=du+idv$, and $d(z)=dudv$.
\subsection{Weil representations}
Let $\psi$ be a non-trivial continuous unitary character of $\R$.  According to the Stone-von Neumann's theorem, there exists  only one unitary irreducible complex representation of $\Ha(W)$ with central character $\psi$, known as the Heisenberg representation. Details are in Section  \ref{heisenbergrepm}. 

 According to Weil's work, the Heisenberg representation  can be extended to  a projective representation of $\Sp(W)$, and then  to an actual representation of a $\C^{\times}$-covering group over $\Sp(W)$. This group is called the Metaplectic group, and the representation is called the Weil representation.   In the real field case,  this special  central cover can descend  to an  $8$-degree or  $2$-degree cover over $\Sp(W)$. There are three common models for realizing the Weil representation: the Sch\"odinger model, the Lattice model and the Fock model.  Details are in Section \ref{Weilmodel}.
\subsection{ Siegel modular forms}
Recall that the complex domain 
$$\mathbb{H}_m=\{ x+iy \in  \Mm(\C) \mid x^T=x, y^T=y, y>0\}$$
is called the Siegel upper half space.  There exists an action of $\Sp_{2m}(\R)$ on $\mathbb{H}_m$, given by 
$$z \longrightarrow g (z)= (a z+b)(cz+d)^{-1},$$
for $g=\begin{pmatrix}a & b\\ c& d\end{pmatrix}\in \Sp_{2m}(\R)$, $z\in \mathbb{H}_m$.  Let $\U_m(\C)=\{ g\in \Sp_{2m}(\R)\mid g(i1_m)=i1_m\}=\{\begin{pmatrix}a & b\\ -b& a\end{pmatrix}\in \Sp_{2m}(\R)\}$. 
It is clear that $\mathbb{H}_m \simeq \Sp_{2m}(\R)/\U_m(\C)$. Throughout the text, we will fix an element $z_0=i1_m\in \mathbb{H}_m$.

 Assuming   $m>1$,  following \cite{Fr2,Ge,Kl,Sc},   let $\rho: \GL(m, \C) \to  \GL(V_{\rho})$ be  a rational representation on a finite dimensional complex vector space $V_{\rho}$. Let $\Gamma$ be a discrete group of $\Sp_{2m}(\R)$. Let $\lambda(-)$ be a multiplier system  of  weight $1/2$,  with respect to $\Gamma$. Let $\gamma$ be  a representations of $\Gamma$. Let  $M_{\rho,\gamma}$ be a vector space over $\C$ that  is   a left $\rho$-module and a right $\gamma$-module, comprised of   matrices.
\begin{definition}
 A matrix-valued function $M: \mathbb{H}_m \to M_{\rho,\gamma}$ is called a $M_{\rho,\gamma}$-valued Siegel modular form with a multiplier
system $\lambda(-)$ of  weight $1/2$, with respect to $\Gamma$ if: 
\begin{itemize}
\item[(1)] $M$ is  holomorphic,
\item[(2)]  $M(gz)=\lambda(g)\sqrt{\det(cz+d)} \rho(cz+d) M(z) \gamma(g^{-1})$, for any $g=\begin{pmatrix} a & b\\ c& d\end{pmatrix}\in \Gamma$.
\end{itemize}
\end{definition}
 This definition is mainly for number-theoretic research and giving clear examples. In Langlands programs, it relates to automorphic forms and automorphic representations. For deeper understanding, see Borel and Jacquet's paper \cite{BoJa}, or  Takase's \cite{Ta0}. As shown in \cite{BeCaRaRo}, some central coverings over discrete groups don't split. In such cases, the definition can be extended to a  covering group  $\widetilde{\Sp}_{2m}(\R)$  by replacing $ \U_m(\C)$ and $\Gamma$  with their  covering subgroups  $\widetilde{\U}_{m}(C)$ and $\widetilde{\Gamma}$.
\subsection{The main results}  
\subsubsection{} For $z\in \mathbb{H}_m$, the theta function and the theta vector function associated with the Weil representation are defined as:
\begin{align*}
\theta_{1/2}(z) &= \sum_{n\in \Z^m} e^{i  \pi n z n^T}, \\
\theta_{3/2}(z) &= \sum_{n\in \Z^m} n^T e^{i  \pi n z n^T}.
\end{align*}
Let  $\Gamma_m(1, 2)=\{ g=\begin{pmatrix} a & b \\ c& d\end{pmatrix} \in \Sp_{2m}(\Z) \mid \textrm{ diagronal } a c^T \textrm{ even and } \textrm{ diagronal } b d^T \textrm{ even}\}$. For $g=\begin{pmatrix} a & b \\ c& d\end{pmatrix}\in \Sp_{2m}(\R)$, $z\in \mathbb{H}_m$, $r\in \Gamma_m(1,2)$,  we define:
\begin{itemize}
\item[(1)] $J(g, z)= cz+d$.
\item[(2)] $J_{1/2}(g, z)=   \epsilon(g; z,z_0) \cdot |\det J(g, z)|^{1/2}$.
\item[(3)] $\sqrt{\det(cz+d)}= J_{1/2}(g, z)m_{X^{\ast}}(g)^{-1}$.
\item[(4)] $\lambda(r)= m_{X^{\ast}}(r)\widetilde{\beta}^{-1}(r)$.
\end{itemize}
\begin{theorem}[Thm.\ref{mainthm1}]\label{thetagamm12}
For $r=\begin{pmatrix} a & b \\ c& d\end{pmatrix}\in \Gamma_m(1, 2)$, we have:
\begin{itemize}
\item[(1)] $\theta_{1/2}(r z)=\lambda(r) \sqrt{\det(cz+d)}\theta_{1/2}( z)$.
\item[(2)] $\theta_{3/2}(r z)=\lambda(r) \sqrt{\det(cz+d)}(cz+d)\theta_{3/2}( z)$.
\end{itemize}
\end{theorem}
It is well-known that theta functions have a long and rich history. For examples, one can see \cite{Be}, \cite{Fr1}, \cite{Fr2}, \cite{Ig}, \cite{Kl}, \cite{Mu}, \cite{Na}, \cite{Ri1,Ri2}, \cite{Sc}, \cite{Ta2}, \cite{Vi}, \cite{Ya1}.  Here, we approach them in a more  personal way, using the theory of Weil representations.  The  Weil representation $\pi_{\psi_0}$ has two irreducible components, and the above theta functions correspond to each component in a one-to-one manner. Here, we derive Theorem \ref{thetagamm12}(1) using  a single irreducible Weil representation of the Metaplectic group, without the assistance of  Jacobi forms and   Jacobi groups. To obtain the above theorem,  we mainly  examined and compared three models associated with the Weil representation: the Schr\"odinger model, the Lattice model, and the Fock model. The methodology  closely follows that of the  case $m=1$, as  done in  Lion-Vergne's book \cite{LiVe} or \cite{Wa}.

In the above theorem, there are some arithmetic constants $\lambda(r)$. These constants have been  extensively studied  by numerous researchers, including Andrianov-Maloletkin \cite{AnMa},  Friedberg \cite{Fr},  Stark \cite{St}, Styer \cite{Sty2}, Richter \cite{Ri1, Ri2}, and others. By observation,  it is evident that $\lambda(r)$ is contingent upon the selection of  $\sqrt{\det(cz+d)}$. In our case, $\epsilon(g; z,z_0) $  is derived from the works of   Satake \cite{Sa1} and  Takase \cite{Ta1}.  $m_{X^{\ast}}(r)$ is a different function  between the $8$-degree cocycle and $2$-degree cocycle associated to the Metaplectic group. This function is attributed to the contributions of Rao  \cite{Ra}, Kudla \cite{Ku}, Perrin \cite{Pe},  Lion-Vergne \cite{LiVe}. The function $\widetilde{\beta}(r)$  is a trivialization  map for the $8$-degree cocycle over $\Gamma_m(1,2)$. It appears in   Lion-Vergne's book \cite{LiVe} and it can be calculated by using   the symplectic Guass sum as introduced by Styer in \cite{Sty1} and \cite{Sty2}.
\subsubsection{} We  further extend the above result to the integral symplectic group $\Sp_{2m}(\Z)$ in some sense.  Let  $\widetilde{c}_{X^{\ast}}$ and $\overline{c}_{X^{\ast}}$ represent   the  Perrin-Rao cocycles over the symplectic group $\Sp(W)$ or $\Sp_{2m}(\R)$, corresponding to degrees 8 and 2, respectively. The associated central extension groups corresponding to $\widetilde{c}_{X^{\ast}}$ and $\overline{c}_{X^{\ast}}$ are denoted by $\widetilde{\Sp}_{2m}(\R)$ and $\overline{\Sp}_{2m}(\R)$, respectively.  The Metaplectic group associated with Satake-Takase's  $2$-cocycle $\beta'_{z_0}$ of center $T$ is represented by  $\Mp^{z_0}_{2m}(\R)$. Let   $\widetilde{\Sp}^{z_0}_{2m}(\R)$ be  the  subgroup of $\Mp^{z_0}_{2m}(\R)$ defined in Section \ref{autoI}.

For our purpose, we explicitly determine the coset $\Gamma_m(1,2)\setminus \Sp_{2m}(\Z)$  by following  R. H. Dye's two papers \cite{Dy1, Dy2} and referring to the works of  D. S. Kim, I.-S. Lee, and Y. H. Park\cite{KiLe, KiPa,Ki}.

 Following  Section \ref{setminus},  let $\iota=(\iota_1, \cdots, \iota_m): \SL_2(\R) \times  \cdots \times \SL_2(\R) \longrightarrow \Sp_{2m}(\R)$. Define $\mathcal{M}'_{2}=\{ \begin{pmatrix} 1& 0\\ 0& 1\end{pmatrix},\begin{pmatrix} 1& 1\\ 0& 1\end{pmatrix},\begin{pmatrix} 1& 0\\-1& 1\end{pmatrix}\}\subseteq \SL_2(\R)$ and $\mathcal{N}'_{2}=\{ \begin{pmatrix} 0& 1\\-1&0\end{pmatrix}\}\subseteq \SL_2(\R)$.   Set  $\iota_i(\begin{pmatrix} 1& 0\\ 0& 1\end{pmatrix})=1_{2m}$, $u_{i}= \iota_i(\begin{pmatrix} 1& 1\\ 0& 1\end{pmatrix})$, $u^-_i=\iota_i(\begin{pmatrix} 1& 0\\-1& 1\end{pmatrix})$, $\omega_{i}= \iota_i(\begin{pmatrix} 0& 1\\-1&0\end{pmatrix})$. Let $V=\sum_{i=1}^m \F_2 e_i + \F_2 e_i^{\ast} $ and  $Q_0(v)=\sum_{i=1}^m x_ix_i^{\ast}$ for $v=\sum_{i=1}^m (x_i e_i+ x^{\ast}_i e_i^{\ast})\in V$. Each $q\in Q_0^{-1}(0)$ corresponds to a matrix $M'_q$.   Let $\mathcal{M}'_{2m}$ be the set of all such $M'_q$.
\begin{lemma}[Lem.\ref{represnetcoset}]
The coset $\Gamma_{m}(1,2) \setminus \Sp_{2m}(\Z)$ can be represented by $\mathcal{M}'_{2m}$.
\end{lemma}
To decompose each $M_q'$, we modify it by another matrix $M_q$ as given  in Section \ref{modification11}.  Let $\mathcal{S}^q_0=\{ i_1, \cdots, i_r\}$ and $\mathcal{S}^q_1=\{ (j_1, k_1), \cdots, (j_s, k_s)\}$ be the sets defined in Section  \ref{modification11}. Each  $M_q$ can then  be decomposed as  
$$M_q= (\prod_{i\in \mathcal{S}^q_0} M_q^{i}) \cdot (\prod_{ (j,k)\in\mathcal{S}^q_1}M_q^{(j,k)}),$$  where $M_q^{i}\in \iota_i(\mathcal{M}'_{2})$ and $M_q^{(j,k)}=\iota_{(j,k)}(u_{1111}^1)\iota_{(j,k)}(u_{1111}^2)$.  For each $M_q$,  define:
\begin{itemize}
\item $m_q=(k_1, \cdots, k_m)$, for $k_i=\left\{ \begin{array}{ll} 1 & \textrm{ if }  i\in \mathcal{S}^q_0, M_q^i= u_i,\\
-1& \textrm{ if } i=j \textrm{ or } k, (j,k)\in \mathcal{S}^q_1, \\  0 &  \textrm{ if } i\in \mathcal{S}^q_0,   M_q^{i}=1_{2m}  \textrm{ or }u_i^- .\end{array}\right.$
\item $\epsilon_{q}=(a_1, \cdots, a_m)$, for $a_i=\left\{ \begin{array}{ll} 1 & \textrm{ if } i\in \mathcal{S}^q_0, M_q^i=u_i^-,\\ -1& \textrm{ if } i=j \textrm{ or } k, (j,k)\in \mathcal{S}^q_1, \\ 0 &  \textrm{ if } i\in \mathcal{S}^q_0, M_q^{i}=1_{2m}, \textrm{ or }u_i. \end{array}\right.$
\item $m_{X^{\ast}}(q)= m_{X^{\ast}}(M_q)=\big(\prod_{i\in \mathcal{S}^q_0} m_{X^{\ast}}(M_q^{i}) \big)\cdot \big(\prod_{ (j,k)\in\mathcal{S}^q_1}m_{X^{\ast}}(M_q^{(j,k)})\big)$.
\end{itemize} For $\widetilde{M_q}=(M_q, t_{M_q})\in  \widetilde{\Sp}^{z_0}_{2m}(\R)$, let $\widetilde{M}_q=(\prod_{i\in \mathcal{S}^q_0} \widetilde{M}_q^{i}) \cdot (\prod_{ (j,k)\in\mathcal{S}^q_1}\widetilde{M}_q^{(j,k)})$.  For $z\in \mathbb{H}_m$,  we define:
\begin{align*}
\theta_{1/2}^{\widetilde{M_q}}(z)
&=m_{X^{\ast}}(q)^{-1}t_{M_q}\sum_{n=(n_1,\cdots,n_m)\in \Z^m} (-1)^{m_q \cdot n^T}e^{i  \pi (n+\tfrac{1}{2}\epsilon_{q}) z (n+\tfrac{1}{2}\epsilon_{q})^T}.\\
& \\
\theta_{3/2}^{\widetilde{M_q}}(z)&=m_{X^{\ast}}(q)^{-1}t_{M_q}\sum_{n=(n_1,\cdots,n_m)\in \Z^m} (-1)^{m_q \cdot n^T}(n+\tfrac{1}{2}\epsilon_{q})^Te^{i  \pi (n+\tfrac{1}{2}\epsilon_{q}) z (n+\tfrac{1}{2}\epsilon_{q})^T}.
\end{align*}
 Let $M_{q_1}, \cdots, M_{q_{n}}$ be  the complete sets of $\mathcal{M}_{2m}$.  We now define:
$$\Theta_{1/2}:  \mathbb{H}_m \longrightarrow M_{1n}(\C); z\longmapsto (\theta_{1/2}^{\widetilde{M_{q_1}}}(z), \cdots, \theta_{1/2}^{\widetilde{M_{q_n}}}(z))$$
$$\Theta_{3/2}:  \mathbb{H}_m \longrightarrow M_{mn}(\C); z\longmapsto (\theta_{3/2}^{\widetilde{M_{q_1}}}(z), \cdots, \theta_{3/2}^{\widetilde{M_{q_n}}}(z)).$$
Let $\overline{\Sp}_{2m}(\Z)$, $\overline{\Gamma}_m(2)$ and $\overline{\Gamma}_{m}(1,2)$ denote the inverse images of  $\Sp_{2m}(\Z)$,  $\Gamma_m(2)$ and $\Gamma_m(1,2)$ in $\overline{\Sp}_{2m}(\R)$, respectively. Define:
\begin{itemize}
\item[(1)] $\overline{\lambda}(r,\epsilon)=\lambda(r) \epsilon$, for $(r,\epsilon)\in \overline{\Gamma}_m(1,2)$.
\item[(2)] $\overline{\gamma}= \Ind_{\overline{\Gamma}_m(1,2)}^{\overline{\Sp}_{2m}(\Z)} \overline{\lambda}^{-1} $.
\end{itemize}

\begin{theorem}[Thm.\ref{mainthm112}]
For $r=\begin{pmatrix}a& b\\ c&d\end{pmatrix}\in \Sp_{2m}(\Z)$, $\epsilon\in \mu_2$ and $\overline{r}=(r,\epsilon)\in \overline{\Sp}_{2m}(\Z)$,  we have:
\begin{itemize}
\item[(1)] $\Theta_{1/2}(\overline{r}z)=\epsilon\sqrt{\det(cz+d)}\Theta_{1/2}(z)\overline{\gamma}(\overline{r}^{-1})$.
\item[(2)] $\Theta_{3/2}(\overline{r}z)=\epsilon \sqrt{\det(cz+d)}(cz+d)\Theta_{3/2}(z)\overline{\gamma}(\overline{r}^{-1})$.
\end{itemize}
\end{theorem}
By modifying the cocycle, we can  reduce $\overline{\gamma}$ to be a representation of a finite group related to Igusa's quotient group $\tfrac{\Sp_{2m}(\Z)}{\Gamma_m(4,8)}$.(cf. Sections \ref{finitI}, \ref{finitII})
\subsection{Questions}
Finally, let us list some questions,  whose correctness and difficulty are uncertain. 
\begin{itemize}
\item[(1)] How can the results be extended to  $\GSp_{2m}$?  Interpreting the Fock space in a nice  form in this case appears challenging initially.
\item[(2)] How can the results be generalized for any $z\in \mathbb{H}_m$ rather than just  $z=z_0$? This ties to Siegel paramodular forms and the works of Johnson-Leung-Roberts \cite{JLRo1,JLRo2} and Poor-Yuen \cite{PoYu}.
\item[(3)] Let $V$ be a $2m$-dimensional vector space over $\F_2$. Then $(\pi, \Sp_{2m}(\F_2), \C[V])$ is a representation via permutation action on $V$. According to N. F. J. Inglis \cite[Theorem 1]{In} or  Guest-Previtali-Spiga \cite[Theorem 1.1]{GuPrSp},  $$\pi=\pi^+\oplus\pi^-, \quad\quad  \C[V]=\C[\Oa_{2m}^+(\F_2)\setminus \Sp_{2m}(\F_2)]\oplus  \C[\Oa^-_{2m}(\F_2)\setminus \Sp_{2m}(\F_2)].  $$
     In our case, the finite-dimensional representation $\overline{\gamma}$ is weekly associated  with $\pi^+$. The question is, in which ``minimal representation" would the corresponding modular form be linked to $\pi^-$. 
\end{itemize}
\section{The Metaplectic group and the  relevant groups}
\subsection{Symplectic Group over $\F_2$} Let $\F_2$ denote the field of $2$ elements. Denote by $\Sp_{2m}(\F_2)$  the symplectic group (for details, see \cite{Ki}). 
Let $P_{X^{\ast}}(\F_2)$, $N_{X^{\ast}}(\F_2)$, $M_{X^{\ast}}(\F_2)$,  $N_{X^{\ast}}^-(\F_2)$,   $B(\F_2)$, $N_0(\F_2)$, $N_0^-(\F_2)$, $N(\F_2)$, $N^-(\F_2)$ be the corresponding subgroups of $\Sp_{2m}(\F_2)$. For $\omega\in \mathfrak{W}$, let $N_{\omega}^-(\F_2)=[\omega N^-(\F_2) \omega^{-1}]\cap N(\F_2)$, $N_{\omega}^{--}(\F_2)=[\omega^{-1} N^-(\F_2) \omega] \cap N(\F_2)$.
\begin{theorem}[Bruhat decomposition]\label{Buha}
\begin{itemize}
\item $\Sp_{2m}(\F_2)=\sqcup_{\omega\in\mathfrak{W}_{P_{X^{\ast}}}} P_{X^{\ast}}(\F_2)\omega P_{X^{\ast}}(\F_2)$.
\item $\Sp_{2m}(\F_2)=\sqcup_{\omega\in \mathfrak{W}} B(\F_2)\omega B(\F_2)=\sqcup_{\omega\in \mathfrak{W}} N(\F_2) \omega B(\F_2)=\sqcup_{\omega\in \mathfrak{W}}N_{\omega}^-(\F_2) \omega B(\F_2)=\sqcup_{\omega\in \mathfrak{W}} B(\F_2)\omega N_{\omega}^{--}(\F_2)$.
\end{itemize}
\end{theorem}
\begin{proof}
See \cite[Sections 2.5--2.9]{Ca}, or \cite{BoTi}, or \cite[p.60]{DiMi}. 
\end{proof}
\subsection{Subgroups of $\Sp_{2m}(\Z)$}
Let  $P_{X^{\ast}}(\Z)$, $N_{X^{\ast}}(\Z)$, $M_{X^{\ast}}(\Z)$,  $N_{X^{\ast}}^-(\Z)$,   $B(\Z)$, $N_0(\Z)$, $N_0^-(\Z)$, $N(\Z)$, $N^-(\Z)$ denote the intersection of their corresponding groups in $\R$ with $\Sp_{2m}(\Z)$.  Following \cite{Ig}, \cite{Mu}, we let 
 \begin{itemize}
 \item $\Gamma_m(2)=\{ g\in \Sp_{2m}(\Z) \mid g\equiv 1_{2m}(\bmod 2)\}$.
 \item $\Gamma_m(1, 2)=\{ g=\begin{pmatrix} a & b \\ c& d\end{pmatrix} \in \Sp_{2m}(\Z) \mid \textrm{ diagronal } a c^T \textrm{ even and } \textrm{ diagronal } b d^T \textrm{ even}\}$.
 \end{itemize}
By \cite[Lemma 25]{Ig},  there exists an exact sequence: 
\begin{equation}\label{gamma}
1\longrightarrow \Gamma_m(2) \longrightarrow \Sp_{2m}(\Z) \stackrel{\kappa_2}{\longrightarrow} \Sp_{2m}(\F_2) \longrightarrow 1.
\end{equation}
Let $I$ denote the inverse image of $B(\F_2)$ in $ \Sp_{2m}(\Z) $. In \cite{Be}, the above group $\Gamma_m(1, 2)$ is written by $\Sp^q(2m,2)$. By \cite[Pro. A 4]{Mu} or \cite{Be}, $\Gamma_m(1,2)$ can be generated by $\omega$, $h(a)$ and $u(b)$, where:  
 \begin{itemize}
 \item $\omega=\begin{pmatrix} 0 & -1_m\\ 1_m & 0\end{pmatrix} \in \Sp_{2m}(\Z)$.
 \item $h(a)=\begin{pmatrix} a & 0\\ 0 &(a^{T})^{-1}\end{pmatrix}$, with $a, (a^{T})^{-1}\in \GL_m(\Z)$.   
 \item $u(b)=\begin{pmatrix} 1_m & b\\ 0 & 1_m\end{pmatrix}$, where $b=b^T$ and  the diagonal entries of $b$ are even integers.
 \end{itemize}
According to \cite[p.207]{Be} or  \cite[Appendix to \S 5]{Mu}, $\Gamma_m(1,2)/\Gamma_m(2)$ is isomorphic to $O_{2m}^+(\F_2)$. 
 \begin{lemma}
 \begin{itemize}
 \item[(1)] $[\Sp_{2m}(\Z):\Gamma_m(2)]=|\Sp_{2m}(\F_2)|= 2^{m^2}\prod_{i=1}^m(2^{2i}-1)$.
 \item[(2)] $[\Gamma_m(1,2):\Gamma_m(2)]=|O_{2m}^+(\F_2)|=2^{m^2-m+1}(2^{m}-1)\prod_{i=1}^{m-1}(2^{2i}-1)$.
 \item[(3)] $[\Sp_{2m}(\Z):\Gamma_m(1,2)]=(2^{m}+1)2^{m-1}$.
 \end{itemize}
 \end{lemma}
 \begin{proof}
 Part (1) follows from \cite[Theorem 3.2]{Ki}. Part (2) follows from \cite[p.232]{Be} and \cite[Theorem 3.4]{KiLe}. Part (3) is a consequence of (1)(2).
 \end{proof}

 \subsection{Iwahori decomposition and application}
 Let $S_{i}=\{ 1, \cdots, i\}$ and $S_0=\emptyset$. Under the basis $\{e_1, \cdots, e_m; e_1^{\ast}, \cdots, e_m^{\ast}\}$, $\omega_{S_i}$ corresponds to the matrix 
 $$\omega_{S_i}=\begin{pmatrix}
0&0&-1_{i}&0\\
0&1_{m-i}&0&0\\
1_i&0&0&0\\
0&0&0&1_{m-i}
\end{pmatrix}.$$
Then $\omega_{S_i}\in \Gamma_m(1,2)$. Consequently, all $\omega_{S}\in \Gamma_m(1,2)$.  From the above Bruhat decomposition for $\Sp_{2m}(\F_2)$, we have:
\begin{theorem}[Bruhat, Iwahori]\label{BrIw}
\begin{itemize}
\item[(1)] $\Sp_{2m}(\Z)=\sqcup_{\omega\in\mathfrak{W}_{P_{X^{\ast}}}} \Gamma_m(2) P_{X^{\ast}}(\Z)\omega P_{X^{\ast}}(\Z)$.
\item[(2)] $\Sp_{2m}(\Z)=\sqcup_{\omega\in \mathfrak{W}} I\omega I$.
 \end{itemize}
\end{theorem}
Following \cite{KiPa}, let us  refine the first decomposition of the above theorem.  Let: 
\begin{itemize}
\item $A_{X^{\ast},i}(\F_2)=\{ p\in P_{X^{\ast}}(\F_2) \mid \omega_{S_i}p\omega_{S_i}^{-1} \in  P_{X^{\ast}}(\F_2)\}$, $A_{X^{\ast},i}(\Z)=\{ p\in P_{X^{\ast}}(\Z) \mid \omega_{S_i}p\omega_{S_i}^{-1} \in  P_{X^{\ast}}(\Z)\}$. 
\item $A'_{X^{\ast},i}(\F_2)=A_{X^{\ast},i}(\F_2) \cap O_{2m}^+(\F_2)$, $A'_{X^{\ast},i}(\Z)=A_{X^{\ast},i}(\Z) \cap \Gamma_m(1,2)$. 
\item   $P_{X^{\ast}}'(\F_2)=P_{X^{\ast}}(\F_2) \cap O_{2m}^+(\F_2)$, $P_{X^{\ast}}'(\Z)=P_{X^{\ast}}(\Z)\cap \Gamma_m(1,2)$.
\end{itemize}
Then by \cite[p.339]{KiPa}, 
 $$ A_{X^{\ast},i}(\F_2)=\left\{ g=h(a)u(b) \,\middle\vert\,
\begin{aligned} a=\begin{pmatrix} a_{11} & 0 \\a_{21} & a_{22} \end{pmatrix} \in \GL_m(\F_2), a_{11}\in \GL_i(\F_2), a_{22}\in \GL_{m-i}(\F_2),  a_{21}\in M_{(m-i)\times i }(\F_2)\\
  b=\begin{pmatrix}0 & b_{12} \\b_{12}^T & b_{22} \end{pmatrix} \in M_{m}(\F_2), b_{12}\in M_{i\times (m-i)}(\F_2), b_{22}=b_{22}^T\in M_{m-i}(\F_2)
  \end{aligned}
\right\}.$$
 $$ A'_{X^{\ast},i}(\F_2)=\left\{ g=h(a)u(b)\,\middle\vert\, h(a), u(b)\in A_{X^{\ast},i}(\F_2),  \textrm{ all diagonal entries of } b \textrm{ zero } \right\}.$$
 $$ P_{X^{\ast}}'(\F_2)=\left\{ g=h(a)u(b)\,\middle\vert\,   \textrm{ all diagonal entries of } b \textrm{ zero } \right\}.$$
\begin{lemma}\label{gamma12}
\begin{itemize}
\item[(1)] $\Sp_{2m}(\Z)=\sqcup_{\omega_{S_i}\in\mathfrak{W}_{P_{X^{\ast}}}} \Gamma_m(2) P_{X^{\ast}}(\Z)\omega_{S_i} (A_{X^{\ast},i}(\Z)\setminus  P_{X^{\ast}}(\Z))$.
\item[(2)] $\Gamma_m(1,2)=\sqcup_{\omega_{S_i}\in\mathfrak{W}_{P_{X^{\ast}}}} \Gamma_m(2) P_{X^{\ast}}'(\Z)\omega_{S_i} (A'_{X^{\ast},i}(\Z)\setminus  P_{X^{\ast}}'(\Z))$.
\end{itemize}
\end{lemma}
\begin{proof}
1) By \cite[(3.4)]{Ki}, $\Sp_{2m}(\F_2)=\sqcup_{\omega_{S_i}\in\mathfrak{W}_{P_{X^{\ast}}}} P_{X^{\ast}}(\F_2) \omega_{S_i} (A_{X^{\ast},i}(\F_2)\setminus  P_{X^{\ast}}(\F_2))$. Note that the maps  $\kappa_2:P_{X^{\ast}}(\Z) \longrightarrow   P_{X^{\ast}}(\F_2)$ and  $\kappa_2: A_{X^{\ast},i}(\Z) \longrightarrow A_{X^{\ast},i}(\F_2)$ both are surjective. Hence, there exists a bijective map:
$$ A_{X^{\ast},i}(\Z)  [ P_{X^{\ast}}(\Z) \cap\Gamma_m(2)] \setminus  P_{X^{\ast}}(\Z) \longrightarrow  A_{X^{\ast},i}(\F_2) \setminus  P_{X^{\ast}}(\F_2).$$
Consequently, $$ \kappa_2: P_{X^{\ast}}(\Z)\omega_{S_i} (A_{X^{\ast},i}(\Z)\setminus  P_{X^{\ast}}(\Z))  \longrightarrow P_{X^{\ast}}(\F_2) \omega_{S_i} (A_{X^{\ast},i}(\F_2)\setminus  P_{X^{\ast}}(\F_2))$$
is surjective. For an element $g\in \Sp_{2m}(\Z)$, assume that $\kappa_2(g)$ belongs to $P_{X^{\ast}}(\F_2) \omega_{S_i} (A_{X^{\ast},i}(\F_2)\setminus  P_{X^{\ast}}(\F_2))$. Then there exists an element $g_i\in P_{X^{\ast}}(\Z)\omega_{S_i} (A_{X^{\ast},i}(\Z)\setminus  P_{X^{\ast}}(\Z)) $ such that $\kappa_2(g)=\kappa_2(g_i)$. This  implies that $g=\gamma g_i $ for some $\gamma\in \Gamma_m(2)$.\\
2) By \cite[(3.4)]{KiPa}, $\Oa^+_{2m}(\F_2)=\sqcup_{\omega_{S_i}\in\mathfrak{W}_{P_{X^{\ast}}}} P_{X^{\ast}}'(\F_2) \omega_{S_i} (A'_{X^{\ast},i}(\F_2)\setminus  P_{X^{\ast}}'(\F_2))$. Similarly, we have:
 $$P_{X^{\ast}}'(\Z) \longrightarrow   P_{X^{\ast}}'(\F_2), A'_{X^{\ast},i}(\Z) \longrightarrow A'_{X^{\ast},i}(\F_2), $$
 $$A'_{X^{\ast},i}(\Z)  [ P_{X^{\ast}}'(\Z) \cap\Gamma_m(2)] \setminus  P_{X^{\ast}}'(\Z) \longrightarrow  A'_{X^{\ast},i}(\F_2) \setminus  P_{X^{\ast}}'(\F_2),$$
  $$P_{X^{\ast}}'(\Z)\omega_{S_i} (A'_{X^{\ast},i}(\Z)\setminus  P_{X^{\ast}}'(\Z))  \longrightarrow P_{X^{\ast}}'(\F_2) \omega_{S_i} (A'_{X^{\ast},i}(\F_2)\setminus  P_{X^{\ast}}'(\F_2))$$ 
 all are surjective.  The remaining proof  is the same as above. 
\end{proof}
\subsection{The quotient $\Gamma_{m}(1,2) \setminus \Sp_{2m}(\Z)$}\label{setminus}
 Note that  there exists a bijective map from $ \Gamma_{m}(1,2) \setminus \Sp_{2m}(\Z) $ to $\Oa^+_{2m}(\F_2) \setminus \Sp_{2m}(\F_2)$.  We determine the subsequent set by adhering to the two papers by R. H. Dye (cf. \cite{Dy1, Dy2}). Let $V=\sum_{i=1}^m \F_2 e_i + \F_2 e_i^{\ast} $.  Following \cite{Dy1}, let $\mathcal{Q}$ denote the associated quadratic forms $Q$ on $V$ such that 
\[ \langle v, v'\rangle=Q(v+v')-Q(v)-Q(v').\]
In particular, let $Q_0$ be an element in $\mathcal{Q}$, defined as:
\[Q_0(v)=\sum_{i=1}^m x_ix_i^{\ast},\]
where $ v=\sum_{i=1}^m (x_i e_i+ x^{\ast}_i e_i^{\ast})$. Then,  $\Oa^+_{2m}(\F_2)$ is the isometry group of $Q_0$. According to \cite[Lemma 1]{Dy2}, we have:
\begin{itemize}
\item[1.] The other quadratic form $Q$ in  $\mathcal{Q}$ has the form: $Q(v)= Q_0(v)+\langle v, q\rangle^2$, for some $q\in V$. 
\item[2.] There exists an element $g\in \Sp_{2m}(\F_2)$ such that  $Q(v)=Q_0(vg)$ iff $t^2+t+Q_0(q)$ is  reducible in $\F_2$ iff $Q_0(q)=0$.
\end{itemize}
Let: 
\begin{itemize}
    \item[(1)] $ Q_0^{-1}(0) = \{ v \in V \mid Q_0(v) = 0 \}$.
    \item[(2)] $Q_0^{-1}(1) = \{ v\in V \mid Q_0(v) = 1 \}$.
\end{itemize}
\begin{lemma}
    $|Q_0^{-1}(0)| = 2^{2m-1} + 2^{m-1}$ and $|Q_0^{-1}(1)| = 2^{2m} - 2^{2m-1} - 2^{m-1}$.
\end{lemma}
\begin{proof}
We prove by induction on $m$. Let $a_m=|Q_0^{-1}(0)|$, $b_m=|Q_0^{-1}(1)|$. If $m=1$,  $a_m=3$, $b_m=1$.  Additionally, we have the following relations:  
\[\begin{cases}
a_m + b_m = 2^{2m}, \\
a_{m+1} = 3a_m + b_m, \\
b_{m+1} = a_m + 3b_m.
\end{cases}\]
By induction, the result follows.
\end{proof}
For any $q \in V$, let us define the transvection $T_q$ on $V$ as follows:
\[ v \longmapsto v + \langle v, q \rangle q. \]
Let $Q_0^{T_q}$ be defined such that $Q_0^{T_q}(v) = Q_0(v T_q)$, for $v \in V$. According to the proof of Theorem 1 in  \cite{Dy2}, we have:
\[ Q_0^{T_q}(v) = Q_0(v) + \langle v, q \rangle^2(1+Q_0(q)). \]
In particular, for $q\in Q_0^{-1}(0)$, we have:
\begin{align}
 Q_0^{T_q}(v) = Q_0(v) + \langle v, q \rangle^2. 
 \end{align}
For $q = \sum_{i=1}^m x_i e_i + x_i^{\ast} e_i^{\ast} \in V$, $T_q$ corresponds to the matrix
\[ M_q' =1_{2m}+\begin{pmatrix} x_q^{\ast T}\\-x_q^T\end{pmatrix}  \cdot \begin{pmatrix} x_q& x_q^{\ast}\end{pmatrix}, \]
where  $x_q=(x_1,\cdots, x_m)$, $x_q^\ast=(x_1^{\ast},\cdots, x_m^{\ast})$. Let $\mathcal{M}'_{2m}$(resp. $\mathcal{N}'_{2m}$) be the set of all such $M_q'$ as $q$ varies over the elements of $Q_0^{-1}(0) $ (resp. $Q_0^{-1}(1)$).
\begin{lemma}
    The coset $\Oa^+_{2m}(\F_2) \setminus \Sp_{2m}(\F_2)$ can be represented by $\mathcal{M}'_{2m}$.
\end{lemma}
\begin{proof}
For any $g\in \Sp_{2m}(\F_2)$, $Q_0^g\in \mathcal{Q}$. By the results from the previous steps (1 and 2),  there exists some $q\in Q_0^{-1}(0)$ such that 
$$Q_0^g(v)=Q_0(vg)=Q_0(v)+\langle v, q\rangle^2=Q_0^{T_q}(v).$$
This implies that $\Oa^+_{2m}(\F_2) g =\Oa^+_{2m}(\F_2)  T_q$.  For two  $q', q\in Q_0^{-1}(0)$, if $Q_0^{T_q}(v)=Q_0^{T_{q'}}(v)$, then $\langle v, q\rangle^2=\langle v, q'\rangle^2$, which implies $\langle v, q\rangle=\langle v, q'\rangle$ and thus $q=q'$.
\end{proof}
\begin{example}
If $m=1$, $\mathcal{M}'_{2}=\{ \begin{pmatrix} 1& 0\\ 0& 1\end{pmatrix},\begin{pmatrix} 1& 1\\ 0& 1\end{pmatrix},\begin{pmatrix} 1& 0\\-1& 1\end{pmatrix}\}$, $\mathcal{N}'_{2}=\{ \begin{pmatrix}0& 1\\-1& 0\end{pmatrix}\}$.
\end{example}
\begin{example}
In $\F_2$, $0=2$. If $m=2$,  
\[\mathcal{M}'_{4}=\Bigg\{\begin{pmatrix}
1 & 0 & 0 & 0 \\
0 & 1 & 0 & 0 \\
0 & 0 & 1 & 0 \\
0 & 0 & 0 & 1
\end{pmatrix},\begin{pmatrix}
1 & 0 & 0 & 0 \\
0 & 1 & 0 & 1 \\
0 & 0 & 1 & 0 \\
0 & 0 & 0 & 1
\end{pmatrix},\begin{pmatrix}
1 & 0 & 1 & 0 \\
0 & 1 & 0 & 0 \\
0 & 0 & 1 & 0 \\
0 & 0 & 0 & 1
\end{pmatrix},
\begin{pmatrix}
1 & 0 & 1 & 1 \\
0 & 1 & 1 & 1 \\
0 & 0 &1 & 0 \\
0 & 0 & 0 & 1
\end{pmatrix},\begin{pmatrix}
1 & 0 & 0 & 0 \\
0 & 1 & 0 & 0 \\
0 & 0 & 1 & 0 \\
0 & -1 & 0 & 1
\end{pmatrix},\]
\[\begin{pmatrix}
1 & 1 & 1 & 0 \\
0 & 1 & 0 & 0 \\
0 & 0 & 1 & 0 \\
0 & -1 & -1 & 1
\end{pmatrix},\begin{pmatrix}
1 & 0 & 0 & 0 \\
0 & 1 & 0 & 0 \\
-1 & 0 & 1 & 0 \\
0 & 0 & 0 & 1
\end{pmatrix},
\begin{pmatrix}
1 & 0 & 0 & 0 \\
1 & 1 & 0 & 1 \\
-1 & 0 & 1 & -1 \\
0 & 0 & 0 & 1
\end{pmatrix},\begin{pmatrix}
1 & 0 & 0 & 0 \\
0 & 1 & 0 & 0 \\
-1 & -1 & 1 & 0 \\
-1 & -1 & 0 & 1
\end{pmatrix},\begin{pmatrix}
2 & 1 & 1 & 1 \\
1 & 2 & 1 & 1 \\
-1 & -1 & 0 & -1 \\
-1 & -1 & -1 & 0
\end{pmatrix}\Bigg\}.\] \[\mathcal{N}'_{4}=\Bigg\{\begin{pmatrix}
1 & 0 & 0 & 0 \\
0 & 2 & 0 & 1 \\
0 & 0 & 1 & 0 \\
0 & -1 & 0 & 0
\end{pmatrix},\begin{pmatrix}
1 & 1& 1 & 1 \\
0 & 2 & 1 & 1 \\
0 & 0 & 1 & 0 \\
0 & -1 & -1 & 0
\end{pmatrix},\begin{pmatrix}
2 & 0 & 1 & 0 \\
0 & 1 & 0 & 0 \\
-1 & 0 & 0 & 0 \\
0 & 0 & 0 & 1
\end{pmatrix},\]
\[
\begin{pmatrix}
2 & 0 & 1 & 1 \\
1 & 1 & 1 & 1 \\
-1 & 0 & 0 &-1 \\
0 & 0 & 0 & 1
\end{pmatrix},\begin{pmatrix}
1 & 0 & 0 & 0 \\
1 & 2 & 0 & 1 \\
-1 & -1 & 1 & -1 \\
-1 & -1 & 0 & 0
\end{pmatrix},
\begin{pmatrix}
2 & 1 & 1 & 0 \\
0 & 1 & 0 & 0 \\
-1 & -1 & 0 & 0 \\
-1 & -1 & -1& 1
\end{pmatrix}\Bigg\}.\]
\end{example}
  We also consider these matrices as elements of  $\Sp_{2m}(\Z)$.
\begin{lemma}\label{represnetcoset}
The coset $\Gamma_{m}(1,2) \setminus \Sp_{2m}(\Z)$ can be represented by $\mathcal{M}'_{2m}$.
\end{lemma}
\begin{proof}
By using the canonical  map $\kappa:  \Sp_{2m}(\Z) \longrightarrow \Sp_{2m}(\F_2)$,  we derive the desired result.
\end{proof}
\begin{remark}
For the other discussion of  the coset $\Oa^+_{2m}(\F_2) \setminus \Sp_{2m}(\F_2)$, one can also see \cite{DaGe}, \cite{GuPrSp}, \cite{In}.
\end{remark}
\subsection{Some calculation} 
\begin{lemma}\label{Mproddcom}
\begin{itemize}
\item[(1)] $\begin{pmatrix}
2 & 1 & 1 & 1 \\
1 & 2 & 1 & 1 \\
-1 & -1 & 0 & -1 \\
-1 & -1 & -1 & 0
\end{pmatrix}=  \begin{pmatrix}
0 & -1 & 1 & 0 \\
-1 &0 & 0 & 1 \\
0 & 0 & 0 & -1 \\
0 & 0 & -1& 0
\end{pmatrix}\begin{pmatrix}
1 & 0 & -1 & 0 \\
0 & 1 & 0 & -1 \\
0 & 0 & 1 & 0 \\
0 & 0 & 0& 1
\end{pmatrix}\begin{pmatrix}
1 & 0 & 0 & 0 \\
0 & 1 & 0 & 0 \\
1 & 1 & 1& 0 \\
1 & 1 & 0& 1
\end{pmatrix}$.
\item[(2)] $\begin{pmatrix}
1 & 0 & 0 & 0 \\
0 & 1 & 0 & 0 \\
1 & 1 & 1& 0 \\
1 & 1 & 0& 1
\end{pmatrix}=\begin{pmatrix}
1 & -1 & 1 & 0 \\
0 & 1 & 0 & 0 \\
0 & 0 & 1& 0 \\
0 & 0 & 1& 1
\end{pmatrix}\begin{pmatrix}
0 & 0 & -1 & 0 \\
0 & 1 & 0 & 0 \\
1& 0 & 0& 0 \\
0& 0 & 0& 1
\end{pmatrix}\begin{pmatrix}
1 &0  & 1 & 0 \\
0 & 1 & 0 & 0 \\
0 & 0 & 1& 0 \\
0 & 0 & 0& 1
\end{pmatrix}\begin{pmatrix}
1 & 1 & 0 & 0 \\
0 & 1 & 0 & 0 \\
0 & 0 & 1& 0 \\
0 & 0 & -1& 1
\end{pmatrix}$.
\end{itemize}
\end{lemma}
\begin{itemize}\label{decomsmatr}
\item[$\ast$] In the above equality (1),  let us denote the first matrix as $u_{1111}$, the second as $u_{1111}^0$,   the third as $u_{1111}^1$ and the fourth as $u_{1111}^2$. Note that  the second matrix belongs to $\Gamma_2(1,2)$.
\item[$\ast\ast$] In the above equality (2),  let us denote the second  matrix as $p_1$,  the third as $\omega_{S}$, the fourth as $n_2$ and the fifth as $m_2$, where $S=\{1\}$.
\end{itemize}
\begin{lemma}\label{calculation4}
If $\dim W=4$, then $\widetilde{c}_{X^{\ast}}(u_{1111}^1u_{1111}^2, \omega)^{-1}\stackrel{\textcircled{1}}{=}\widetilde{c}_{X^{\ast}}(u_{1111}^2, \omega)^{-1}\stackrel{\textcircled{2}}{=}m_{X^{\ast}}(u_{1111}^2)\stackrel{\textcircled{3}}{=}m_{X^{\ast}}(u_{1111}^1u_{1111}^2)=e^{-\tfrac{\pi i}{4}}$.
\end{lemma}
\begin{proof}
Let $h_1=u_{1111}^1$, $h_2=u_{1111}^2$. \\
1) $\widetilde{c}_{X^{\ast}}(h_1h_2, \omega)=\widetilde{c}_{X^{\ast}}(h_1,h_2)\widetilde{c}_{X^{\ast}}(h_1h_2, \omega)=\widetilde{c}_{X^{\ast}}(h_1,h_2\omega)\widetilde{c}_{X^{\ast}}(h_2, \omega)=\widetilde{c}_{X^{\ast}}(h_2, \omega).$ So the equality $\textcircled{1}$ holds.\\
2)  By (\ref{chap2mx}), $m_{X^{\ast}}(h_2)=  \gamma(x(h_2), \psi^{\tfrac{1}{2}})^{-1} \gamma(\psi^{\tfrac{1}{2}})^{-j(h_2)}=\gamma(x(h), \psi^{\tfrac{1}{2}})^{-1} \gamma(\psi^{\tfrac{1}{2}})^{-j(h)}=m_{X^{\ast}}(h)$. So the equality $\textcircled{3}$ holds.\\
3) Recall  ($\ast\ast$) after Lemma \ref{Mproddcom}. Let $S'=\{2\}$, and $\omega_{S'}=\begin{pmatrix}
1& 0 &  0& 0 \\
0 & 0 & 0 & -1 \\
0& 0 & 1& 0 \\
0& 1 & 0& 0
\end{pmatrix}$. So $\omega=\omega_S\omega_{S'}$. Then:
  \[m_{X^{\ast}}(h_2)=\gamma(x(h_2), \psi^{\tfrac{1}{2}})^{-1} \gamma(\psi^{\tfrac{1}{2}})^{-j(h_2)}=\gamma(\psi^{\tfrac{1}{2}})^{-1}=e^{-\tfrac{\pi i}{4}};\]
  \[\widetilde{c}_{X^{\ast}}(h_2, \omega)=\widetilde{c}_{X^{\ast}}(p_1\omega_Sn_2m_2, \omega)=\widetilde{c}_{X^{\ast}}(\omega_Sn_2, \omega_S\omega_{S'})=\widetilde{c}_{X^{\ast}}(\omega_Sn_2, \omega_S)\]
  \[=\widetilde{c}_{X^{\ast}}(\begin{pmatrix}
0& -1  \\
 1& 0
\end{pmatrix}\begin{pmatrix}
1& 1  \\
 0&     1
\end{pmatrix}, \begin{pmatrix}
0& -1  \\
 1& 0
\end{pmatrix})=\gamma(\psi^{\tfrac{1}{2}})=e^{\tfrac{\pi i}{4}}\]
\end{proof}
\subsection{Modification of $M_q'$}\label{modification11}
For later use, let us consider the composition of the matrix $M_q'$  in terms of its indices $i$. Consider $q= \sum_{i=1}^m x_i e_i + x_i^{\ast} e_i^{\ast}\in Q_0^{-1}(0)$. Write:
\begin{itemize}
\item $q_i=x_i e_i + x_i^{\ast} e_i^{\ast}$.
\item $\mathcal{S}^q_0=\{ i \mid Q_0(q_i)=0\}=\{ i_1, \cdots, i_r\}$.
\item $\mathcal{S'}^q_1=\{ i \mid Q_0(q_i)=1\}$.
\end{itemize}
Note that $Q_0(q)=\sum_{i=1}^m Q_0(q_i)=0$, so the number of $\mathcal{S'}^q_1$ is even. We partition $\mathcal{S'}^q_1$
  into distinct pairs  $(j_1,k_1), \cdots, (j_s, k_s)$.  Fixing such a decomposition, we denote $\mathcal{S}^q_1$ as  the set of all such pairs. 
In relation to $T_q$, we define the map  $T_q'$ of composition of transvections on $V$ as follows:
\[T_q': v \longmapsto v + \sum_{i\in \mathcal{S}^q_0} \langle v, q_i \rangle q_i +  \sum_{(j,k) \in \mathcal{S}^q_1} \langle v, q_j+ q_k\rangle (q_j+ q_k).\]
Then:
\begin{align*}
&Q_0^{T'_q}(v) =Q_0(vT'_q)\\
&=  Q_0(v) + \sum_{i\in \mathcal{S}^q_0} \langle v, q_i \rangle^2 Q_0( q_i)+ \sum_{(j,k) \in \mathcal{S}^q_1} \langle v, q_j+ q_k\rangle^2 Q_0((q_j+ q_k))+\sum_{i\in \mathcal{S}^q_0} \langle v, q_i \rangle^2 +  \sum_{(j,k) \in \mathcal{S}^q_1} \langle v, q_j+ q_k\rangle^2\\
&= Q_0(v)+\sum_{i\in \mathcal{S}^q_0} \langle v, q_i \rangle^2 +  \sum_{(j,k) \in \mathcal{S}^q_1} \langle v, q_j+ q_k\rangle^2\\
&= Q_0(v) + \langle v, q \rangle^2=Q_0^{T_q}(v). 
\end{align*}
Hence $\Oa^+_{2m}(\F_2) T_q =\Oa^+_{2m}(\F_2) T_q'$.  Let $M_q''$  denote the corresponding  matrix.

 Let $V_i=\F_2e_i+\F_2e_i^{\ast}$ and $V_{(j,k)}=\F_2e_j+\F_2e_j^{\ast}+\F_2e_k+\F_2e_k^{\ast}$. Then there exists a group homomorphism:
$$\iota=(\iota_{i_1}, \cdots, \iota_{i_r}, \iota_{(j_1, k_1)}, \cdots, \iota_{(j_s, k_s)} ): \SL_2(V_{i_1}) \times  \cdots \times \SL_2(V_{i_r}) \times \Sp(V_{(j_1,k_1)}) \times  \cdots \times \Sp(V_{(j_s,k_s)})\longrightarrow \Sp(V).$$
\begin{itemize}
\item For each $i\in \mathcal{S}^q_0$, let  $x^{i}_q=(0,\cdots,0, x_i, 0,\cdots,0)$, $x_q^{\ast i}=(0,\cdots,0, x_i^{\ast}, 0,\cdots,0)$ and  $M_q^{'i}= 1_{2m}+\begin{pmatrix}(x^{\ast i}_q)^T\\-(x_q^{ i})^T\end{pmatrix}  \cdot \begin{pmatrix} x^{i}_q& x_q^{\ast i}\end{pmatrix}$.   we treat $\mathcal{M}'_{2}$,  $\mathcal{N}'_{2}$ as the subsets of $ \SL_2(V_i)$.  Then $M_q^{'i}=\iota_i(u)$, for $u\in \mathcal{M}'_{2}$.
\item For each $(j,k)\in \mathcal{S}^q_1$, let  $x^{(j,k)}_q=(0,\cdots,0, x_j, 0,\cdots,0,x_k, 0,\cdots,0,)$, $x_q^{\ast(j,k)}=(0,\cdots,0, x_j^{\ast}, 0,\cdots,0, x_k^{\ast}, 0,\cdots,0,)$ and $M_q^{'(j,k)}= 1_{2m}+\begin{pmatrix}(x_q^{\ast(j,k)})^T\\-(x^{(j,k)}_q)^T\end{pmatrix}  \cdot \begin{pmatrix}x^{(j,k)}_q& x_q^{\ast(j,k)}\end{pmatrix}=\iota_{(j,k)}(u_{1111})$.
\item  These matrices $M_q^{'i}$, $M_q^{'(j,k)}$ are pairwise commutative with each other.
\item $M_q''=(\prod_{i\in \mathcal{S}^q_0} M_q^{'i}) \cdot (\prod_{ (j,k)\in\mathcal{S}^q_1}M_q^{'(j,k)})$.
\end{itemize}
We also consider these matrices   as  elements of  $\Sp_{2m}(\Z)$. By Lemma \ref{Mproddcom}, we have:
 \begin{itemize}
 \item $u_{1111}=u_{1111}^0 u_{1111}^1u_{1111}^2$, and $u_{1111}^0\in \Gamma_2(1,2)$.
 \item $M_q^{'(j,k)}=\iota_{(j,k)}(u_{1111})=\iota_{(j,k)}(u_{1111}^0)\iota_{(j,k)}(u_{1111}^1)\iota_{(j,k)}(u_{1111}^2)$, with $\iota_{(j,k)}(u_{1111}^0)\in \Gamma_m(1,2)$.
 \item These matrices $\iota_{(j,k)}(u_{1111}^{l})$ also commute with the other  matrices $M_q^{'i}$ and  $M_q^{'(j',k')}$ of different indexes $i$ and $(j',k')$.
 \end{itemize}
 Finally, let us modify the matrices $M_q'$ and $M_q''$ as follows:
 \begin{definition}\label{Mqijk}
 \begin{itemize}
 \item $M_q^{i}=M_q^{'i}$, for $i\in\mathcal{S}^q_0$. 
 \item  $M_q^{(j,k)}=\iota_{(j,k)}(u_{1111}^1)\iota_{(j,k)}(u_{1111}^2)$, for $(j,k)\in \mathcal{S}^q_1$.
 \item $M_q=(\prod_{i\in \mathcal{S}^q_0} M_q^{i}) \cdot (\prod_{ (j,k)\in\mathcal{S}^q_1}M_q^{(j,k)})$.
 \item  Let $\mathcal{M}_{2m}$ be the set of all such $M_q$ as $q$ varies over the elements of $Q_0^{-1}(0) $.
 \end{itemize}
 \end{definition}
 It is clear that $\Gamma_m(1,2) M_q'=\Gamma_m(1,2) M_q''=\Gamma_m(1,2) M_q$.
\subsection{The group $\Gamma_m(d,2d)$}
For a matrix $g$, let $g(i,j)$ denote the element of its  $(i,j)$-position. Denote by $\mathfrak{sp}_{2m}(\Z/d\Z)$ the additive group of all $2m \times 2m$ matrices $g$ with entries in $\Z/d\Z$ such that $({g}^T)\omega +\omega g =0$.   Let: 
\begin{itemize}
\item $\Gamma_m(d)=\{ g\in \Sp_{2m}(\Z) \mid g=1_{2m}+ d g', g' \in M_{2m}(\Z) \}$. 
\item $\Gamma_m(d,2d)=\{ g\in \Gamma_m(d) \mid g=1_{2m}+ d g', g'(m+i,i)\equiv  g'(i,m+i)\equiv 0 (\bmod 2), i = 1, 2, \cdots, m \}$.
\end{itemize}

\begin{proposition}\label{gamma2q}
For an even $d$ and $m\geq 2$, we have:
\begin{itemize}
\item[(1)] $\Gamma_m(d,2d)$ is a normal subgroup of $\Sp_{2m}(\Z)$.
\item[(2)] $[\Gamma_m(d),\Gamma_m(d)]=\Gamma_m(d^2,2d^2)$.
\item[(3)]  There exists an exact sequence:
$ 0 \rightarrow\frac{\Gamma_m(d^2)}{\Gamma_m(d^2,2d^2)} \longrightarrow \frac{\Gamma_m(d)}{\Gamma_m(d^2,2d^2)}\longrightarrow \mathfrak{sp}_{2m}(\Z/d\Z)\longrightarrow 0.$
\item[(4)] The images of all elements $u_{ij}(d)$,  $u^-_{ij}(d)$, $v_{ij}(d)$  generate $\frac{\Gamma_m(d)}{\Gamma_m(d^2,2d^2)}$.
\item[(5)] The images of all  elements $u_{ii}(d^2)$,   $u^-_{ii}(d^2)$ generate $\frac{\Gamma_m(d^2)}{\Gamma_m(d^2,2d^2)}$.
\end{itemize}
\end{proposition}
\begin{proof}
See \cite[Section 10, Prop.10.1, (10.3)]{Sato}.
\end{proof}

\section{The Heisenberg representation and its models}\label{heisenbergrepm}

\subsection{Heisenberg representation I}
  A $\R$-subspace $Y$  is called a \emph{lagrangian} in $W$ if
$Y$ is a maximal isotropic subspace. Let $\psi_{Y}$ denote   a  character of $Y\times \R$ extended from $\psi$.  Let
$$ \pi_{Y,\psi}=\Ind_{Y\times \R}^{\Ha(W)}\psi_{Y}, \quad\quad  V_{Y, \psi}=\Ind_{Y \times \R}^{\Ha(W)} \C.$$
Then $\pi_{Y,\psi}$ defines a Heisenberg representation of  $\Ha(W)$ with central character $\psi$. 
\begin{example}[Schr\"odinger Model]\label{ex1}
If   $Y=X^{\ast}$, let $\mathcal{H}(X^{\ast})$ be  the set of measurable functions $f: \Ha(W) \longrightarrow \C$ such that
\begin{itemize}
\item[(i)] $f([x^{\ast},t]h)=\psi(t) f(h)$, all $x^{\ast}\in X^{\ast}$, $t\in \R$, and almost all $h\in \Ha(W)$;
\item[(ii)] $\int_{(X^{\ast}\times \R)\setminus \Ha(W)} ||f(\dot{h})||^2 d\dot{h}<+\infty$, where  $d\dot{h}$ is an $\Ha(W)$-right invariant measure on $(X^{\ast}\times \R)\setminus \Ha(W)$.
\end{itemize}
As  $(X^{\ast}\times \R)\setminus \Ha(W)$ is homeomorphic to $X^{\ast}$,   $\pi_{X^{\ast},\psi}$ can be realized on  $L^2(X)$ by the following formulas:
\begin{equation}\label{chap8representationsp211}
\pi_{X^{\ast},\psi}[(x,0)(x^{\ast},0)(0,k)]g(y)=\psi(k+\langle x+y,x^{\ast}\rangle) g(x+y),
\end{equation}
for  $g\in L^2(X)$, $x\in X$, $x^{\ast}\in X^{\ast}$, $k\in \R$.
\end{example}

A set $L$ is called  a \emph{lattice} in $W$ if  it is a free $\Z$-module of full rank in $W$.  For a  lattice $L$, we define:
$$L^{\ast}=\{ w\in W\mid \psi(\langle w, l\rangle)=1, \textrm{ for all } l\in L\}.$$
If $L^{\ast}=L$, we call $L$ a \emph{self-dual lattice} with respect to $\psi$. Now let $L$ be a self-dual lattice. Let $\Ha(L)=L\times \R$  denote  the correspondence Heisenberg  subgroup of $\Ha(W)$.  Let $\psi_{L}$ denote an extended  character of $\Ha(L)$  from $\psi$ of  $\R$.
Let us define the representation:
 $$(\pi_{L,\psi}=\Ind_{\Ha(L)}^{\Ha(W)} \psi_{L}, \quad V_{L,\psi}=\Ind_{\Ha(L)}^{\Ha(W)} \C).$$
Then $\pi_{L,\psi}$ also defines a Heisenberg representation of $\Ha(W)$ associated to  the center character $\psi$.
\begin{example}[Lattice model]\label{ex2}
If  $\psi=\psi_0$, $L=\Z e_1 \oplus  \cdots \oplus \Z e_m \oplus \Z e_1^{\ast} \oplus \cdots \oplus \Z e_m^{\ast}$, we let $\mathcal{H}(L)$ be the set of measurable functions $f: W \longrightarrow \C$ such that
\begin{itemize}
\item[(i)] $f(l+w)=\psi(-\tfrac{\langle x_1, x_l^{\ast}\rangle }{2}-\tfrac{\langle l, w\rangle}{2}) f(w)$, for all $l=x_l+x_{l}^{\ast}\in L=(L\cap X)\oplus (L\cap X^{\ast})$, almost all  $w\in W$;
\item[(ii)] $\int_{L\setminus W} ||f(w)||^2 dw<+\infty$.
\end{itemize}
Here, we choose a $W$-right invariant measure on $L\setminus W$. Then $\pi_{L,\psi}$ can be realized on $\mathcal{H}(L)$ by the following formulas:
$$\pi_{L,\psi}([w',t])f(w)=\psi(t+\tfrac{\langle w, w'\rangle}{2})f(w+w')$$
for $w,w'\in W$, $t\in \R$.
\end{example}
\subsection{Heisenberg representation II}
Let $V $ be a complex vector space of dimension $m$ with a  basis $\{f_1, \cdots, f_m\}$. Let $(,)_{V}$ be a  Hermitian form on $V$ given as follows:
$$(v, v')_V=\sum_{i=1}^m \overline{z}_iz'_i, \quad\quad v=\sum_{i=1}^m z_i f_i, v'=\sum_{i=1}^m z_i' f_i.$$
Let $\Ha(V)=V\oplus \R$ be the group of elements $(v, t)$,  with the
multiplication law given by
$$(v, t)(v',t')=(v+v', t+t'+ \tfrac{\Im(v,v')_V}{2}),   \quad\quad v, v'\in V, t,t'\in \R.$$
Then there exists a group isomorphism:
\begin{align}
\phi : &\quad\quad\quad\quad\quad \Ha(V) \longrightarrow \Ha(W)  ;\\
          &(\sum^m_{k=1} (a_k +ib_k ) f_k, t)\longmapsto  (\sum^m_{k=1} (a_k e_k +b_k e_k^{\ast}),t).
 \end{align}
 Let us recall the Fock model, which realizes the Heisenberg representation. Under the basis  $\{ f_1, \cdots, f_m\}$, we identity the space $V$ with $\C^m$. 
 \begin{example}[Fock model]\label{ex3}
If  $\psi=\psi_0$, we let $ \mathcal{H}_{F}$ denote the Fock space   of holomorphic   functions $f$ on $\C^m$ such that 
 \[\Vert f\Vert^2=\int_{\C^m} \mid f(w)\mid^2 e^{-\pi \Vert w\Vert^2} d(w)<+\infty,\]
 where $d(w)= 2^{-m} |dw \wedge d\overline{w}|$. The Heisenberg representation   of $\Ha(V)$ associated to $\psi$  can be realized on $\mathcal{H}_{F}$ by the following formulas:
$$\pi_{ F, \psi}([w',t])f(w)=\psi(t)e^{-\tfrac{\pi}{2}\Vert w'\Vert^2-\pi w\overline{w'}}f(w+w')$$
for $w,w'\in V$, $t\in \R$.
\end{example}
\begin{proof}
See \cite[p.43, (1.71)]{Fo}.
\end{proof}
In accordance with  the works of Satake-Takase(as cited in\cite{Sa1,Sa2,Ta1,Ta2}), a generalized Fock space can be associated with each element $z$ in $\mathbb{H}_m$.  Let us recall their results with the help of J.-H. Yang's two papers(\cite{Ya1, Ya2}) and Namikawa's book \cite{Na}.
\begin{itemize}\label{ccccc}
\item Identity $W$ with $\R^m \times \R^m$ and $V$ with $\C^m$ as above.  
\item Write the element of $W$ or $V$ in a row form. 
\item Write $z=P_z+iQ_z$ and $\widehat{z}=-\overline{z}$, for $z\in \mathbb{H}_m$.
\item Define $d_z(v)=\det(Q_z)^{-1}d(v)=d_{\widehat{z}}(v)$ on $\C^m$, where  $d(v)=2^{-m} |dv \wedge d\overline{v}|$.
\item Define $q_{z}(v)=e^{2\pi i\big(-\tfrac{1}{2}vz^{-1}v^{T}\big)}$ on $V$.
\item Define the kernel function \( \kappa(z', v'; z, v) \) as:
  $$\kappa(z',v'; z,v)=e^{\pi i(v'-\overline{v})\cdot (z'-\overline{z})^{-1} \cdot (v'-\overline{v})^T},$$ for $z,z'\in \mathbb{H}_m$ and $v,v'\in V$.
\item Define \( \kappa_z(v', v) \) as:
 $$\kappa_z(v',v)=\kappa(z,v'; z,v)=e^{\tfrac{\pi}{2} (v'-\overline{v}) \cdot  Q^{-1}_z \cdot (v'-\overline{v})^T},$$ for $v,v'\in V$.
\item Let $\mathcal{X}=\{ S\in \Mm(\C)\mid S^T=S, Re(S)>0\}$. Define the function:
$${\det}^{-1/2}:\mathcal{X} \to \C^{\times}; S \longmapsto {\det}^{-1/2}(S),$$
where  ${\det}^{-1/2}(S)=\int_{\R^m} e^{-\pi x\cdot S \cdot x^T} dx$.
\item Define $ \gamma(z',z)= \det^{-1/2}\big(\tfrac{z'-\overline{z}}{2i}\big) \cdot \det(\Im z')^{1/4}\cdot \det(\Im z)^{1/4}$.
\item Denote  $h_{-1}=\begin{pmatrix}
1_m& 0\\ 0& -1_m\end{pmatrix}\in \GSp_{2m}(\R)$. Denote $g^{h_{-1}}=(h_{-1})^{-1}gh_{-1}$.
\item Define $\epsilon(g;\widehat{z'},\widehat{z})= \gamma(g(\widehat{z'}),g(\widehat{z}))/\gamma(\widehat{z'},\widehat{z})$, for   $g\in \Sp_{2m}(\R)$. 
\item $\beta_z(g_1,g_2)=\epsilon(g_1^{h_{-1}};z, g_2^{h_{-1}}(z))$, for $g_i\in \Sp_{2m}(\R)$. 
\item Define an action of $\Sp_{2m}(\R)\ltimes \Ha(W)$ on $\mathbb{H}_m \times V$ by $[g,h]\cdot [z,v] \longrightarrow [g(z), (v+xz+x^{\ast})J(g,z)^{-1}]$, for $h=(x,x^{\ast};t)\in \Ha(W)$.
\item Define an automorphy factor $\eta$ on $ [\Sp_{2m}(\R) \ltimes \Ha(W)] \times [\mathbb{H}_m \times V] $ as follows:
\begin{itemize}
\item[(a)]   $\eta:  \Ha(W) \times [\mathbb{H}_m \times V] \longrightarrow \C^{\times};([x,x^{\ast};t], [z,v])\longmapsto e^{2\pi i[t+\frac{1}{2}(xz+x^{*})x^T+ vx^T]}$.
\item[(b)] $\eta: \Sp_{2m}(\R) \times [\mathbb{H}_m \times V] \longrightarrow \C^{\times};(g, [z,v]) \longmapsto e^{2\pi i[-\frac{1}{2} v  (cz+d)^{-1}cv^T]}$, for $g=\begin{pmatrix}
   a& b \\
   c & d
  \end{pmatrix}$.
\item[(c)] For $(g,h)=(1,h^{g^{-1}})(g,0)\in  \Sp_{2m}(\R) \ltimes \Ha(W)$, define 
$$\eta([g,h], [z,v])=\eta\Big(h^{g^{-1}}, [g(z),vJ(g,z)^{-1}]\Big) \eta(g,[z,v]).$$ 
\end{itemize}
\end{itemize}
\begin{lemma}
For $\widetilde{g}_1=[g_1,h_1],\widetilde{g}_2=[g_2,h_2]\in\Sp_{2m}(\R) \ltimes \Ha(W) $,  and $ \widetilde{z}=[z,v]\in \mathbb{H}_m \times V$, we have:
$$\eta(\widetilde{g}_1\widetilde{g}_2, \widetilde{z})=\eta(\widetilde{g}_1, \widetilde{g}_2(\widetilde{z}))\eta(\widetilde{g}_2, \widetilde{z}).$$
\end{lemma}
\begin{proof}
See  \cite[p.395]{Sa2} or \cite[p.121]{Ta1}.
\end{proof}
\begin{lemma}
For $\widetilde{g}=[g,h]\in\Sp_{2m}(\R) \ltimes \Ha(W) $,  and $ \widetilde{z}=[z,v],\widetilde{z'}=[z',v']\in \mathbb{H}_m \times V$, we have:
$$\kappa(\widetilde{g}(\widetilde{z}); \widetilde{g}(\widetilde{z'}))=\eta( \widetilde{g},  \widetilde{z})\kappa(\widetilde{z}; \widetilde{z'})\overline{\eta( \widetilde{g},  \widetilde{z'})}.$$
\end{lemma}
\begin{proof}
See \cite[p.396]{Sa2} or \cite[p.121]{Ta1}.
\end{proof}
\begin{example}[Generalized Fock model]\label{ex4}
If  $\psi=\psi_0$, we let $ \mathcal{H}_{F_{\widehat{z}}}$ denote the  Fock space   of holomorphic  functions $f$ on $\C^m$ such that 
 \[\Vert f\Vert^2= \int_{\C^m} \mid f(v)\mid^2\kappa_z(v,v)d_z(v)< +\infty.\]
 The Heisenberg representation   of $\Ha(W)$  can be realized on $\mathcal{H}_{F_{\widehat{z}}}$ by the following formulas:
$$\pi_{F_{\widehat{z}}, \psi}(h)f(v)=\eta((h^{h_{-1}})^{-1}, -\overline{z}, v) f(v+x\overline{z}+x^{\ast})=\psi(t)e^{-2\pi i[\frac{1}{2}(x\overline{z}+x^{*})x^T+vx^T]}f(v+x\overline{z}+x^{\ast}).$$
for $w=x+x^{\ast}\in W$,  $t\in \R$, $h=(w,t)\in \Ha(W)$. 
\end{example}
\begin{proof}
See \cite[pp.125-126]{Ta1}. 
\begin{align*}
&\Vert \pi_{F_{\widehat{z}}, \psi}(h)f\Vert^2\\
&=\int_{\C^m} \mid e^{-2\pi i\big[\frac{1}{2}(x\overline{z}+x^{\ast})x^T+ vx^T\big]}f(v+ x\overline{z}+x^{\ast})\mid^2\kappa_z(v,v)d_z(v)\\
&\stackrel{v=v_x+iv_y}{=}\int_{\C^m}  e^{-2\pi xQ_zx^T+4\pi v_y x^T}\mid f(v+ x\overline{z}+x^{\ast})\mid^2e^{-2\pi v_yQ_z^{-1}v_y^T} d_z(v)\\
&\stackrel{v'=v+ x\overline{z}+x^{\ast}}{=}\int_{\C^m}  e^{-2\pi xQ_zx^T+4\pi (v'_y+xQ_z) x^T}\mid f(v')\mid^2e^{-2\pi (v'_y+xQ_z)Q_z^{-1}(v'_y+xQ_z)^T} d_z(v')\\
&=\int_{\C^m}  e^{-2\pi xQ_zx^T+4\pi (v'_y+xQ_z) x^T}\mid f(v')\mid^2 \kappa_z(v',v') e^{-2\pi v'_yx^T} e^{-2\pi x(v'_y)^T}e^{-2\pi xQ_zx^T} d_z(v')\\
&=\int_{\C^m} \mid f(v')\mid^2 \kappa_z(v',v')d_z(v')<+\infty.
\end{align*}
\begin{align*}
&\pi_{F_{\widehat{z}}, \psi}(h')\pi_{F_{\widehat{z}}, \psi}(h)f(v)\\
&=\psi(t')e^{-2\pi i\big[\frac{1}{2}(x'\overline{z}+x^{*'})x^{'T}+ vx^{'T}\big]}\pi_{F_{\widehat{z}}, \psi}(h)f(v+ x'\overline{z}+x^{\ast'})\\
&=\psi(t')e^{-2\pi i\big[\frac{1}{2}(x'\overline{z}+x^{*'})x^{'T}+vx^{'T}\big]}\psi(t)e^{-2\pi i\big[\frac{1}{2}(x\overline{z}+x^{\ast})x^T+(v+ x'\overline{z}+x^{\ast'})x^T\big]}f(v+ x'\overline{z}+x^{\ast'}+ x\overline{z}+x^{\ast})\\
&=\psi(t'+t)e^{-2\pi i\big[\frac{1}{2}(x'\overline{z}+x^{*'})(x^{'T}+x^T)+v(x^{'T}+x^T)+\frac{1}{2}( x\overline{z}+x^{\ast})x^T+\frac{1}{2}(x'\overline{z}+x^{\ast'})x^T\big]}f(v+ x'\overline{z}+x^{\ast'}+ x\overline{z}+x^{\ast})\\
&=\psi(t'+t)e^{-2\pi i\big[\frac{1}{2}(x'\overline{z}+x^{*'}+x\overline{z}+x^{\ast})(x^{'T}+x^T)+ v(x^{'T}+x^T)+\frac{1}{2}(x'\overline{z}+x^{\ast'})x^T-\frac{1}{2}(x\overline{z}+x^{\ast})x^{'T}\big]}f(v+ x'\overline{z}+x^{\ast'}+ x\overline{z}+x^{\ast})\\
&=\psi(t'+t)e^{-2\pi i\big[\frac{1}{2}(x'\overline{z}+x^{*'}+x\overline{z}+x^{\ast})(x^{'T}+x^T)+v(x^{'T}+x^T)\big]}e^{-2\pi i\big[\frac{1}{2} x^{'\ast}x^T-\frac{1}{2}x^{\ast}x^{'T}\big]}f(v+ x'\overline{z}+x^{\ast'}+ x\overline{z}+x^{\ast})\\
&=\psi(t'+t-\tfrac{1}{2} x^{'\ast}x^T+\tfrac{1}{2}x^{\ast}x^{'T})e^{-2\pi i\big[\frac{1}{2}(x'\overline{z}+x^{*'}+x\overline{z}+x^{\ast})(x^{'T}+x^T)+v(x^{'T}+x^T)\big]}f(v+ x'\overline{z}+x^{\ast'}+ x\overline{z}+x^{\ast})\\
&=\pi_{F_{\widehat{z}}, \psi}(h'h)f(v).
\end{align*}
\end{proof}
\begin{remark}\label{uneqi}
Take $z=\widehat{z}=z_0=i1_{m}$. Following  \cite[p.397]{Sa1}, there exists a unitary equivalence $\mathcal{A}$ from $\mathcal{H}_{F_{z_0}}$ to $\mathcal{H}_F$, given by
$$\mathcal{A}: \mathcal{H}_{F_{z_0}} \longrightarrow \mathcal{H}_F; f \longmapsto \phi(v)=  e^{-\tfrac{\pi}{2}v v^T} f(-vi).$$ Furthermore, $\mathcal{A}$ defines an $\mathcal{H}(W)$-intertwining  operator. 
\end{remark}
\begin{proof}
1) $v=v_x+i v_y$, $w=-vi=v_y-v_x i$, $w_x=v_y,w_y=-v_x$.
\begin{align*}
 &\int_{\C^m} \mid \phi(v)\mid^2 e^{-\pi \Vert v\Vert^2} d(v)\\
 &=\int_{\C^m} \mid e^{-\tfrac{\pi}{2}v v^T} f(-vi)\mid^2 e^{-\pi \Vert v\Vert^2} d(v) \\
 &=\int_{\C^m}  e^{-2\pi v_x v_x^T} \mid f(-vi)\mid^2 d(v)\\
  &\stackrel{w=-vi }{=}\int_{\C^m}  e^{-2\pi w_y w_y^T} \mid f(w)\mid^2 d(w)<+\infty.
 \end{align*}
2) 
\begin{align*}
&\mathcal{A}(\pi_{ F_{z_0}, \psi}(h)f)(v)= e^{-\tfrac{\pi}{2}v v^T} [\pi_{F_{z_0}, \psi}(h)f](-vi)\\
&=e^{-\tfrac{\pi}{2}v v^T}\psi(t)e^{-2\pi i[\frac{1}{2}(-xi+x^{*})x^T-vix^T]}f(-vi-xi+x^{\ast});
\end{align*}
\begin{align*}
&\pi_{F, \psi}(h)\mathcal{A}(f)(v)=\psi(t)e^{-\tfrac{\pi}{2}(xx^T+x^{\ast}x^{\ast T})-\pi v(x-x^{\ast} i)^T}\mathcal{A}(f)(v+x+x^{\ast}i)\\
&=e^{-\tfrac{\pi}{2}(v+x+x^{\ast}i) (v+x+x^{\ast}i)^T}\psi(t)e^{-\tfrac{\pi}{2}(xx^T+x^{\ast}x^{\ast T})-\pi v(x-x^{\ast} i)^T}f(-vi-xi+x^{\ast})\\
&=e^{-\tfrac{\pi}{2}v v^T}e^{-2\pi vx^T}\psi(t)e^{-\pi(xx^T+ix^{\ast}x^T)}f(-vi-xi+x^{\ast})\\
&=\mathcal{A}(\pi_{F_{z_0}, \psi}(h)f)(v).
\end{align*}
\end{proof}
\begin{lemma}
For $\widehat{z},\widehat{z'}\in \mathbb{H}_m$, there exists the following unitary equivalence $U_{\widehat{z},\widehat{z'}}$ mapping from $\mathcal{H}_{F_{\widehat{z}}}$ to $\mathcal{H}_{F_{\widehat{z'}}}$, given by
$$U_{\widehat{z'},\widehat{z}}: \mathcal{H}_{F_{\widehat{z}}} \longrightarrow \mathcal{H}_{F_{\widehat{z'}}}; \phi \longmapsto f(v')= \gamma(\widehat{z'}, \widehat{z})\int_{\C^m} \kappa(\widehat{z'},v'; \widehat{z},v)^{-1}\phi(v) \kappa_{\widehat{z}}(v,v) d_{\widehat{z}}(v).$$ Furthermore, $U_{\widehat{z'},\widehat{z}}$ defines an $\Ha(W)$-intertwining  operator. 
\end{lemma}
\begin{proof}
See \cite[p.125]{Ta1} or \cite[Lemma 3]{Sa1}. Let us prove the second statement. Let $w=x+x^{\ast}$, $ h=(w,t)$, $ h'=(-x,x^{\ast};t)=(h^{h_{-1}})^{-1}\in \Ha(W)$.\\
0) 
\[e^{-2\pi i[\frac{1}{2}(x\overline{z}+x^{*})x^T+vx^T]}\psi(t)=\eta(h', [-\overline{z}, v]).\]
\[\kappa_{\widehat{z}}(v,v)=\kappa_{-\overline{z}}(v,v)=\kappa_z(v,v).\]
\[d_{\widehat{z}}(v)=d_{-\overline{z}}(v)=d_{z}(v).\]
\[h'\cdot [-\overline{z},v]=[-\overline{z}, v+x\overline{z}+x^{\ast}].\]
1)
\begin{align}
&U_{\widehat{z'},\widehat{z}}\circ \pi_{F_{\widehat{z}}, \psi}(h)\phi(v')\\
&=  \gamma(-\overline{z'}, -\overline{z})\int_{\C^m} \kappa(-\overline{z'},v'; -\overline{z},v)^{-1}[\pi_{F_{\widehat{z}}, \psi}(h)\phi](v) \kappa_{\widehat{z}}(v,v) d_{\widehat{z}}(v)\\
&= \gamma(-\overline{z'}, -\overline{z})\int_{\C^m} \kappa(-\overline{z'},v'; -\overline{z},v)^{-1}\psi(t)e^{-2\pi i[\frac{1}{2}(x\overline{z}+x^{*})x^T+vx^T]}\phi(v+x\overline{z}+x^{\ast}) \kappa_z(v,v) d_z(v)\\
&= \gamma(-\overline{z'}, -\overline{z})\int_{\C^m} \kappa(-\overline{z'},v'; -\overline{z},v)^{-1}\eta(h', [-\overline{z}, v])\phi(v+x\overline{z}+x^{\ast}) \kappa(-\overline{z}, v;-\overline{z}, v) d_z(v)\\
&= \gamma(-\overline{z'}, -\overline{z})\int_{\C^m} \kappa(h'[-\overline{z'},v']; h'[-\overline{z},v])^{-1}\eta(h', [-\overline{z}, v])\phi(v+x\overline{z}+x^{\ast}) \kappa(h'[-\overline{z}, v];h'[-\overline{z}, v]) d_z(v)\\
&\quad\quad\quad \quad \cdot\eta( h', [-\overline{z'},v'])\overline{\eta(h',  [-\overline{z},v])}\eta( h', [-\overline{z}, v])^{-1}\overline{\eta(h',  [-\overline{z}, v])}^{-1}\\
&\stackrel{w=v+x\overline{z}+x^{\ast}}{=} \gamma(-\overline{z'}, -\overline{z})\eta( h', [-\overline{z'},v'])\int_{\C^m} \kappa(h'[-\overline{z'},v']; [-\overline{z},w])^{-1}\phi(w) \kappa([-\overline{z}, w];[-\overline{z}, w]) d_z(w);
\end{align}
2) 
\begin{align}
&\pi_{F_{\widehat{z}'}, \psi}(h) \circ  U_{\widehat{z'},\widehat{z}}\phi(v')\\
&= \psi(t)e^{-2\pi i[\frac{1}{2}(x\overline{z'}+x^{*})x^T+v'x^T]}[ U_{\widehat{z'},\widehat{z}}\phi](v'+x\overline{z'}+x^{\ast})\\
&= \psi(t)e^{-2\pi i[\frac{1}{2}(x\overline{z'}+x^{*})x^T+v'x^T]}\gamma(-\overline{z'}, -\overline{z})\int_{\C^m} \kappa(-\overline{z'},v'+x\overline{z'}+x^{\ast}; -\overline{z},v)^{-1}\phi(v) \kappa_{\widehat{z}}(v,v) d_{\widehat{z}}(v)\\
 &=\gamma(-\overline{z'}, -\overline{z})\eta( h', [-\overline{z'},v'])\int_{\C^m} \kappa(h'[-\overline{z'},v']; [-\overline{z},v])^{-1}\phi(v) \kappa([-\overline{z}, v];[-\overline{z},v]) d_z(v).
\end{align}
\end{proof}
\begin{lemma}
For $z,z',z''\in \mathbb{H}_m$,  $U_{z,z'}\circ U_{z',z''}=  U_{z,z''}$.
\end{lemma}
\begin{proof}
See \cite[pp.399-400]{Sa1}.
\end{proof}
\section{The Weil representation and its models}\label{Weilmodel}
By Weil's famous work, the Heisenberg representation can give rise  to a projective representation of $\Sp(W)$ and then  to an actual representation of a $\C^{\times}$-covering group over $\Sp(W)$. As usual, we call the actual representation as a Weil representation. 
\subsection{Schr\"odinger model}\label{schod}
 For  explicit $2$-cocycles,  one can refer to the works  in   \cite{LiVe, Pe, Ra}. Following  the notations used in \cite{Wa}, we denote $\widetilde{c}_{X^{\ast}}$ and $\overline{c}_{X^{\ast}}$ as   Perrin-Rao cocycles over the symplectic group $\Sp(W)$, corresponding to degrees 8 and 2, respectively. By  \cite[p.360, Def.5.2]{Ra},  we can  define the normalizing constant by
\begin{align}\label{chap2mx}
m_{X^{\ast}}: \Sp(W) \longrightarrow \mu_8; g \longmapsto \gamma(x(g), \psi^{\tfrac{1}{2}})^{-1} \gamma(\psi^{\tfrac{1}{2}})^{-j(g)}
\end{align}
 for $g=p_1\omega_Sp_2$, $j(g)=|S|$. Then:
\begin{equation}\label{chap228inter}
\overline{c}_{X^{\ast}}(g_1, g_2)=m_{X^{\ast}}(g_1g_2)^{-1} m_{X^{\ast}}(g_1) m_{X^{\ast}}(g_2) \widetilde{c}_{ {X^{\ast}}}(g_1,g_2), \quad\quad g_i\in \Sp(W).
\end{equation}
Let $\widetilde{\Sp}(W)$ and $\overline{\Sp}(W)$ denote the associated central extension group associated to $\widetilde{c}_{X^{\ast}}$ and $\overline{c}_{X^{\ast}}$ respectively. Let $dx$ denote the self-dual Haar measure on $\R$ with respect to $\psi$.  Let $dx$,  $dx^{\ast}$ be the Haar measures on $X$, $X^{\ast}$, given by the product  measures from $dx$ of  $\R$.

The Weil representation $\pi_{X^{\ast}, \psi}$ of $\widetilde{\Sp}(W) \ltimes \Ha(W)$  can be realized on $L^2(X)$ by the following formulas:
\begin{equation}\label{chap2representationsp11}
\pi_{X^{\ast}, \psi}[(x,0)\cdot (x^{\ast},0)\cdot (0,k)]f(y)=\psi(k+\langle x+ y,x^{\ast}\rangle) f(x+y),
\end{equation}
\begin{equation}\label{chap2representationsp2}
\pi_{X^{\ast}, \psi}( u(b))f(y)=\psi(\tfrac{1}{2}\langle y,yb\rangle) f(y),
\end{equation}
\begin{equation}\label{chap2representationsp3}
\pi_{X^{\ast}, \psi}( h(a))f(y)=|\det(a)|^{1/2} f(ya),
\end{equation}
\begin{equation}\label{chap2representationsp4}
\pi_{X^{\ast}, \psi}([\omega,t])f(y)=\int_{X^{\ast}} t\psi(\langle y, y^{\ast}\rangle) f(y^{\ast}\omega^{-1}) dy^{\ast},
\end{equation}
where $f(y)\in S(X)$, $x,y \in X,  x^{\ast}, y^{\ast}\in   X^{\ast}$, $k\in \R$, $t\in \mu_8$, and $h(a)=\begin{pmatrix}
  a& 0\\
  0 &a^{\ast -1 } \end{pmatrix}$, $u(b)=\begin{pmatrix}
  1_m&b\\
  0 & 1_m
\end{pmatrix}$, $\omega\in \Sp( W)$. 
For the other elements  $\omega_S $, let $X_S=\Span_\R\{ e_i\mid i\in S\}$, $X_{S'}=\Span_\R\{ e_i\mid i\notin S\}$ and $X_S^{\ast}= \Span_\R\{ e^{\ast}_i\mid i\in S\}$, $X_{S'}^{\ast}= \Span_\R\{ e^{\ast}_i\mid i\notin S\}$. By \cite[p.388, Prop.2.1.4]{Pe} or \cite[p.351, Lem.3.2 (3.9)]{Ra}, for $x=x_s+x_{s'}\in X_S\oplus X_{S'}$,
 \begin{equation}
  \begin{split}
 \pi_{X^{\ast},\psi}(\omega_S)f(x_s+x_{s'})&=\int_{X_S^{\ast}} \psi(\langle x_s,  z_s^{\ast}\rangle ) f(x_{s'}\omega_S+z_s^{\ast}\omega^{-1}_S)dz_s^{\ast}.
    \end{split}
  \end{equation}
  \subsection{} Let $g=p_1\omega_S p_2\in \Sp_{2m}(\R)$, $p_i=\begin{pmatrix}
a_i& b_i\\
0& d_i
\end{pmatrix}$, $i=1,2$. Following \cite{Wa}, let us define:
 \begin{equation}\label{chap3alpp}
\nu(-1,g)=(\det (a_1a_2), -1)_{\R} \gamma(-1, \psi^{\tfrac{1}{2}})^{-|S|}.
\end{equation} 
By \cite{Wa}, we have:
\begin{align}\label{chap3equ}
\widetilde{c}_{X^{\ast}}(g_1^{\alpha}, g_2^{\alpha})=\widetilde{c}_{X^{\ast}}(g_1, g_2)\nu(\alpha,g_1)^{-1}\nu(\alpha,g_2)^{-1}\nu(\alpha,g_1g_2).
\end{align}
\begin{lemma}\label{chap3equ}
\begin{itemize}
\item[(1)] $ \gamma(x(g), \psi^{\tfrac{1}{2}})^2=(\det (a_1a_2), -1)_{\R}$.
\item[(2)] $ \gamma(\psi^{\tfrac{1}{2}})^2=\gamma(-1, \psi^{\tfrac{1}{2}})^{-1}$.
\item[(3)] $\nu(-1,g)=m_{X^{\ast}}(g)^{-2}$.
\end{itemize}
\end{lemma}
\begin{proof}
1)  $x(g)=\det(a_1a_2) \R^{\times 2}$,  \[\gamma(x(g), \psi^{\tfrac{1}{2}})^2=\gamma(\det(a_1a_2), \psi^{\tfrac{1}{2}})^2=\left\{ \begin{array}{cr} 1& \textrm{ if } \det(a_1a_2)>0,\\ (e^{-\tfrac{\pi i}{2}})^2=-1 & \textrm{ if } \det(a_1a_2)<0. \end{array}\right.\]
2)  $\gamma(\psi^{\tfrac{1}{2}})=e^{\tfrac{\pi i}{4}}$, $\gamma(\psi^{\tfrac{1}{2}})^2=e^{\tfrac{\pi i}{2}}$, $\gamma(-1, \psi^{\tfrac{1}{2}})=e^{-\tfrac{\pi i}{4}}/e^{\tfrac{\pi i}{4}}=e^{-\tfrac{\pi i}{2}}$.\\
3) It follows from (\ref{chap2mx})(\ref{chap3alpp}) and the above (1)(2).
\end{proof}
\begin{lemma}\label{h1h1h1h1eq}
 $\widetilde{c}_{X^{\ast}}(g_1^{h_{-1}}, g_2^{h_{-1}})\widetilde{c}_{X^{\ast}}(g_1, g_2)=1$.
\end{lemma}
\begin{proof}
\begin{align*}
 &\widetilde{c}_{X^{\ast}}(g_1^{h_{-1}}, g_2^{h_{-1}})\widetilde{c}_{X^{\ast}}(g_1, g_2)\\
 &=\widetilde{c}_{X^{\ast}}(g_1, g_2)^2\nu(-1,g_1)^{-1}\nu(-1,g_2)^{-1}\nu(-1,g_1g_2)\\
 &=\widetilde{c}_{X^{\ast}}(g_1, g_2)^2m_{X^{\ast}}(g_1)^2m_{X^{\ast}}(g_2)^2m_{X^{\ast}}(g_1g_2)^{-2}\\
 &=\overline{c}_{X^{\ast}}(g_1, g_2)^2\\
 &=1.
 \end{align*}
\end{proof}
\subsection{Fock model} 
Follow the notations from  Section \ref{ccccc}. 
\begin{lemma}
For $g=\begin{pmatrix}
a& b\\c& d\end{pmatrix}\in \Sp_{2m}(\R)$, there exists the following unitary equivalence $T^{\widehat{z}}_{g^{h_{-1}}}$ mapping from $\mathcal{H}_{F_{\widehat{z}}}$ to $\mathcal{H}_{F_{g^{h_{-1}}(\widehat{z})}}$, given by
\begin{align}
T^{\widehat{z}}_{g^{h_{-1}}}: \mathcal{H}_{F_{\widehat{z}}} \longrightarrow \mathcal{H}_{F_{g^{h_{-1}}(\widehat{z})}}, \phi \longmapsto & \eta((g^{h_{-1}})^{-1}, g^{h_{-1}}(\widehat{z}), v) \phi(vJ(g^{h_{-1}},\widehat{z}))\\
&=e^{-\pi i v(c\overline{z}+d) c^Tv^T}\phi(v(c\overline{z}+d)).
\end{align} 
Furthermore, the following diagram is commutative for any $h=(x,x^{\ast};t)\in \Ha(W)$.
\begin{equation}\label{eq1}
\begin{CD}
 \mathcal{H}_{F_{\widehat{z}}}@>T^{\widehat{z}}_{g^{h_{-1}}}>> \mathcal{H}_{F_{g^{h_{-1}}(\widehat{z})}} \\
       @V\pi_{F_{\widehat{z}},\psi}(h) VV @VV\pi_{F_{g^{h_{-1}}(\widehat{z})},\psi}(h^{g^{-1}}) V  \\
 \mathcal{H}_{F_{\widehat{z}}}@>T^{\widehat{z}}_{g^{h_{-1}}}>> \mathcal{H}_{F_{g^{h_{-1}}(\widehat{z})}} 
\end{CD}
\end{equation} 
\end{lemma}
\begin{proof}
1) \[g^{-1}= \begin{pmatrix}
d^{T} & -b^{T} \\
-c^{T} & a^{T}
\end{pmatrix}, g^{h_{-1}}= \begin{pmatrix} a& -b\\-c& d\end{pmatrix}, J(g^{h_{-1}}, \widehat{z})=c\overline{z}+d;\]
 \begin{align*}
 &\eta((g^{h_{-1}})^{-1}, g^{h_{-1}}(\widehat{z}), v)\\
 &=e^{2\pi i\Big[\tfrac{1}{2}\langle vg^{h_{-1}}, vJ((g^{h_{-1}})^{-1},g^{h_{-1}}(\widehat{z}))^{-1}\rangle\Big]}\\
&= e^{2\pi i\Big[\tfrac{1}{2} \langle v(-c), vJ(g^{h_{-1}}, \widehat{z})\rangle\Big]}\\
&=e^{2\pi i\Big[-\tfrac{1}{2} v(c\overline{z}+d) c^Tv^T\Big]}.
\end{align*}
2) Write  $z_1=\widehat{z}$,  $z_2=g^{h_{-1}}(\widehat{z})$, $h=(x,x^{\ast};t)$,  $g_1=g^{h_{-1}}$, $h_1=h^{h_{-1}}$, $h^{g^{-1}}=(x_1,x_1^{\ast};t) $, $w_1=(x_1,x_1^{\ast})$. Then:
 \[h^{g^{-1}}=h^{h_{-1}g_1^{-1}h_{-1}}=h_1^{g_1^{-1}h_{-1}};\]
 \[ (h_1^{g_1^{-1}})^{-1}=[h^{g^{-1}h_{-1}}]^{-1}=[x_1,-x_1^{\ast}; -t]^{-1}=[-x_1,x_1^{\ast};t];\]
 \begin{align*}
 &\pi_{F_{z_2},\psi}(h^{g^{-1}}) \circ T^{z_1}_{g_1} f(v)\\
 &=\eta((h_1^{g_1^{-1}})^{-1}, g_1(z_1), v)  [T^{z_1}_{g_1} f](v+w_1 \begin{pmatrix} -g_1(z_1)\\ 1\end{pmatrix}) \\
  &=\eta\big((h_1^{g_1^{-1}})^{-1}, g_1(z_1), v\big) \eta\big(g_1^{-1}, g_1(z_1), v+w_1 \begin{pmatrix} -g_1(z_1)\\ 1\end{pmatrix}\big) f\bigg([v+w_1 \begin{pmatrix} -g_1(z_1)\\ 1\end{pmatrix}]J(g_1,z_1)\bigg) ;
 & \\
 &T^{z_1}_{g_1}\circ \pi_{F_{z_1},\psi}(h) f(v)\\
 &=\eta(g_1^{-1}, g_1(z_1), v) [\pi_{F_{z_1},\psi}(h) f](vJ(g_1,z_1))\\
 &=\eta(g_1^{-1}, g_1(z_1), v) \eta((h^{h_{-1}})^{-1}, z_1, vJ(g_1,z_1)) f(vJ(g_1,z_1)+w \begin{pmatrix} -z_1\\ 1\end{pmatrix});
  \end{align*}
 i)  \begin{align*}
 &(w_1 \begin{pmatrix} -g_1(z_1)\\ 1\end{pmatrix})J(g_1,z_1)\\
 &=w g^{-1} \begin{pmatrix} -g_1(z_1)\\ 1\end{pmatrix}J(g_1,z_1)\\
 &=wh_{-1} g_1^{-1} \begin{pmatrix} -g_1(z_1)\\ -1\end{pmatrix}J(g_1,z_1)\\
 &=w\begin{pmatrix} -z_1\\ 1\end{pmatrix};
  \end{align*}
 ii) 
  \begin{align*}
  &\eta([g_1, h^{h_{-1}}]^{-1}, g_1(z_1), v)\\
  &=\eta\big([1,  h^{h_{-1}}]^{-1} [g_1,  0]^{-1} , g_1(z_1), v\big)\\
  &=\eta\Big((h^{h_{-1}})^{-1}, z_1, vJ\big(g_1^{-1}, g_1(z_1)\big)^{-1}\Big)\eta\big(g_1^{-1}, g_1(z_1), v\big)\\
  &=\eta\big((h^{h_{-1}})^{-1}, z_1, vJ(g_1,z_1)\big)\eta\big(g_1^{-1}, g_1(z_1), v\big);
   \end{align*}
  \begin{align*}
  &\eta([g_1, h^{h_{-1}}]^{-1}, g_1(z_1), v)\\
  &=\eta\big( [g_1,  0]^{-1} [1,  h_1^{g_1^{-1}}]^{-1} , g_1(z_1), v\big)\\
  &=\eta\big(g_1^{-1}, g_1(z_1), v+(-x_1,x_1^{\ast}) \cdot \begin{pmatrix} g_1(z_1)\\ 1\end{pmatrix}\big) \eta\big(( h_1^{g_1^{-1}})^{-1}, g_1(z_1), v\big).
   \end{align*}
\end{proof}
\begin{lemma}
\begin{itemize}
\item[(1)] $T^{g_2^{h_{-1}}(\widehat{z})}_{g_1^{h_{-1}}} \circ T^{\widehat{z}}_{g_2^{h_{-1}}}=T^{\widehat{z}}_{g_1^{h_{-1}}g_2^{h_{-1}}}$.
\item[(2)] For $g\in \Sp_{2m}(\R)$, $U_{ g^{h_{-1}}(\widehat{z'}), g^{h_{-1}}(\widehat{z})} \circ T^{\widehat{z}}_{g^{h_{-1}}}=\epsilon(g^{h_{-1}};\widehat{z'},\widehat{z}) T^{\widehat{z'}}_{g^{h_{-1}}} \circ U_{\widehat{z'},\widehat{z}}$.
\end{itemize}
\end{lemma}
\begin{proof}
See \cite[p.125]{Ta1}.\\
1) Put $r=g_2^{h_{-1}}$, $s=g_1^{h_{-1}}$. Then:
 \begin{align*}
 &T^{r(\widehat{z})}_{s} \circ T^{\widehat{z}}_{r} \phi(v)\\
 &=\eta\big(s^{-1}, sr(\widehat{z}), v\big) [T^{\widehat{z}}_{r}\phi]\Big(vJ\big(s,r(\widehat{z})\big)\Big)\\
 &=\eta\big(s^{-1}, sr(\widehat{z}), v\big) \eta\Big(r^{-1}, r(\widehat{z}), vJ\big(s,r(\widehat{z})\big)\Big) \phi\Big(vJ\big(s,r(\widehat{z})J(r,\widehat{z})\big)\Big)\\
 &=\eta\big(s^{-1}, sr(\widehat{z}), v\big) \eta\Big(r^{-1}, r(\widehat{z}), vJ\big(s,r(\widehat{z})\big)\Big) \phi\big(vJ(sr,\widehat{z})\big);\\
 &  \\
& T^{\widehat{z}}_{sr} \phi(v)\\
&=\eta\big((sr)^{-1}, sr(\widehat{z}), v\big)\phi\big(vJ(sr,\widehat{z})\big); \\
& \\
 &\eta\big((sr)^{-1}, sr(\widehat{z}), v\big)\\
 &= \eta\big(r^{-1}s^{-1}, sr(\widehat{z}), v\big)\\
&= \eta\Big(r^{-1}, r(\widehat{z}), v J\big(s^{-1}, sr(\widehat{z})\big)^{-1}\Big)  \eta\big(s^{-1}, sr(\widehat{z}), v\big)\\
&= \eta\Big(r^{-1}, r(\widehat{z}), v J\big(s, r(\widehat{z})\big)\Big)  \eta\big(s^{-1}, sr(\widehat{z}), v\big).
 \end{align*}
 2) 
 \[\begin{CD}
 \mathcal{H}_{F_{\widehat{z}}}@>T^{\widehat{z}}_{g^{h_{-1}}}>> \mathcal{H}_{F_{g^{h_{-1}}(\widehat{z})}}@>U_{ g^{h_{-1}}(\widehat{z'}), g^{h_{-1}}(\widehat{z})} >> \mathcal{H}_{F_{g^{h_{-1}}(\widehat{z'})}} 
 \end{CD}\]
 \[\begin{CD}
 \mathcal{H}_{F_{\widehat{z}}}@>U_{\widehat{z'},\widehat{z}} >> \mathcal{H}_{F_{\widehat{z'}}}@>T^{\widehat{z'}}_{g^{h_{-1}}}>> \mathcal{H}_{F_{g^{h_{-1}}(\widehat{z'})}}
\end{CD}\]
Thus, the difference between $U_{g^{h_{-1}}(\widehat{z'}), g^{h_{-1}}(\widehat{z})} \circ T^{\widehat{z}}_{g^{h{-1}}}$ and $T^{\widehat{z'}}_{g^{h{-1}}} \circ U_{\widehat{z'},\widehat{z}}$ is a constant. This constant is precisely $\epsilon(g^{h_{-1}};\widehat{z'},\widehat{z})$, as can be inferred from the reference \cite[pp.400-401]{Sa1} or \cite[p.125]{Ta1}.
\end{proof}

Then the Heisenberg representation $\pi_{F_{\widehat{z}},\psi}$ of $\Ha(W)$,  as described in Example \ref{ex4}, can be extended to be  a projective representation of   $\Sp_{2m}(\R)\ltimes \Ha(W)$ by defining the action on functions $f(v)\in \mathcal{H}_{F_{\widehat{z}}}$ as follows:
\[
\left\{ \begin{array}{l}
 \pi_{F_{\widehat{z}}, \psi}(h)f(v)=\eta((h^{h_{-1}})^{-1}, \widehat{z}, v) f(v+x\overline{z}+x^{\ast}),\\
 \pi_{F_{\widehat{z}},\psi}(g) f(v)=U_{\widehat{z}, g^{h_{-1}}(\widehat{z})} \circ T^{\widehat{z}}_{g^{h_{-1}}} f(v),
\end{array}\right.\]
where $g \in \Sp_{2m}(\mathbb{R})$ and $h = (x, x^{\ast}; t) \in \Ha(W)$. 
\begin{lemma}
\begin{itemize}
\item[(1)] It is well-defined.
\item[(2)] $\pi_{F_{\widehat{z}},\psi}(g_1) \pi_{F_{\widehat{z}},\psi}(g_2)=\beta_{\widehat{z}}(g_1,g_2)^{-1}\pi_{F_{\widehat{z}},\psi}(g_1g_2)$.
\end{itemize}
\end{lemma}
\begin{proof}
\[[g,h]=[g,0]\cdot [1,h]=[1,h^{g^{-1}}]\cdot[g,0].\]
\begin{align*}
&\pi_{F_{\widehat{z}},\psi}(g)  \pi_{F_{\widehat{z}},  \psi}(h)\\
&=U_{\widehat{z}, g^{h_{-1}}(\widehat{z})} \circ T^{\widehat{z}}_{g^{h_{-1}}}\circ \pi_{ F_{\widehat{z}}, \psi}(h)\\
&\xlongequal{\widehat{z}_1=g^{h_{-1}}(\widehat{z})}U_{\widehat{z}, \widehat{z}_1} \circ \pi_{F_{\widehat{z}_1},\psi}(h^{g^{-1}})\circ T^{\widehat{z}}_{g^{h_{-1}}}\\
&= \pi_{F_{\widehat{z}},\psi}(h^{g^{-1}})\circ  U_{\widehat{z}, \widehat{z}_1} \circ T^{\widehat{z}}_{g^{h_{-1}}}\\  
&=\pi_{F_{\widehat{z}},\psi}(h^{g^{-1}}) \pi_{F_{\widehat{z}},\psi}(g).
  \end{align*}
  By the above lemma, 
  \[U_{ g_1^{h_{-1}}(\widehat{z}), g_1^{h_{-1}}g_2^{h_{-1}}(\widehat{z})} \circ T^{g_2^{h_{-1}}(\widehat{z})}_{g_1^{h_{-1}}}=\epsilon(g_1^{h_{-1}};\widehat{z},g_2^{h_{-1}}(\widehat{z})) T^{\widehat{z}}_{g_1^{h_{-1}}} \circ U_{\widehat{z},g_2^{h_{-1}}(\widehat{z})}\]
  \begin{align*}
&\pi_{F_{\widehat{z}},\psi}(g_1)  \pi_{F_{\widehat{z}}, \psi}(g_2)\\
&=U_{\widehat{z}, g_1^{h_{-1}}(\widehat{z})} \circ T^{\widehat{z}}_{g_1^{h_{-1}}}\circ U_{\widehat{z}, g_2^{h_{-1}}(\widehat{z})} \circ T^{\widehat{z}}_{g_2^{h_{-1}}}\\
&=\epsilon(g_1^{h_{-1}};\widehat{z},g_2^{h_{-1}}(\widehat{z}))^{-1}U_{\widehat{z}, g_1^{h_{-1}}(\widehat{z})} \circ U_{ g_1^{h_{-1}}(\widehat{z}), g_1^{h_{-1}}g_2^{h_{-1}}(\widehat{z})} \circ T^{g_2^{h_{-1}}(\widehat{z})}_{g_1^{h_{-1}}}  \circ T^{\widehat{z}}_{g_2^{h_{-1}}}\\
&=\beta_{\widehat{z}}(g_1,g_2)^{-1} \pi_{F_{\widehat{z}},\psi}(g_1g_2). 
  \end{align*}
\end{proof}
\begin{example}\label{exm}
Take $z=z_0=i1_{m}$. Let   $g=\begin{pmatrix}a & -b\\ b& a\end{pmatrix}\in \U_{m}(\C)$. Then $g^{h_{-1}}=\begin{pmatrix}a & b\\ -b& a\end{pmatrix}$. 
Additionally, the representation $\pi_{F_{z_0},\psi}(g)$ acting on a function $f(v)$ is defined as:
$$ \pi_{F_{z_0},\psi}(g) f(v)=e^{-\pi i v(-bi+a) b^Tv^T}f(v(-bi+a)).$$
\begin{itemize}
\item[(1)] If we take $f(v)=e^{-\tfrac{\pi}{2}vv^T}$, then $f(v)\in \mathcal{H}_{F_{z_0}}$. Moreover, $\pi_{F_{z_0},\psi}(g) f(v)=f(v)$.
\begin{align*}
&\int_{\C^m}  e^{-2\pi v_y v_y^T} \mid f(v)\mid^2 d_{z_0}(v)\\
&=\int_{\C^m}  e^{-2\pi v_y v_y^T} e^{-\pi (v_x v_x^T-v_yv_y^T)} d_{z_0}(v)\\
&=\int_{\R^m \times \R^m}  e^{-\pi (v_x v_x^T+v_yv_y^T)} dv_xdv_y<+\infty.
\end{align*}
\begin{align*}
\pi_{F_{z_0},\psi}(g) f(v)&=e^{-\pi i v(-bi+a) b^Tv^T}e^{-\tfrac{\pi}{2}v(-bi+a)(-bi+a)^Tv^T}\\
&=e^{\pi v(-ia b^T)v^T+ v(-bb^T)v^T}e^{\tfrac{\pi}{2}v(-aa^T+bb^T+ab^Ti+ba^Ti)v^T}\\
&=e^{-\tfrac{\pi}{2}vv^T}\\
&=f(v).
\end{align*}
\item[(2)] If we take $f_k(v)=v_kie^{-\tfrac{\pi}{2}vv^T}$ for $v=(v_1, \cdots, v_m)\in V$, then $f_k(v)\in \mathcal{H}_{F_{z_0}}$. Moreover, if let  $F(v)=\begin{pmatrix} f_1(v)\\ \vdots\\f_m(v)\end{pmatrix}$, then   $\pi_{F_{z_0},\psi}(g)F(v)=(-bi+a)^TF(v)$.
\end{itemize}
\end{example}
\begin{align*}
\pi_{F_{z_0},\psi}(g) f_k(v)
&=e^{-\pi i v(-bi+a) b^Tv^T}e^{-\tfrac{\pi}{2}v(-bi+a)(-bi+a)^Tv^T}(iv(-bi+a))_k\\
&=f(v)(iv(-bi+a))_k.\\
& \\
\pi_{F_{z_0},\psi}(g) F(v)&=f(v)(iv(-bi+a))^T\\
&=(-bi+a)^T F(v).
\end{align*}
\begin{example}\label{example232} 
Take $z=\widehat{z}=z_0=i1_{m}$.  Let us give the explicit actions of $ u(b), h(a)$, for  $u(b)=\begin{pmatrix}
  1&b\\
  0 & 1
\end{pmatrix}$,  $h(a)=\begin{pmatrix}
  a&0\\
  0 & (a^T)^{-1}
\end{pmatrix} \in \Sp_{2m}(\R)$. 
\end{example}

1) If $g=h(a)$, then $h(a)^{h_{-1}}=h(a)$,  $\widehat{z'}=g^{h_{-1}}(\widehat{z})=a\widehat{z} a^{T}=iaa^T$. Write $v=v_x+iv_y$ and $v'=v'_x+iv'_y$.
\begin{align*}
 \gamma(\widehat{z}, \widehat{z'})&={\det}^{-1/2}\big(\tfrac{\widehat{z}-\overline{\widehat{z'}}}{2i}\big) \cdot \det(\Im \widehat{z})^{1/4}\cdot \det(\Im \widehat{z'})^{1/4}\\
 &={\det}^{-1/2}\big(\tfrac{1_m+aa^T}{2}\big) \cdot |\det(a)|^{1/2};\\
& \\
 \kappa_{\widehat{z'}}(v'a^{T},v'a^{T}) & =e^{-2\pi v'_ya^{T}[aa^T]^{-1} a(v'_y)^T}\\
 &=e^{-2\pi v'_y(v'_y)^T}\\
 &=\kappa_{z_0}(v',v');\\
&\\
 d_{\widehat{z'}}(v')=&\det(aa^T)^{-1} d_{z_0}(v');
 \end{align*}
  \begin{align*}
 \kappa(\widehat{z},v; \widehat{z'},v'a^T)^{-1} 
 &=e^{-\pi i(v-\overline{v'}a^T)\cdot (\widehat{z}-\overline{\widehat{z'}})^{-1}\cdot \big(v-\overline{v'}a^T\big)^T}\\
  &=e^{-\pi(v-\overline{v'}a^T)\cdot (1_m+aa^T)^{-1}\cdot \big(v-\overline{v'}a^T\big)^T}.
 \end{align*}
\begin{align*}
& \pi_{F_z,\psi}(g) f(v)\\
&=U_{\widehat{z}, g^{h_{-1}}(\widehat{z})} \circ T^{\widehat{z}}_{g^{h_{-1}}} f(v)\\
&= \gamma(\widehat{z}, \widehat{z'})\int_{\C^m} \kappa(\widehat{z},v; \widehat{z'},v'')^{-1}\big[T^{\widehat{z}}_{g^{h_{-1}}} f\big](v'') \kappa_{\widehat{z'}}(v'',v'') d_{\widehat{z'}}(v'')\\
&= \gamma(\widehat{z}, \widehat{z'})\int_{\C^m} \kappa(\widehat{z},v; \widehat{z'},v'')^{-1}f(v''(a^{T})^{-1})\kappa_{\widehat{z'}}(v'',v'') d_{\widehat{z'}}(v'')\\
& \stackrel{v'=v''(a^{T})^{-1}}{=}|\det(a)|^2\gamma(\widehat{z}, \widehat{z'})\int_{\C^m} \kappa(\widehat{z},v; \widehat{z'},v'a^T)^{-1}f(v')\kappa_{\widehat{z'}}(v'a^{T},v'a^{T}) d_{\widehat{z'}}(v')\\
&=|\det(a)|^{\tfrac{1}{2}}{\det}^{-1/2}\big(\tfrac{1_m+aa^T}{2}\big)\int_{\C^m} e^{-\pi(v-\overline{v'}a^T)\cdot (1_m+aa^T)^{-1}\cdot \big(v-\overline{v'}a^T\big)^T}f(v')\kappa_{z_0}(v',v') d_{z_0}(v').
\end{align*}
2) If $g=u(b)$, then $u(b)^{h_{-1}}=u(-b)$,  $\widehat{z'}=g^{h_{-1}}(\widehat{z})=1_m i-b$.
\begin{align*}
\gamma(\widehat{z}, \widehat{z'})
 &={\det}^{-1/2}\big(\tfrac{\widehat{z}-\overline{\widehat{z'}}}{2i}\big) \cdot \det(\Im \widehat{z})^{1/4}\cdot \det(\Im \widehat{z'})^{1/4}\\
 &={\det}^{-1/2}\big(\tfrac{2i 1_m+b}{2i}\big);
 \end{align*}
 \begin{align*}
 \kappa_z(v',v')&\stackrel{v'=v'_x+iv'_y}{=}e^{-2\pi v'_y(v'_y)^T};\\
 & \\
  \kappa(\widehat{z},v; \widehat{z'},v')^{-1}
 &=e^{-\pi i(v-\overline{v'})\cdot (\widehat{z}-\overline{\widehat{z'}})^{-1} \cdot (v-\overline{v'})^T}\\
 &=e^{-\pi i(v-\overline{v'})(2i 1_m+b)^{-1} (v-\overline{v'})^T};
 \end{align*}
\begin{align*}
& \pi_{F_z,\psi}(g) f(v)=U_{\widehat{z}, g^{h_{-1}}(\widehat{z})} \circ T^{\widehat{z}}_{g^{h_{-1}}} f(v)\\
&= \gamma(\widehat{z}, \widehat{z'})\int_{\C^m} \kappa(\widehat{z},v; \widehat{z'},v')^{-1}\big[T^{\widehat{z}}_{g^{h_{-1}}} f\big](v') \kappa_z(v',v') d_z(v')\\
&= \gamma(\widehat{z}, \widehat{z'})\int_{\C^m} \kappa(\widehat{z},v; \widehat{z'},v')^{-1}f(v')\kappa_z(v',v') d_z(v').
\end{align*}
\section{Intertwining operators I}
\subsection{Intertwining operators}\label{interwop} Retain the notations from Examples \ref{ex1} and \ref{ex2}. By \cite[pp.164-165]{We}, or \cite[pp.142-145]{LiVe}, there exists  a pair of explicit isomorphisms between   $\mathcal{H}(X^{\ast})\simeq L^2(X)$ and  $\mathcal{H}(L)$, given as follows:
 \begin{equation}\label{inter11}
 \begin{split}
\theta_{L, X^{\ast}}(f')( w)&=\sum_{l\in L/L\cap X^{\ast}} f'( w+l)\psi(\tfrac{\langle l, w\rangle}{2}+\tfrac{\langle x_{l}, x^{\ast}_l\rangle}{2})\\
&=\sum_{l\in L\cap X}f'(x+l)\psi(\langle l, w\rangle) \psi(\tfrac{\langle x, x^{\ast}\rangle}{2}),
    \end{split}
 \end{equation}
 \begin{equation}\label{inter12}
 \begin{split}
 \theta_{X^{\ast}, L}(f)(x)&=  \int_{X^{\ast}/X^{\ast}\cap L} f([\dot{x}^{\ast},0]\cdot [x, 0]) d\dot{x}^{\ast}\\
  &= \int_{X^{\ast}/X^{\ast}\cap L} \psi(\tfrac{\langle\dot{x}^{\ast}, x\rangle}{2})f(x+\dot{x}^{\ast}) d\dot{x}^{\ast},
  \end{split}
      \end{equation}
    for $ w=x+ x^{\ast} \in W$, $l=x_l+ x_l^{\ast}\in L$, with $x,x_l\in X$,  $x^{\ast},x^{\ast}_l\in X^{\ast}$, and  $f'\in \mathcal{H}(X^{\ast})$ with $f'|_X \in S(X)\subseteq L^2(X)$,  $f\in L^1(W)\cap \mathcal{H}(L)$. Through $\theta_{L, X^{\ast}}$ and $\theta_{X^{\ast}, L}$, we transfer the action of $\Sp(W)$ on $ \mathcal{H}(X^{\ast})$  to $\mathcal{H}(L)$. For $g\in  \Sp(W)$, and $f=\theta_{L, X^{\ast}}(f')\in \mathcal{H}(L)$,  we can define
     \begin{equation}
     \pi_{L, \psi}(g)f=\theta_{L, X^{\ast}}[\pi_{X^{\ast}, \psi}(g)(f')]=\theta_{L, X^{\ast}}[\pi_{X^{\ast}, \psi}(g)\theta_{X^{\ast},L}(f)].
     \end{equation}
Under such action, for $g_1, g_2\in \Sp(W)$,
     \begin{equation}
     \begin{split}
     \pi_{L, \psi}(g_1)[\pi_{L, \psi}(g_2)f]&=\theta_{L, X^{\ast}}[\pi_{X^{\ast}, \psi}(g_1)\theta_{X^{\ast},L}]([\pi_{L, \psi}(g_2)f])\\
     &=\theta_{L, X^{\ast}}(\pi_{X^{\ast}, \psi}(g_1)\theta_{X^{\ast},L}\theta_{L, X^{\ast}}
     [\pi_{X^{\ast}, \psi}(g_2)\theta_{X^{\ast},L}(f)])\\
     &=\theta_{L, X^{\ast}}[\pi_{X^{\ast}, \psi}(g_1)\pi_{X^{\ast}, \psi}(g_2)\theta_{X^{\ast},L}(f)]\\
     &=\widetilde{c}_{X^{\ast}}(g_1, g_2)\pi_{L, \psi}(g_1g_2)f.
     \end{split}
     \end{equation}
 Under the basis $\{ e_1, \cdots, e_m; e_1^{\ast}, \cdots, e_m^{\ast}\}$ of $W$, we  identity $\Sp(W)$ with $\Sp_{2m}(\R)$, and $\Sp(L)$ with $\Sp_{2m}(\Z)$.
 \subsection{Explicit Elements I } 
 \ \\
 Case $1^{}$: $g=u(b)\in \Gamma_m(2)$, with  $b=b^T$.
\begin{equation*}
\begin{split}
 &\pi_{L, \psi}(g)f( w)\\
 &=\sum_{l\in L\cap X}\pi_{X^{\ast}, \psi}(g)(f')(x+l)\psi(\langle l, w\rangle) \psi(\tfrac{\langle x, x^{\ast}\rangle}{2}) \\
&=\sum_{l\in L\cap X}\psi(\tfrac{1}{2}\langle (x+l),(x+l)b\rangle) f'(x+l)\psi(\langle l, w\rangle)\psi(\tfrac{\langle x, x^{\ast}\rangle}{2})\\
&= \sum_{l\in L\cap X} \int_{{X^\ast}/X^{\ast}\cap L}  \psi(\tfrac{1}{2}\langle (x+l),(x+l)b\rangle) \psi(\tfrac{\langle\dot{y}^{\ast}, x+l\rangle}{2})f(x+l+\dot{y}^{\ast})  \psi(\langle l, w\rangle)\psi(\tfrac{\langle x, x^{\ast}\rangle}{2}) d\dot{y}^{\ast}\\
&=  \sum_{l\in L\cap X} \int_{X^{\ast}/X^{\ast}\cap L}   \psi(\tfrac{\langle\dot{y}^{\ast}, x+l\rangle}{2})f(x+l+\dot{y}^{\ast})  \psi(\langle l, x^{\ast}+xb\rangle)\psi(\tfrac{\langle x, x^{\ast}+ xb\rangle}{2}) d\dot{y}^{\ast}\\
&= f(wg).
\end{split}
\end{equation*}   
Case $2$:  $g=h(a)\in \Sp_{2m}(\Z)$.
\begin{equation*}
\begin{split}
&\pi_{L, \psi}(g)f(w)\\
&=\sum_{l\in L\cap X}\pi_{X^{\ast},\psi}(g)(f')(x+l)\psi(\langle l, w\rangle) \psi(\tfrac{\langle x, x^{\ast}\rangle}{2}) \\
&=\sum_{l\in L\cap X}|\det(a)|^{1/2}f'((x+l)a)\psi(\langle l, w\rangle) \psi(\tfrac{\langle x, x^{\ast}\rangle}{2})\\
&= \sum_{l\in L\cap X}\int_{X^{\ast}/X^{\ast}\cap L} \psi(\tfrac{\langle\dot{y}^{\ast}, (x+l)a\rangle}{2})f(xa+la+\dot{y}^{\ast})  \psi(\langle l, w\rangle) \psi(\tfrac{\langle x, x^{\ast}\rangle}{2})d\dot{y}^{\ast}\\
&= \sum_{l\in L\cap X}\int_{X^{\ast}/X^{\ast}\cap L} \psi(\tfrac{\langle\dot{y}^{\ast}, (x+l)a\rangle}{2})f(xa+la+\dot{y}^{\ast}) \psi(\langle la, x^{\ast}(a^{\ast})^{-1} \rangle) \psi(\tfrac{\langle xa, x^{\ast}(a^{\ast})^{-1} \rangle}{2}) d\dot{y}^{\ast}\\
&= \sum_{l\in L\cap X}\int_{X^{\ast}/X^{\ast}\cap L} \psi(\tfrac{\langle\dot{y}^{\ast}, xa+l\rangle}{2})f(xa+l+\dot{y}^{\ast}) \psi(\langle l, x^{\ast}(a^{\ast})^{-1}\rangle) \psi(\tfrac{\langle xa, x^{\ast}(a^{\ast})^{-1}\rangle}{2}) d\dot{y}^{\ast}\\
&=f(wg).
\end{split}
\end{equation*}
Case 3: $g=\omega_S \in \Sp_{2m}(\Z)$.  Let $w=x+ x^{\ast}$, for $x=x_s+x_{s'}\in \mathcal{X}_S\oplus\mathcal{X}_{S'} $, $x^{\ast}=x^{\ast}_s+x^{\ast}_{s'}\in \mathcal{X}^{\ast}_S\oplus\mathcal{X}^{\ast}_{S'} $.
\begin{equation*}\label{eq31}
\begin{split}
&\pi_{L, \psi}(g)f(w)\\
& =\sum_{l\in L\cap X}[\pi_{X^{\ast}, \psi}(g)f'](x+l) \psi(\langle l, w\rangle) \psi(\tfrac{\langle x, x^{\ast}\rangle}{2})\\
&= \sum_{l\in L\cap X} \int_{X_S^{\ast}}\psi(\langle x_s+l_s,  z_s^{\ast}\rangle ) f'((x_{s'}+l_{s'})\omega_S+z_s^{\ast}\omega^{-1}_S) \psi(\langle l, w\rangle) \psi(\tfrac{\langle x, x^{\ast}\rangle}{2}) dz_s^{\ast}\\
&= \sum_{l\in L\cap X}\int_{X_S^{\ast}}\psi(\langle x_s+l_s,  z_s^{\ast}\rangle ) f'((x_{s'}+l_{s'})\omega_S+z_s^{\ast}\omega^{-1}_S) \psi(\langle l, w\rangle) \psi(\tfrac{\langle x, x^{\ast}\rangle}{2}) dz_s^{\ast}\\
&=\sum_{l\in L\cap X} \int_{X_S^{\ast}}   \int_{X^{\ast}/X^{\ast}\cap L}\psi(\langle x_s+l_s,  z_s^{\ast}\rangle )  \psi(\tfrac{\langle\dot{y}^{\ast}, x_{s'}+l_{s'}+z_s^{\ast}\omega^{-1}_S\rangle}{2})f(x_{s'}+l_{s'}+z_s^{\ast}\omega^{-1}_S+\dot{y}^{\ast}) \psi(\langle l, w\rangle) \psi(\tfrac{\langle x, x^{\ast}\rangle}{2})d\dot{y}^{\ast} dz_s^{\ast}\\
&=\sum_{l\in L\cap X} \int_{X_S^{\ast}}   \int_{X^{\ast}/X^{\ast}\cap L} \psi(\langle x_s+l_s,  z_s^{\ast}\rangle ) \psi(\tfrac{\langle\dot{y}^{\ast}, x_{s'}+z_s^{\ast}\omega^{-1}_S\rangle}{2})f(x_{s'}+z_s^{\ast}\omega^{-1}_S+\dot{y}^{\ast}) \psi(-\langle l_{s'}, \dot{y}^{\ast}\rangle)\psi(\langle l, w\rangle) \psi(\tfrac{\langle x, x^{\ast}\rangle}{2})d\dot{y}^{\ast} dz_s^{\ast}\\
&=  \int_{X^{\ast}/X^{\ast}\cap L} f(x_{s'}+z_s^{\ast}\omega^{-1}_S+\dot{y}^{\ast})\psi(\langle x_s,  z_s^{\ast})) \psi(\tfrac{\langle\dot{y}^{\ast}, x_{s'}+z_s^{\ast}\omega^{-1}_S\rangle}{2})\psi(\tfrac{\langle x, x^{\ast}\rangle}{2})\cdot[\sum_{l\in L\cap X}\psi(-\langle l_{s'}, \dot{y}^{\ast}\rangle) \psi(\langle l, w\rangle)\psi(\langle l_s,  z_s^{\ast}\rangle ) ] d\dot{y}^{\ast}dz_s^{\ast}\\
&=\int_{L\cap X_S^{\ast}}   \int_{X_S^{\ast}/X_S^{\ast}\cap L}    f(x_{s'}\omega_S+(-x_{s}^{\ast}+z_s^{\ast})\omega^{-1}_S+\dot{y}_s^{\ast}+ x_{s'}^{\ast})\psi(\langle x_s,  z_s^{\ast})) \psi(\tfrac{\langle\dot{y}_s^{\ast}+ x_{s'}^{\ast}, x_{s'}+(-x_s^{\ast}+z_s^{\ast})\omega^{-1}_S\rangle}{2})\psi(\tfrac{\langle x, x^{\ast}\rangle}{2})dz_s^{\ast}d\dot{y}_s^{\ast}\\
&=\psi(\tfrac{\langle x_s, x_s^{\ast}\rangle}{2})\int_{L\cap X_S^{\ast}}   \int_{X_S^{\ast}/X_S^{\ast}\cap L}    f(x_{s'}\omega_S+x_{s}^{\ast}\omega_S+z_s^{\ast}\omega^{-1}_S+\dot{y}_s^{\ast}+ x_{s'}^{\ast}\omega_S)\psi(\langle x_s,  z_s^{\ast}\rangle)) \psi(\tfrac{\langle\dot{y}_s^{\ast}, (-x_s^{\ast}+z_s^{\ast})\omega^{-1}_S\rangle}{2})dz_s^{\ast}d\dot{y}_s^{\ast}\\
&=\psi(\tfrac{\langle x_s, x_s^{\ast}\rangle}{2})\int_{L\cap X_S^{\ast}}   \int_{X_S^{\ast}/X_S^{\ast}\cap L}    f(x_{s'}\omega_S+z_s^{\ast}\omega^{-1}_S+\dot{y}_s^{\ast}+x^{\ast}\omega_S)\psi(\langle x_s,  z_s^{\ast}\rangle)) \psi(\tfrac{\langle\dot{y}_s^{\ast}, (-x_s^{\ast}+z_s^{\ast})\omega^{-1}_S\rangle}{2})dz_s^{\ast}d\dot{y}_s^{\ast}\\
&=\psi(\tfrac{\langle x_s, x_s^{\ast}\rangle}{2})\int_{L\cap X_S^{\ast}}   \int_{X_S^{\ast}/X_S^{\ast}\cap L}    f(x_{s'}\omega_S+z_s^{\ast}\omega^{-1}_S+\dot{y}_s^{\ast}+x^{\ast}\omega_S)\psi(\langle x_s,  z_s^{\ast}\rangle)) \psi(\tfrac{\langle\dot{y}_s^{\ast}, (-x_s^{\ast}+z_s^{\ast})\omega^{-1}_S\rangle}{2})dz_s^{\ast}d\dot{y}_s^{\ast}\\
&=\psi(\tfrac{\langle x_s\omega_S, x_s^{\ast}\omega_S\rangle}{2})\int_{L\cap X_S}   \int_{X_S^{\ast}/X_S^{\ast}\cap L}      f(x_{s'}\omega_S+x_{s'}^{\ast}\omega_S+l_s+\dot{y}_s^{\ast}+x_s^{\ast}\omega_S)\psi(\langle l_s,x_s\omega_S\rangle)) \psi(\tfrac{\langle\dot{y}_s^{\ast},l_s+x_s^{\ast}\omega_S\rangle}{2})dl_sd\dot{y}_s^{\ast}\\
&=f(w\omega_S ).
\end{split}
\end{equation*}
Case $4$: $g=u_-(c)\in \Gamma_m(2)$, with  $ c=c^T$. Then $u(-c)=\omega u_-(c)\omega^{-1}$. 
\begin{equation}\label{chap6eq4}
\begin{split}
\pi_{L, \psi}(g)f(w)&=\pi_{L,\psi}(\omega^{-1} u(-c)\omega) f(w)\\
&=\widetilde{c}_{X^{\ast}}(\omega^{-1}, u(-c))^{-1}\widetilde{c}_{X^{\ast}}(\omega^{-1}u(-c), \omega)^{-1}\pi_{L,\psi}(\omega^{-1})\pi_{L,\psi}(u(-c)) \pi_{L,\psi}(\omega) f(w)\\
&=\widetilde{c}_{X^{\ast}}(\omega^{-1}u(-c), \omega)^{-1}\pi_{L,\psi}(\omega^{-1})\pi_{L,\psi}(u(-c)) \pi_{L,\psi}(\omega) f(w)\\
&=\widetilde{c}_{X^{\ast}}(\omega^{-1}u(-c), \omega)^{-1} f(wg).
\end{split}
\end{equation}
Note:
\[\widetilde{c}_{X^{\ast}}(\omega^{-1}u(-c), \omega)\widetilde{c}_{X^{\ast}}(\omega^{-1}u(-c) \omega, \omega^{-1})=\widetilde{c}_{X^{\ast}}(\omega^{-1}u(-c) , 1)\widetilde{c}_{X^{\ast}}( \omega, \omega^{-1})=1,\]
\[\widetilde{c}_{X^{\ast}}(\omega^{-1}u(-c), \omega)^{-1} =\widetilde{c}_{X^{\ast}}(\omega^{-1}u(-c) \omega, \omega^{-1})=\widetilde{c}_{X^{\ast}}(u_-(c) , \omega).\]
By Prop.\ref{gamma2q}, for any element $g\in \Gamma_m(2)$ ($m\geq 2$),  it can be expressed as a product of elements of the form  $u(b)$, $u_-(c)$. Consequently, there exists an element  $\epsilon_g\in \mu_8$, which is a product of Weil's indices, such that the following equation holds:
\begin{equation}\label{chapeq5}
 \pi_{L, \psi}(g)f(w)= \epsilon_g f(wg).
 \end{equation}
  Note that $\epsilon_g$ is uniquely determined by $(\ref{chapeq5})$. For $g_1, g_2\in \Gamma_m(2)$, we have:
 \begin{align*}
 &\pi_{L, \psi}(g_1)\pi_{L, \psi}(g_2)f(w)\\
 &=\epsilon_{g_1} \epsilon_{g_2}f(wg_1g_2)\\
 &=\widetilde{c}_{X^{\ast}}(g_1,g_2)\pi_{L, \psi}(g_1g_2)f(w)\\
 &= \widetilde{c}_{X^{\ast}}(g_1,g_2)\epsilon_{g_1g_2}f(wg_1g_2).
 \end{align*}
 Hence
 $$\widetilde{c}_{X^{\ast}}(g_1,g_2)=\epsilon_{g_1} \epsilon_{g_2}\epsilon_{g_1g_2}^{-1}.$$
 \begin{example}
 \begin{itemize}
 \item[(1)] If $g=u(b)\in \Gamma_m(2)$, then $\epsilon_{g}=1 $. 
 \item [(2)] If $g=h(b)\in \Gamma_m(2)$, then $\epsilon_{g}=1 $. 
 \item[(3)] If $g=u_-(c)\in \Gamma_m(2)$, then $\epsilon_{g}=\widetilde{c}_{X^{\ast}}(u_-(c), \omega) $. 
 \end{itemize}
\end{example}
\subsection{Explicit Elements II }\label{exlicit}
We focus on the other elements 
$g$ of the form $g=u(b)\in \Sp_{2m}(\Z) \setminus \Gamma_m(2)$. \\
Case 5: $g=u_{ij}=u(b) \in \Sp_{2m}(\Z)$, for $i\neq j$. 
\begin{equation*}
\begin{split}
 &\pi_{L, \psi}(g)f( w)\\
&= \sum_{l\in L\cap X} \int_{{X^\ast}/X^{\ast}\cap L}  \psi(\tfrac{1}{2}\langle (x+l),(x+l)b\rangle) \psi(\tfrac{\langle\dot{y}^{\ast}, x+l\rangle}{2})f(x+l+\dot{y}^{\ast})  \psi(\langle l, w\rangle)\psi(\tfrac{\langle x, x^{\ast}\rangle}{2}) d\dot{y}^{\ast}\\
&=\sum_{l\in L\cap X} \int_{X^{\ast}/X^{\ast}\cap L} \psi(\tfrac{1}{2}\langle l,lb\rangle) f(x+l+\dot{y}^{\ast}) \psi(-\tfrac{\langle x+l, \dot{y}^{\ast}\rangle}{2})   \psi(\langle l, x^{\ast}+xb)\psi(\tfrac{\langle x, x^{\ast}+xb\rangle}{2})d \dot{y}^{\ast}\\
&=\sum_{l\in L\cap X} \int_{X^{\ast}/X^{\ast}\cap L} f(x+l+\dot{y}^{\ast}) \psi(-\tfrac{\langle x+l, \dot{y}^{\ast}\rangle}{2})   \psi(\langle l, x^{\ast}+xb)\psi(\tfrac{\langle x, x^{\ast}+xb\rangle}{2})d \dot{y}^{\ast}\\
&=f(wg).
\end{split}
\end{equation*}
Case 6: $g=u_{ii}=u(b) \in \Sp_{2m}(\Z)$. 
\begin{equation*}\label{ui}
\begin{split}
 &\pi_{L, \psi}(g)f( w)\\
&=\sum_{l\in L\cap X} \int_{X^{\ast}/X^{\ast}\cap L} \psi(\tfrac{1}{2}\langle l,lb\rangle) f(x+l+\dot{y}^{\ast}) \psi(-\tfrac{\langle x+l, \dot{y}^{\ast}\rangle}{2})   \psi(\langle l, x^{\ast}+xb)\psi(\tfrac{\langle x, x^{\ast}+xb\rangle}{2})d \dot{y}^{\ast}\\
&=\psi(-\tfrac{1}{4}x_i)\sum_{l\in L\cap X} \int_{X^{\ast}/X^{\ast}\cap L} f(x+l+\dot{y}^{\ast}) \psi(-\tfrac{\langle x+l, \dot{y}^{\ast}\rangle}{2})   \psi(\langle l, x^{\ast}+xb+\tfrac{1}{2} e_{i}^{\ast}\rangle)\psi(\tfrac{\langle x, x^{\ast}+xb+\tfrac{1}{2} e_{i}^{\ast}\rangle}{2})d \dot{y}^{\ast}\\
&=\psi(-\tfrac{1}{4}x_i)f(x+x^{\ast}+xb+\tfrac{1}{2} e_i^{\ast} )\\
&=\psi(-\tfrac{1}{4}x_i)f(wg+\tfrac{1}{2} e_i^{\ast}).
\end{split}
\end{equation*} 
Case 7: $g=u^-_{ii}=\omega^{-1} u_{ii} \omega\in \Sp_{2m}(\Z)$, $w=(x,x^{\ast})$, $w\omega^{-1}=(x,x^{\ast})\begin{pmatrix}0 & 1_m \\ -1_m & 0\end{pmatrix}=(-x^{\ast}, x)$.
\begin{equation*}\label{ui-}
\begin{split}
 &\pi_{L, \psi}(g)f( w)\\
 &=\pi_{L, \psi}(\omega^{-1} u_{ii} \omega)f( w)\\
 &=\widetilde{c}_{X^{\ast}}(\omega^{-1}, u_{ii})^{-1}\widetilde{c}_{X^{\ast}}(\omega^{-1}u_{ii}, \omega)^{-1}\pi_{L,\psi}(\omega^{-1})\pi_{L,\psi}(u_{ii}) \pi_{L,\psi}(\omega) f(w)\\
&=\widetilde{c}_{X^{\ast}}(\omega^{-1}u_{ii}, \omega)^{-1}\pi_{L,\psi}(\omega^{-1})\pi_{L,\psi}(u_{ii}) \pi_{L,\psi}(\omega) f(w)\\
&=\widetilde{c}_{X^{\ast}}(\omega^{-1}u_{ii}, \omega)^{-1}[\pi_{L,\psi}(u_{ii}) \pi_{L,\psi}(\omega)] f(w\omega^{-1})\\
&=\widetilde{c}_{X^{\ast}}(\omega^{-1}u_{ii}, \omega)^{-1}\psi(\tfrac{1}{4}x_i^{\ast})[ \pi_{L,\psi}(\omega)] f(w\omega^{-1}u_{ii}+\tfrac{1}{2} e_i^{\ast})\\
&=\widetilde{c}_{X^{\ast}}(\omega^{-1}u_{ii}, \omega)^{-1}\psi(\tfrac{1}{4}x_i^{\ast}) f(wg+\tfrac{1}{2} e_i)\\
&=m_{X^{\ast}}(g)^{-1}\psi(\tfrac{1}{4}x_i^{\ast}) f(wg+\tfrac{1}{2} e_i)\\
&=e^{-\tfrac{\pi i}{4}} \psi(\tfrac{1}{4}x_i^{\ast}) f(wg+\tfrac{1}{2} e_i).
 \end{split}
\end{equation*} 
Case 8: $g=u^-_{ij}=\omega^{-1} u_{ij} \omega\in \Sp_{2m}(\Z)$, for $i\neq j$.
\begin{equation*}\label{ui-}
\begin{split}
 &\pi_{L, \psi}(g)f( w)\\
 &=\pi_{L, \psi}(\omega^{-1} u_{ij} \omega)f( w)\\
 &=\widetilde{c}_{X^{\ast}}(\omega^{-1}, u_{ij})^{-1}\widetilde{c}_{X^{\ast}}(\omega^{-1}u_{ij}, \omega)^{-1}\pi_{L,\psi}(\omega^{-1})\pi_{L,\psi}(u_{ij}) \pi_{L,\psi}(\omega) f(w)\\
&=\widetilde{c}_{X^{\ast}}(\omega^{-1}u_{ij}, \omega)^{-1}\pi_{L,\psi}(\omega^{-1})\pi_{L,\psi}(u_{ij}) \pi_{L,\psi}(\omega) f(w)\\
&=\widetilde{c}_{X^{\ast}}(\omega^{-1}u_{ij}, \omega)^{-1} f(wg).
\end{split}
\end{equation*}
Case 9: $g=M_q^{(jk)}$. (cf. Def. \ref{Mqijk}) Let $g_1=\iota_{(j,k)}(u_{1111}^1)$ and $g_2=\iota_{(j,k)}(u_{1111}^2)$. Then $g=g_1g_2$.  Note:
\begin{itemize}
\item $g_1=\iota_{(j,k)}(u_{1111}^1)=u_{jj}(-1)\cdot u_{kk}(-1)$;
\item $g_2=\iota_{(j,k)}(u_{1111}^2)=u_{jj}^-(1)\cdot u_{kk}^-(1)\cdot u_{jk}^-(1)=\omega^{-1}u_{jj}(-1)u_{kk}(-1)u_{jk}(-1) \omega$. 
\end{itemize}
\begin{equation*}\label{Mqjk1}
\begin{split}
 &\pi_{L, \psi}(g_1)f( w)\\
 &=\pi_{L, \psi}(u_{jj}(-1)\cdot u_{kk}(-1))f( w)\\
 &=\psi(\tfrac{1}{4}x_j+\tfrac{1}{4}x_k)f(wg_1-\tfrac{1}{2} e_j^{\ast}-\tfrac{1}{2} e_k^{\ast}).
\end{split}
\end{equation*}
\begin{equation*}\label{Mqjk2}
\begin{split}
 &\pi_{L, \psi}(g_2)f( w)\\
 &=\pi_{L, \psi}(\omega^{-1}u_{jj}(-1)u_{kk}(-1)u_{jk}(-1) \omega)f( w)\\
 &=\widetilde{c}_{X^{\ast}}(\omega^{-1}, u_{jj}(-1)u_{kk}(-1)u_{jk}(-1))^{-1}\widetilde{c}_{X^{\ast}}(\omega^{-1}u_{jj}(-1)u_{kk}(-1)u_{jk}(-1), \omega)^{-1} \\
 &\cdot \pi_{L, \psi}(\omega^{-1}) \pi_{L, \psi}( u_{jj}(-1)u_{kk}(-1)u_{jk}(-1))\pi_{L, \psi}(\omega)f(w)\\
 &=\widetilde{c}_{X^{\ast}}(g_2\omega^{-1}, \omega)^{-1}\pi_{L, \psi}(\omega^{-1}) \pi_{L, \psi}( u_{jj}(-1)u_{kk}(-1)u_{jk}(-1))\pi_{L, \psi}(\omega)f(w)\\
 &=\widetilde{c}_{X^{\ast}}(g_2, \omega^{-1})^{-1} \pi_{L, \psi}( u_{jj}(-1)u_{kk}(-1)u_{jk}(-1))\pi_{L, \psi}(\omega)f(w\omega^{-1})\\
 &=\widetilde{c}_{X^{\ast}}(g_2, \omega)^{-1}\psi(\tfrac{1}{4}x_j^{\ast}) \pi_{L, \psi}(u_{kk}(-1)u_{jk}(-1))\pi_{L, \psi}(\omega)f(w\omega^{-1} u_{jj}(-1)+\tfrac{1}{2} e_j^{\ast})\\
 &=\widetilde{c}_{X^{\ast}}(g_2, \omega)^{-1}\psi(\tfrac{1}{4}x_j^{\ast}+\tfrac{1}{4}x_k^{\ast}) \pi_{L, \psi}(u_{jk}(-1))\pi_{L, \psi}(\omega)f(w\omega^{-1} u_{jj}(-1)u_{kk}(-1)+\tfrac{1}{2} e_j^{\ast}+\tfrac{1}{2} e_j^{\ast})\\
  &=\widetilde{c}_{X^{\ast}}(g_2, \omega)^{-1}\psi(\tfrac{1}{4}x_j^{\ast}+\tfrac{1}{4}x_k^{\ast})\pi_{L, \psi}(\omega)f(w\omega^{-1} u_{jj}(-1)u_{kk}(-1)u_{jk}(-1)+\tfrac{1}{2} e_j^{\ast}+\tfrac{1}{2} e_j^{\ast})\\
    &=\widetilde{c}_{X^{\ast}}(g_2, \omega)^{-1}\psi(\tfrac{1}{4}x_j^{\ast}+\tfrac{1}{4}x_k^{\ast})f(wg_2+\tfrac{1}{2} e_j+\tfrac{1}{2} e_k).
\end{split}
\end{equation*}
\begin{equation*}\label{Mqjk3}
\begin{split}
 &\pi_{L, \psi}(g)f( w)\\
 &=\pi_{L, \psi}(g_1g_2)f( w)\\
 &=\widetilde{c}_{X^{\ast}}(g_1, g_2)^{-1} \pi_{L, \psi}(g_1)\pi_{L, \psi}(g_2)f( w)\\
 &=\psi(\tfrac{1}{4}x_j+\tfrac{1}{4}x_k)\pi_{L, \psi}(g_2)f(wg_1-\tfrac{1}{2} e_j^{\ast}-\tfrac{1}{2} e_k^{\ast})\\
 &=\widetilde{c}_{X^{\ast}}(g_2, \omega)^{-1} \psi(\tfrac{1}{4}x_j+\tfrac{1}{4}x_k+\tfrac{1}{4}x_j^{\ast}+\tfrac{1}{4}x_k^{\ast})f(wg-\tfrac{1}{2} e_j^{\ast}-\tfrac{1}{2} e_k^{\ast}-\tfrac{1}{2} e_j-\tfrac{1}{2} e_k).
\end{split}
\end{equation*}
\begin{equation*}\label{Mqjk4}
\begin{split}
&\widetilde{c}_{X^{\ast}}(g_2, \omega)^{-1}\\
&=\widetilde{c}_{X^{\ast}}(g_1, g_2\omega)^{-1}\widetilde{c}_{X^{\ast}}(g_2, \omega)^{-1}\\
&=\widetilde{c}_{X^{\ast}}(g, \omega)^{-1}\\
&\stackrel{\textrm{ Lem. } \ref{calculation4}}{=}m_{X^{\ast}}(g)\\
&=e^{-\tfrac{\pi i}{4}}.
\end{split}
\end{equation*}

\subsubsection{} According to Lemma \ref{gamma12}(2), every element $g$ of $\Gamma_m(1,2)$ can be expressed as 
$g=\gamma p_1 \omega_{S_i} p_2$, for $\gamma \in \Gamma_m(2)$, $p_i\in P_{X^{\ast}}'(\Z)$.  By the above cases (1)---(5), we have:
\begin{equation}\label{chapeq8}
 \pi_{L, \psi}(g)f(w)= \epsilon_g f(wg).
 \end{equation}
for some $\epsilon_g\in \mu_8$.  Hence:
\begin{lemma}\label{gamma1SP}
The restriction of the cocycle $[\widetilde{c}_{X^{\ast}}]$ to the subgroup $\Gamma_{m}(1,2)$ is trivial, with an explicit trivialization
 $$\epsilon: \Gamma_m(1,2) \to \mu_8; g\longmapsto \epsilon_g,$$ 
 such that $\widetilde{c}_{X^{\ast}}(g_1, g_2) = \epsilon(g_1) \epsilon(g_2) \epsilon(g_1 g_2)^{-1}$, for  all $g_i \in \Gamma_m(1,2)$.
\end{lemma}

\section{Intertwining operators II}
This section  is intended to  present an explicit trivialization map for the group $\Gamma_{m}(1,2)$ in the above lemma \ref{gamma12}.  A trivialization map for all $m$ is given by the two authors in the classical book \cite{LiVe}. The explicit formula for $m=1$, which refers to the Gauss sum, is what it  calculates. For $m\geq 2$,  the formula pertains to  the symplectic Guass sum, as defined by  Styer in \cite{Sty1}, \cite{Sty2}.
\subsection{Lion-Vergne's result} Retaining the notations from Section \ref{heisenbergrepm}, let 
$Y$ and  $Z$ be two Lagrangian subspaces of $W$.  Correspondingly,  let  $(\pi_{Y,\psi}, \mathcal{H}(Y))$, 
$(\pi_{Z,\psi}, \mathcal{H}(Z))$ be the  Heisenberg representations of $\Ha(W)$ associated to $\psi$. Following  \cite{Pe}, we denote by $A_{YZ}$ the isomorphism from $Y/Y\cap Z$ to $[Z/(Y\cap Z)]^{\ast}$. Let $d\dot{y}$, $d\dot{z}$ represent the quotient measures on  $Y/(Y\cap Z)$ and $Z/(Z\cap Y)$ respectively.   We define an intertwining operator $\mathcal{F}_{ZY}$ as follows:
\[\begin{array}{rccl}
\mathcal{F}_{ZY}: & \mathcal{H}(Y)& \longrightarrow &\mathcal{H}(Z);\\
 &f &\longmapsto &  \mathcal{F}_{ZY}(f)(h)=\int_{Z/Z\cap Y} f([0, z] h) |A_{ZY}|^{1/2} d\dot{z}.
 \end{array}\]
 Associated to three  Lagrangian subspaces $X, Y, Z$, there exists the Leray invariant $L(X, Y, Z)$.(cf. \cite[p.55]{MoViWa})  We will denote the quadratic form associated by $q(X,Y,Z)$. Let $\gamma(q(X, Y, Z))$ denote  the corresponding  Weil index.
 \begin{theorem}
 $\gamma(q(X, Y, Z)) \cdot \mathcal{F}_{XZ} \circ \mathcal{F}_{ZY}\circ \mathcal{F}_{YX}=\id_{\mathcal{H}(X)}$.
 \end{theorem}
 \begin{proof}
 See \cite[Theorem 1.4]{Pe}.
 \end{proof}
 Now, consider  a Lagrangian subspace $X^{\ast}$ of $W$. For $g_1, g_2 \in \Sp(W)$, denote  $q_{X^{\ast}}(g_1,g_2)=q(X^{\ast}, X^{\ast} g_2^{-1}, X^{\ast} g_1)$. Then 
 \begin{align}
\widetilde{c}_{X^{\ast}}(g_1, g_2)=\gamma(\psi(q_{X^{\ast}}(g_1, g_2))).
\end{align} 
 \subsubsection{}  Retains the notations from Section \ref{schod}. 
  For $g_1=\begin{pmatrix} a & b\\ c& d\end{pmatrix}\in \Gamma_m(1,2)$,   let $Y^{\ast}= X^{\ast}g_1$ and $Y=Xg_1$. By \cite[Section 2.2.19]{LiVe}(analogously to Section \ref{interwop}),  there also exists  a pair of explicit isomorphisms between   $\mathcal{H}(Y^{\ast})\simeq L^2(Y)$ and  $\mathcal{H}(L)$, given as follows:
 \begin{equation}\label{inter1}
 \begin{split}
\theta_{L, Y^{\ast}}(f')( w)&=\sum_{l\in L/L\cap Y^{\ast}} f'( w+l)\psi(\tfrac{\langle l, w\rangle}{2}+\tfrac{\langle y_{l}, y^{\ast}_l\rangle}{2})\\
&=\sum_{l\in L\cap Y}f'(y+l)\psi(\langle l, w\rangle) \psi(\tfrac{\langle y, y^{\ast}\rangle}{2}),
    \end{split}
 \end{equation}
 \begin{equation}\label{inter2}
 \begin{split}
 \theta_{Y^{\ast}, L}(f)(y)&=  \int_{Y^{\ast}/Y^{\ast}\cap L} f([\dot{y}^{\ast},0]\cdot [y, 0]) d\dot{y}^{\ast}\\
  &= \int_{Y^{\ast}/Y^{\ast}\cap L} \psi(\tfrac{\langle\dot{y}^{\ast}, y\rangle}{2})f(y+\dot{y}^{\ast}) d\dot{y}^{\ast},
  \end{split}
      \end{equation}
    for $ w=y+ y^{\ast} \in W$, $l=y_l+ y_l^{\ast}\in L$, with $y,y_l\in Y$,  $y^{\ast},y^{\ast}_l\in Y^{\ast}$, and  $f'\in \mathcal{H}(Y^{\ast})$ with $f'|_Y \in S(Y)\subseteq L^2(Y)$,  $f\in L^1(W)\cap \mathcal{H}(L)$. 
 Let $g_2\in \Gamma_m(1,2)$ and $Z^{\ast}=X^{\ast} g_2$. Following \cite[p.149]{LiVe}, let  $\beta(Y^{\ast}, Z^{\ast})$ be an element in $T$ such that
     $$ \theta_{Y^{\ast}, L}  \circ  \theta_{ L, Z^{\ast}}=\beta(Y^{\ast}, Z^{\ast}) \mathcal{F}_{Y^{\ast}Z^{\ast}}.$$
     \begin{lemma}
$\beta(X^{\ast}, Z^{\ast})\beta(Z^{\ast},Y^{\ast})\beta(Y^{\ast}, X^{\ast})=\widetilde{c}_{X^{\ast}}(g_2g_1^{-1}, g_1)^{-1}$.
\end{lemma}
\begin{proof}
 \begin{align*}
 &1_{\mathcal{H}(X^{\ast})}\\
&=\theta_{X^{\ast}, L}  \circ  \theta_{ L, Z^{\ast}} \circ \theta_{Z^{\ast}, L}  \circ  \theta_{ L, Y^{\ast}}\circ \theta_{Y^{\ast}, L}  \circ  \theta_{ L, X^{\ast}}\\
&=\beta(X^{\ast}, Z^{\ast}) \beta(Z^{\ast}, Y^{\ast}) \beta(X^{\ast}, Z^{\ast})  \mathcal{F}_{X^{\ast}Z^{\ast}} \circ \mathcal{F}_{Z^{\ast}Y^{\ast}} \circ \mathcal{F}_{Y^{\ast}X^{\ast}}\\
&=\beta(X^{\ast}, Z^{\ast}) \beta(Z^{\ast}, Y^{\ast}) \beta(X^{\ast}, Z^{\ast}) \gamma(\psi(q(X^{\ast}, Y^{\ast}, Z^{\ast})) )^{-1} 1_{\mathcal{H}(X^{\ast})}.
\end{align*} 
Hence 
\begin{align*}
&\beta(X^{\ast}, Z^{\ast}) \beta(Z^{\ast}, Y^{\ast}) \beta(X^{\ast}, Z^{\ast})\\
&= \gamma(\psi(q(X^{\ast}, Y^{\ast}, Z^{\ast})) )= \gamma(\psi(q(X^{\ast}, X^{\ast}g_1, X^{\ast} g_2)) )\\
&=\gamma(\psi(-q( X^{\ast}, X^{\ast}g_1^{-1}, X^{\ast}g_2g_1^{-1})) )\\
&=\widetilde{c}_{X^{\ast}}(g_2g_1^{-1}, g_1)^{-1}.
\end{align*} 
\end{proof}
For any $g\in \Gamma_m(1,2)$, let us define $$\widetilde{\beta}(g)=\beta(X^{\ast}g, X^{\ast}).$$
\begin{lemma}\label{widetildec}
$\widetilde{c}_{X^{\ast}}(h_1,h_2)=\widetilde{\beta}(h_1)^{-1}\widetilde{\beta}(h_2)^{-1}\widetilde{\beta}(h_1h_2)$,  for $h_i=\begin{pmatrix} a_i& b_i\\c_i& d_i\end{pmatrix}\in \Gamma_m(1,2)$.
\end{lemma}
\begin{proof}
 Write $h_1=g_2g_1^{-1}$, $h_2=g_1$, $h_1h_2=g_2$. Then: $\widetilde{\beta}(g_2)^{-1}=\beta(X^{\ast}, Z^{\ast})=\widetilde{\beta}([g_2g_1^{-1}]g_1)^{-1}$,
$\beta(Z^{\ast}, Y^{\ast})=\beta(X^{\ast}g_2 g_1^{-1}, X^{\ast})=\widetilde{\beta}(g_2g_1^{-1})$, $\beta(Y^{\ast}, X^{\ast})=\widetilde{\beta}(g_1)$. Hence $\widetilde{\beta}(g_2)^{-1}\widetilde{\beta}(g_2 g_1^{-1})\widetilde{\beta}(g_1)=\widetilde{c}_{X^{\ast}}(g_2g_1^{-1},g_1)^{-1}$. 
\end{proof}   
\subsection{The calculation of $\widetilde{\beta}(g)$ } 
\subsubsection{$\det c\neq 0$} Let $g=\begin{pmatrix} a & b\\ c& d\end{pmatrix}$ with $\det c\neq 0$. Write $Y^{\ast}=X^{\ast}g $. Then $Y^{\ast}$  and $X^{\ast}$ are two transverse Lagrangian subspaces, and  $W=X^{\ast}\oplus Y^{\ast}$. Note that $X^{\ast}\cap L= \sum_{i=1}^m \Z e_i^{\ast}$ and $Y^{\ast}\cap L= \sum_{i=1}^m \Z [e_i^{\ast} c+ e_i^{\ast} d]$.  Then 
$$L/[X^{\ast}\cap L+Y^{\ast}\cap L] \simeq [\sum_{i=1}^m \Z e_i] /[\sum_{i=1}^m \Z (e_i^{\ast})c].$$
The number of elements of  $L/[X^{\ast}\cap L+Y^{\ast}\cap L] $ is just $|\det c|$. Following \cite[Section 2.2.5, Prop.]{LiVe}, 
$$\widetilde{\beta}(g)=|\det c|^{-1/2}\sum_{l\in L/[(X^{\ast} \cap L)\oplus (X^{\ast}g \cap L)]} e^{-\pi i \langle x_l, x_l^{\ast}\rangle} e^{-\pi i \langle p_1(l), p_2(l)\rangle},$$
where $p_1:W\to X^{\ast}$ and $p_2:W \to  Y^{\ast}$, and $l=x_l+x_{l}^{\ast}\in L=(L\cap X)\oplus (L\cap X^{\ast})$.

 For $w=\sum^m_{i=1}x_i e_i +\sum_{j=1}^m x_j^{\ast} e_j^{\ast}\in W$, let us write $ x=( x_1, \cdots,  x_m)$ and $ x^{\ast}=( x_1^{\ast}, \cdots,  x_m^{\ast})$, and $ \underline{e}=(e_1, \cdots,  e_m)$, $ \underline{e}^{\ast}=(e_1^{\ast},\cdots,  e_m^{\ast})$. For $g\in \Sp(W)$, let us write $( \underline{e},  \underline{e}^{\ast})^T g=A_g( \underline{e}, \underline{e}^{\ast})^T $, for some $A_g\in \Sp_{2m}(\R)$. Identity $g$ with $A_g$, and  write also $A_g=\begin{pmatrix} a & b\\ c& d\end{pmatrix}$.
For $ x  \underline{e}^T \in L$, write 
$$ x  \underline{e}^T= xc^{-1}(c  \underline{e}^T +d   \underline{e}^{\ast T}) - xc^{-1}d    \underline{e}^{\ast T}.$$

\begin{align*}
&\widetilde{\beta}(g)=|\det c|^{-1/2}\sum_{ x  \underline{e}^T \in (\Z^m   \underline{e}^T) /(\Z^m c  \underline{e}^T)}  e^{-\pi i \langle - xc^{-1}d   \underline{e}^{\ast T}, xc^{-1}(c  \underline{e}^T +d   \underline{e}^{\ast T})\rangle}\\
&\quad\quad=|\det c|^{-1/2}\sum_{ x  \underline{e}^T \in (\Z^m   \underline{e}^T) /(\Z^m c  \underline{e}^T)}  e^{-\pi i  xc^{-1}d x^T}\\
&\quad\quad=|\det c|^{-1/2} \overline{ G(d,c)}.
\end{align*} 
 
\subsubsection{$\det c=0$}  Set  $X^{\ast} \cap Y^{\ast} =M$ and $L^M= L\cap M^{\bot} +M$. In this case, we   consider the non-degenerate symplectic vector space $W'=M^{\bot}/M$ and $L'=L^M/M $. According to \cite[2.2.6, Lemma]{LiVe},  $L'$ is a self-dual lattice of $W'$. Let $X^{\ast'}=X^{\ast}/M$, $Y^{\ast'}=Y^{\ast} /M$. Then  $Y^{\ast'}$  and $X^{\ast'}$ are two transverse Lagrangian subspaces of $W'$, and  $W'=X^{\ast'}\oplus Y^{\ast'}$. Let $|\det c'|=|L'/[X^{\ast'}\cap L'+Y^{\ast'}\cap L'] |$. Then:
    $$\widetilde{\beta}(g)=|\det c'|^{-1/2}\sum_{l\in L'/[(X^{\ast'} \cap L')\oplus (Y^{\ast'}\cap L')]} e^{-\pi i \langle x_{l}, x_{l}^{\ast}\rangle} e^{-\pi i \langle p_1(l), p_2(l)\rangle},$$
where $p_1:W'\to X^{\ast'}$ and $p_2:W '\to  Y^{\ast'}$, and $l=x_l+x_{l}^{\ast}\in L' \simeq L\cap M^{\bot}/L\cap M$. Moreover, according to \cite[2.2.8, Prop.]{LiVe},
 $$\widetilde{\beta}(g)=|\det c'|^{-1/2}\sum_{l=x^{\ast}+y^{\ast}\in [L\cap (X^{\ast}+Y^{\ast})]/[(X^{\ast} \cap L)+(Y^{\ast}\cap L)]} e^{-\pi i \langle x_{l}, x_{l}^{\ast}\rangle} e^{-\pi i \langle x^{\ast}, y^{\ast}\rangle},$$
 where $|\det c'|=|L'/[X^{\ast'}\cap L'+Y^{\ast'}\cap L'] |=|[L\cap (X^{\ast}+Y^{\ast})]/[(X^{\ast} \cap L)+(Y^{\ast}\cap L)]|$. 
 \begin{example}\label{exampleP}
 If $g\in P_{X^{\ast}}(\Z)\cap \Gamma_m(1,2)$ or $g=\omega_{S_i}$, then $\widetilde{\beta}(g)=1$.
 \end{example}
 Recall $\iota=(\iota_1, \cdots, \iota_m): \SL_2(\Z) \times \cdots \times  \SL_2(\Z) \longrightarrow \Sp_{2m}(\Z)$.
 \begin{example}
 If $g\in \iota_i(g_i)$, for $g_i=\begin{pmatrix} a & b\\ c& d\end{pmatrix}\in \Gamma_1(1,2)\subseteq \SL_2(\Z)$, with $c\neq 0$,  then $\widetilde{\beta}(g)=|\det c|^{-1/2} \overline{ G(d,c)}=\widetilde{\beta}(g_i)$.
 \end{example}
 \begin{proof}
 $Y^{\ast}=(\sum_{j=1,j\neq i}^m\R e_j^{\ast})\oplus \R (e_i^{\ast}c+e_i^{\ast}d)$,  $X^{\ast} \cap Y^{\ast} =M=(\sum_{j=1,j\neq i}^m\R e_j^{\ast})$, $W'=\R e_i\oplus \R e_i^{\ast}$, $L'=\Z e_i \oplus \Z e_i^{\ast}$, $X^{\ast'}=\R e_i^{\ast}$, $ Y^{\ast'}=\R (e_i^{\ast}c+e_i^{\ast}d)$, $|\det c'|=|L'/[X^{\ast'}\cap L'+Y^{\ast'}\cap L'] |=|(\Z e_i +\Z e_i^{\ast})/(\Z e_i^{\ast}+ \Z e_i^{\ast} c)|=|c|$. For $ x=x_ie_i \in L'$, write 
$$ x= x_ic^{-1}(ce_i  +d  e_i^{\ast}) -x_ic^{-1}d   e_i^{\ast}.$$
\begin{align*}
\widetilde{\beta}(g)&=|\det c|^{-1/2}\sum_{x_i\in \Z/c\Z}  e^{-\pi i \langle(-x_ic^{-1}d  )  e_i^{\ast}, x_ic^{-1}(ce_i  +d  e_i^{\ast}))\rangle}\\
&=|\det c|^{-1/2}\sum_{x_i\in\Z/c\Z}  e^{-\pi i  x_ic^{-1}d x_i}\\
&=|\det c|^{-1/2} \overline{ G(d,c)}.
\end{align*}
 \end{proof}
 \begin{example}\label{exampleineqj}
 If $g=u_{ij}^-(2t)$, for $i\neq j$, $t\in \Z$,    then $\widetilde{\beta}(g)=1$.
 \end{example}
 \begin{proof}
 $Y^{\ast}=(\sum_{k=1,k\neq i, j}^m\R e_k^{\ast})\oplus \R (e_j(-2t)+e_i^{\ast})\oplus \R (e_i(-2t)+e_j^{\ast}) $,  $X^{\ast} \cap Y^{\ast} =M=\sum_{k=1,k\neq i, j}^m\R e_k^{\ast}$, $W'=\R e_i\oplus \R e_i^{\ast}\oplus \R e_j\oplus \R e_j^{\ast}$, $L'=\Z e_i \oplus \Z e_i^{\ast}\oplus \Z e_j \oplus \Z e_j^{\ast}$, $X^{\ast'}=\R e_i^{\ast}\oplus \R e_j^{\ast}$, $ Y^{\ast'}=\R (e_j(-2t)+e_i^{\ast})\oplus \R (e_i(-2t)+e_j^{\ast})$, $|\det c'|=|L'/[X^{\ast'}\cap L'+Y^{\ast'}\cap L'] |=|(\Z  e_i +\Z e_j)/(2t\Z e_i+ 2t\Z e_j )|=4|t|^2$. For $ x=x_ie_i+x_j e_j \in L'$, write 
$$ x= x_j(-2t)^{-1}(e_j(-2t)+e_i^{\ast}) -x_j(-2t)^{-1}  e_i^{\ast}+ x_i(-2t)^{-1}(e_i(- 2t)+e_j^{\ast}) -x_i(-2t)^{-1}  e_j^{\ast}.$$
\begin{align*}
\widetilde{\beta}(g)&=|\det c|^{-1/2}\sum_{x_i\in\Z/2t\Z, x_j\in\Z/2t\Z}   e^{-\pi i \langle x_j(2t)^{-1}e_i^{\ast} +x_i(2t)^{-1}  e_j^{\ast}, x_j(-2t)^{-1}(e_j(-2t)+e_i^{\ast})+ x_i(-2t)^{-1}(e_i (-2t)+e_j^{\ast}) \rangle}\\
&=|4t^2|^{-1/2}\sum_{x_i\in\Z/2t\Z, x_j\in\Z/2t\Z}  e^{\pi i \tfrac{1}{t} x_i x_j}\\
&=1.
\end{align*}
 \end{proof}
 Recall the function  $\epsilon$ from Lemma  \ref{gamma1SP}. 
 \begin{lemma}
 $\widetilde{\beta}(g)\epsilon_g=1$, for all $g\in \Gamma_m(1,2)$.
 \end{lemma}
 \begin{proof}
 For any $g_1,g_2\in \Gamma_m(1,2)$, we have the following relationship:
  \begin{equation}\label{trivgamma}
 \widetilde{c}_{X^{\ast}}(g_1,g_2)=\widetilde{\beta}(g_1)^{-1}\widetilde{\beta}(g_2)^{-1}\widetilde{\beta}(g_1g_2)=\epsilon_{g_1}\epsilon_{g_1}\epsilon_{g_1g_2}^{-1}.
 \end{equation} Thus, 
the map  $g \to \widetilde{\beta}(g)\epsilon_g$ defines  a character of $\Gamma_m(1,2)$. Since $\Gamma_m(1,2)$ is  generated by $ P_{X^{\ast}}(\Z)\cap \Gamma_m(1,2)$ and these $\omega_{S_i}$'s,  the result follows.
 \end{proof}

\subsection{Symplectic Gauss Sum}
Let us first recall some results of the symplectic Gauss sum from \cite{Sty1}, \cite{Sty2}. Let  $g= \begin{pmatrix} a & b\\ c& d\end{pmatrix}\in \Gamma_m(1,2)$, with $\det c\neq 0$.  We say two integer vectors $l,l'\in \Z^m$ are congruent modulo $c^T$ iff $l-l' =c^T l_0$, for some $l_0\in \Z^m$. Following \cite{Sty1}, we can define the symplectic Gauss sum:
$$G(d,c)=\sum_{l \bmod c^T}e^{\pi i [lc^{-1}d l^T]}.$$

\section{Intertwining operators III}
 Retain the notations from Examples \ref{ex1} and \ref{ex4}. By \cite[Section 3]{Ta1} or \cite[p.401]{Sa1}, there exists  a pair of  unitary equivalence  between   $\mathcal{H}(X^{\ast})\simeq L^2(X)$ and  $\mathcal{H}_{F_{\widehat{z}}}$. Here, we provide two explicit unitary equivalence maps derived from their papers. For completeness, we will also present the proofs. Identity $W=X\oplus X^{\ast}$ with $\R^m\oplus \R^m$ and denote $q_{\widehat{z}}(v)=e^{-\pi i(v\widehat{z}^{-1}v^T)}$.  
\begin{itemize}
\item $\theta_{ F_{\widehat{z}}, X^{\ast}}: \mathcal{H}(X^{\ast})\simeq L^2(X) \longrightarrow \mathcal{H}_{F_{\widehat{z}}}$.
  \begin{equation}\label{inter3}
 \begin{split}
\theta_{F_{\widehat{z}},X^{\ast}}(f)( v)& =\det(2\Im(z))^{\tfrac{1}{4}} \int_{\R^m} q_{\widehat{z}}(v+x\widehat{z})^{-1}q_{\widehat{z}}(v)f(x)dx\\
&=\det(2\Im(z))^{\tfrac{1}{4}} \int_{\R^m} e^{2\pi i\big( \tfrac{1}{2} x\widehat{z} x^T +xv^T\big)}f(x)dx.
    \end{split}
 \end{equation}
 \item  $\theta_{  X^{\ast},F_{\widehat{z}}}: \mathcal{H}_{F_{\widehat{z}}} \longrightarrow \mathcal{H}(X^{\ast})\simeq L^2(X) $.
 \begin{equation}\label{inter4}
 \begin{split}
\theta_{  X^{\ast},F_{\widehat{z}}}(f')(x)&=\det(2\Im(z))^{\tfrac{1}{4}} \int_{\C^m} \overline{q_{\widehat{z}}(v+x\widehat{z})}(\overline{ q_{\widehat{z}}(v)})^{-1} k_{\widehat{z}}(v,v)f'(v)d_{\widehat{z}}(v)\\
&=  \int_{\C^m} \det(2\Im(z))^{\tfrac{1}{4}} e^{2\pi i\big[\tfrac{1}{2}xzx^T-x\overline{v}^T \big]} k_{\widehat{z}}(v,v)f'(v)d_{\widehat{z}}(v).
  \end{split}
      \end{equation}
      \end{itemize}

 Let us verify that $ \theta_{  X^{\ast},F_{\widehat{z}}}$, $\theta_{  X^{\ast},F_{\widehat{z}}}$ both are $\Ha(W)$-intertwining  operators. Let $h=(y,y^{\ast};t)\in \Ha(W)$.\\
 1) 
 \begin{align*}
 &\theta_{ F_{\widehat{z}}, X^{\ast}} [\pi_{X^{\ast},\psi }(h) f](v)\\
 &=\det(2\Im(z))^{\tfrac{1}{4}} \int_{\R^m} e^{2\pi i\big( \tfrac{1}{2} x\widehat{z} x^T +xv^T\big)}[\pi_{X^{\ast},\psi }(h) f](x)dx\\
 &=\det(2\Im(z))^{\tfrac{1}{4}} \int_{\R^m} e^{2\pi i\big( \tfrac{1}{2} x\widehat{z} x^T +xv^T\big)}\psi(t+\langle x,y^{\ast}\rangle+\tfrac{1}{2}\langle y, y^{\ast}\rangle) f(x+y)dx\\
  &=\det(2\Im(z))^{\tfrac{1}{4}} \int_{\R^m} e^{2\pi i\big( \tfrac{1}{2} (x-y)\widehat{z} (x-y)^T +(x-y)v^T\big)}\psi(t+\langle x-y,y^{\ast}\rangle+\tfrac{1}{2}\langle y, y^{\ast}\rangle) f(x)dx\\
  &=\det(2\Im(z))^{\tfrac{1}{4}}\psi(t-\tfrac{1}{2}\langle y, y^{\ast}\rangle) \int_{\R^m} \psi\big( \tfrac{1}{2} (x-y)\widehat{z} (x-y)^T +(x-y)v^T\big)\psi(\langle x,y^{\ast}\rangle) f(x)dx\\
  &=\det(2\Im(z))^{\tfrac{1}{4}}\psi(t-\tfrac{1}{2} y (y^{\ast})^T+\tfrac{1}{2} y\widehat{z} y^T -yv^T) \int_{\R^m}\psi\big( \tfrac{1}{2} x\widehat{z}x^T -x\widehat{z}y^T  +xv^T+x(y^{\ast})^T\big) f(x)dx;
 \end{align*}
 \begin{align}
 &\pi_{F_{\widehat{z}}, \psi}(h) [\theta_{ F_{\widehat{z}}, X^{\ast}} f](v)\\
 &=\psi(t)e^{-2\pi i[\frac{1}{2}(y\overline{z}+y^{*})y^T+vy^T]}[\theta_{ F_{\widehat{z}}, X^{\ast}} f](v+y\overline{z}+y^{\ast})\\
&=\psi(t)e^{-2\pi i[\frac{1}{2}(y\overline{z}+y^{*})y^T+vy^T]}\det(2\Im(z))^{\tfrac{1}{4}} \int_{\R^m} e^{2\pi i\big( \tfrac{1}{2} x\widehat{z} x^T +x(v+y\overline{z}+y^{\ast})^T\big)} f(x) dx\\
&= \theta_{ F_{\widehat{z}}, X^{\ast}} [\pi_{ X^{\ast},\psi}(h) f](v).
  \end{align}
2)
\begin{align*}
& \kappa_{\widehat{z}}(v-y\overline{z}-y^{\ast},v-y\overline{z}-y^{\ast})\\
& \stackrel{v=v_x+iv_y}{=}e^{-2\pi (v_y+yQ_z)Q_z^{-1}(v_y+yQ_z)^T}\\
&=e^{-2\pi v_yQ_z^{-1}v_y^T}e^{-4\pi v_yy^T}e^{-2\pi yQ_zy^T}\\
&=\kappa_{\widehat{z}}(v,v)e^{2\pi i(-\overline{v} y^T +v y^T)} e^{2\pi i(\tfrac{1}{2}yzy^T-\tfrac{1}{2}y\overline{z}y^T)};
\end{align*}
 \begin{align*}
&\theta_{  X^{\ast},F_{\widehat{z}}}([\pi_{F_{\widehat{z}}, \psi}(h)f'])(x)\\
&= \int_{\C^m} \det(2\Im(z))^{\tfrac{1}{4}} e^{2\pi i\big[\tfrac{1}{2}\langle x,x z\rangle-\langle x,\overline{v}\rangle \big]} \kappa_{\widehat{z}}(v,v)[\pi_{F_{\widehat{z}}, \psi}(h)f'](v)d_{\widehat{z}}(v)\\
&= \int_{\C^m} \det(2\Im(z))^{\tfrac{1}{4}} e^{2\pi i\big[\tfrac{1}{2}\langle x,x z\rangle-\langle x,\overline{v}\rangle \big]} \kappa_{\widehat{z}}(v,v)\psi(t)e^{-2\pi i[\frac{1}{2}(y\overline{z}+y^{*})y^T+vy^T]}f'(v+y\overline{z}+y^{\ast})d_{\widehat{z}}(v)\\
&=  \det(2\Im(z))^{\tfrac{1}{4}} \psi(t) \int_{\C^m}e^{2\pi i\big[\tfrac{1}{2}\langle x,x z\rangle-\langle x,\overline{v}-yz-y^{\ast}\rangle \big]} \kappa_{\widehat{z}}(v-y\overline{z}-y^{\ast},v-y\overline{z}-y^{\ast})e^{-2\pi i[\frac{1}{2}(y\overline{z}+y^{*})y^T+(v-y\overline{z}-y^{\ast})y^T]}f'(v)d_{\widehat{z}}(v)\\
&=  \det(2\Im(z))^{\tfrac{1}{4}} \psi(t) \int_{\C^m}e^{2\pi i\big[\tfrac{1}{2}\langle x,x z\rangle-\langle x,\overline{v}-yz-y^{\ast}\rangle \big]} \kappa_{\widehat{z}}(v-y\overline{z}-y^{\ast},v-y\overline{z}-y^{\ast})e^{-2\pi i[-\frac{1}{2}(y\overline{z}+y^{*})y^T+vy^T]}f'(v)d_{\widehat{z}}(v)\\
&=  \det(2\Im(z))^{\tfrac{1}{4}} \psi(t) \int_{\C^m}e^{2\pi i\big[\tfrac{1}{2}\langle x,x z\rangle-\langle x,\overline{v}-yz-y^{\ast}\rangle \big]} \kappa_{\widehat{z}}(v,v)e^{2\pi i[\frac{1}{2}(yz+y^{*})y^T-\overline{v}y^T]}f'(v)d_{\widehat{z}}(v);
\end{align*}
\begin{align*}
&\pi_{X^{\ast},\psi}(h) [\theta_{  X^{\ast},F_{\widehat{z}}}(f')](x)\\
&=\psi(t+\langle x,y^{\ast}\rangle+\tfrac{1}{2}\langle y, y^{\ast}\rangle) [\theta_{  X^{\ast},F_{\widehat{z}}}(f')](x+y)\\
&=\psi(t+\langle x,y^{\ast}\rangle+\tfrac{1}{2}\langle y, y^{\ast}\rangle) \int_{\C^m} \det(2\Im(z))^{\tfrac{1}{4}} e^{2\pi i\big[\tfrac{1}{2}\langle x+y,(x+y) z\rangle-\langle x+y,\overline{v}\rangle \big]} \kappa_{\widehat{z}}(v,v)f'(v)d_{\widehat{z}}(v)\\
&=\psi(t)\det(2\Im(z))^{\tfrac{1}{4}}  \int_{\C^m} \psi\big( \tfrac{1}{2}xzx^T+ x(yz+y^{\ast}-\overline{v})^T\big) e^{2\pi i\big(\tfrac{1}{2} y (y^{\ast})^T +\tfrac{1}{2}yzy^T-y\overline{v}^T \big)} \kappa_{\widehat{z}}(v,v)f'(v)d_{\widehat{z}}(v)\\
&=\theta_{  X^{\ast},F_{\widehat{z}}}([\pi_{F_{\widehat{z}}, \psi}(h)f'])(x).
\end{align*}
Let us verify that $\theta_{  X^{\ast},F_{\widehat{z}}}$ and  $\theta_{  F_{\widehat{z}},X^{\ast}}$ are indeed  inverses of  one another. It suffices to show the following:
\begin{align*}
&\theta_{  X^{\ast},F_{\widehat{z}}}\Big[\theta_{F_{\widehat{z}},X^{\ast}}(f)\Big](x)\\
&=  \int_{\C^m}\int_{\R^m}  \det(2\Im(z))^{\tfrac{1}{2}} e^{2\pi i\big[\tfrac{1}{2}\langle x,x z\rangle-\langle x,\overline{v}\rangle \big]} \kappa_{\widehat{z}}(v,v) e^{2\pi i\big( \tfrac{1}{2} y\widehat{z} y^T +yv^T\big)}f(y)dy d_{\widehat{z}}(v)\\
&=  \int_{\R^m}\det(2\Im(z))^{\tfrac{1}{2}} e^{2\pi i\big( \tfrac{1}{2} xz x^T\big)} e^{2\pi i\big( \tfrac{1}{2} y\widehat{z} y^T\big)}f(y) \bigg[ \int_{\C^m}  e^{2\pi i\big[-\langle x,\overline{v}\rangle \big]} \kappa_{\widehat{z}}(v,v) e^{2\pi i\big( yv^T\big)}d_{\widehat{z}}(v) \bigg] dy\\
&\stackrel{v=v_x+iv_y}{=}  \int_{\R^m}\det(2\Im(z))^{\tfrac{1}{2}} e^{2\pi i\big( \tfrac{1}{2} xz x^T\big)} e^{2\pi i\big( \tfrac{1}{2} y\widehat{z} y^T\big)}f(y) \bigg[ \int_{\C^m}  e^{-2\pi v_yQ_z^{-1}v_y^T} e^{-2\pi \big( (x+y) v_y^T\big)}e^{2\pi i \big( (-x+y) v_x^T\big)}d_{\widehat{z}}(v) \bigg] dy\\
&=\int_{\R^m}\det(2\Im(z))^{\tfrac{1}{2}} e^{2\pi i\big( \tfrac{1}{2} xP_z x^T- \tfrac{1}{2} yP_zy^T\big)} e^{\pi (x Q_z y^T- \tfrac{1}{2}xQ_zx^T- \tfrac{1}{2}yQ_zy^T)}f(y) dy\\
&\quad\quad \cdot \bigg[ \int_{\C^m}  e^{-2\pi(\tfrac{1}{2}xQ_z+\tfrac{1}{2}yQ_z+ v_y)Q_z^{-1}(\tfrac{1}{2}xQ_z+\tfrac{1}{2}yQ_z+v_y)^T} e^{2\pi i \big( (-x+y) v_x^T\big)}d_{\widehat{z}}(v) \bigg] \\
&=\int_{\R^m} e^{2\pi i\big( \tfrac{1}{2} xP_z x^T- \tfrac{1}{2} yP_zy^T\big)} e^{\pi (x Q_z y^T- \tfrac{1}{2}xQ_zx^T- \tfrac{1}{2}yQ_zy^T)}f(y) dy\bigg[ e^{2\pi i \big( (-x+y) v_x^T\big)}dv_x\bigg] \\
&=f(x).
\end{align*}
In the rest of this section, we assume that $z=\widehat{z_0}=z_0=i1_m$.  Retain the notations from Remark \ref{uneqi}. Recall the Bargmann transformation  from the classical book \cite[p.40 and p.45]{Fo}:
$$\begin{array}{lrcl}
 \mathcal{B}:& L^2(\R^m)  &\longrightarrow & \mathcal{H}_F; \\
 &f& \longmapsto & \mathcal{B}(f)(v)=2^{\tfrac{m}{4}} \int_{\R^m} f(x) e^{-\pi xx^T-\tfrac{\pi}{2} vv^T+2\pi xv^T} dx.
\end{array}$$
$$\begin{array}{lrcl}
 \mathcal{B}^{-1}:& \mathcal{H}_F  &\longrightarrow & L^2(\R^m); \\
 &\phi& \longmapsto & \mathcal{B}^{-1}(\phi)(x)=2^{\tfrac{m}{4}} \int_{\C^m} \phi(v) e^{-\pi xx^T-\tfrac{\pi}{2} \overline{v}(\overline{v})^T+2\pi x\overline{v}^T}  e^{-\pi \mid v\mid^2} dv.
\end{array}$$
\begin{lemma}
$\mathcal{A}\circ \theta_{ F_{z_0}, X^{\ast}}=\mathcal{B}$ and $\theta_{  X^{\ast}, F_{z_0}} \circ \mathcal{A}^{-1}=\mathcal{B}^{-1}$.
\end{lemma}
\begin{proof}
1) \begin{align*}
&\mathcal{A}\circ \theta_{ F_{z_0}, X^{\ast}}(f)(v)\\
&=e^{-\tfrac{\pi}{2}v v^T} \theta_{ F_{z_0}, X^{\ast}}(f)(-vi)\\
&=2^{\tfrac{m}{4}} \int_{\R^m}e^{-\tfrac{\pi}{2}v v^T} e^{2\pi i\big( \tfrac{1}{2}i xx^T -ixv^T\big)}f(x)dx\\
&=2^{\tfrac{m}{4}} \int_{\R^m}e^{-\tfrac{\pi}{2}v v^T-\pi  xx^T +2\pi xv^T}f(x)dx.
\end{align*}
2) 
\[ \mathcal{A}^{-1}(\phi)(v)=\phi(vi) e^{-\tfrac{\pi}{2}v v^T};\]
\begin{align*}
& \theta_{  X^{\ast}, F_{z_0}} \circ \mathcal{A}^{-1}(\phi)(x)\\
&=\int_{\C^m} \det(2\Im(z))^{\tfrac{1}{4}} e^{2\pi i\big[\tfrac{1}{2}xzx^T-x\overline{v}^T \big]} k_{z_0}(v,v)\mathcal{A}^{-1}(\phi)(v)d_{z_0}(v)\\
&=2^{\tfrac{m}{4}} \int_{\C^m}  e^{2\pi i\big[\tfrac{1}{2}ixx^T-x\overline{v}^T \big]} e^{-2\pi v_y v_y^T} \phi(vi) e^{-\tfrac{\pi}{2}v v^T}d_{z_0}(v)\\
&=2^{\tfrac{m}{4}} \int_{\C^m} e^{-\pi xx^T} e^{2\pi i\big[-x\overline{v}^T-2\pi v_y v_y^T-\tfrac{\pi}{2}v v^T\big]}  \phi(vi) d_{z_0}(v)\\
&\stackrel{v'=vi}{=}2^{\tfrac{m}{4}} \int_{\C^m} e^{-\pi xx^T} e^{2\pi i\big[-ix\overline{v'}^T\big]}e^{-2\pi (-v'_x) (-v'_x)^T+\tfrac{\pi}{2}v' v'^T}  \phi(v')d_{z_0}(v')\\
&=2^{\tfrac{m}{4}} \int_{\C^m} e^{-\pi xx^T}  e^{2\pi x\overline{v'}^T}  e^{-\tfrac{\pi}{2} \overline{v'}{\overline{v'}}^T-\pi \mid v'\mid^2} \phi(v') d_{z_0}(v').
\end{align*}
\end{proof}
\begin{proposition}\label{Propequ}
$\theta_{ F_{z_0}, X^{\ast}}\big(\pi_{X^{\ast}, \psi}(g)f\big)=\pi_{ F_{z_0}, \psi}(g)\big(\theta_{ F_{z_0}, X^{\ast}}(f)\big)$, for $g\in \Sp_{2m}(\R)$.
\end{proposition}
\subsection{ Proof of Proposition \ref{Propequ}}
See \cite[Theorem 4.2]{Ta1}.  Let $f=\big(\theta_{  X^{\ast},F_{\widehat{z}}}(f')\big)$, for $f' \in S(\R^m)\subseteq L^2(\R^m)$.\\
1) Let  $g=u(b)$.
\begin{align*}
&\theta_{ F_{z_0}, X^{\ast}}\big(\pi_{X^{\ast}, \psi}(g)(f)\big)(v)\\
&=\det(2 \cdot 1_m)^{\tfrac{1}{4}} \int_{\R^m} e^{2\pi i\big( \tfrac{1}{2} x i x^T +xv^T\big)}\pi_{X^{\ast}, \psi}(g)f(x)dx\\
&=2^{\tfrac{m}{4}} \int_{\R^m} e^{2\pi i\big( \tfrac{1}{2} x i x^T +xv^T\big)}\psi(\tfrac{1}{2}\langle x,xb\rangle) f(x)dx\\
&=2^{\tfrac{m}{4}} \int_{\R^m} e^{2\pi i\big( \tfrac{1}{2} x i x^T +xv^T+\tfrac{1}{2}xbx^T\big)}f(x)dx\\
&=2^{\tfrac{m}{4}}\det(2\Im(z_0))^{\tfrac{1}{4}}  \int_{\R^m} \int_{\C^m}  e^{2\pi i\big( \tfrac{1}{2} x i x^T +xv^T+\tfrac{1}{2}xbx^T\big)}e^{2\pi i\big[\tfrac{1}{2}xz_0x^T-x\overline{v}^{'T} \big]} k_{\widehat{z_0}}(v',v')f'(v')d_{\widehat{z_0}}(v')dx\\
&=2^{\tfrac{m}{4}}\det(2\Im(z_0))^{\tfrac{1}{4}}  \int_{\C^m}  \Big[\int_{\R^m} e^{2\pi i\big( \tfrac{1}{2} x i x^T +xv^T+\tfrac{1}{2}xbx^T\big)}e^{2\pi i\big[\tfrac{1}{2}xz_0x^T-x\overline{v}^{'T}  \big]} dx\Big]k_{z_0}(v',v')f'(v')d_{z_0}(v')\\
&=2^{\tfrac{m}{4}}\det(2\Im(z_0))^{\tfrac{1}{4}}  \int_{\C^m}  \Big[\int_{\R^m} e^{-\pi \big(  x (21_m-bi) x^T \big)}e^{2\pi ix\big(v^T-\overline{v}^{'T}  \big)} dx\Big]k_{z_0}(v',v')f'(v')d_{z_0}(v')\\
&\xlongequal{\textrm{ \cite[App.A, Thm.1]{Fo} }} 2^{\tfrac{m}{4}}\det(2\Im(z_0))^{\tfrac{1}{4}}  \int_{\C^m}  \Big[{\det}^{-1/2}( 21_m-ib) e^{-\pi((v-\overline{v}')( 21_m-ib)^{-1} (v^T-\overline{v}^{'T}) ) }\Big]k_{z_0}(v',v')f'(v')d_{z_0}(v')\\
&={\det}^{-1/2}(\tfrac{ 2i1_m+b}{2i}) \int_{\C^m}  \Big[ e^{-\pi i((v-\overline{v}')( 2i1_m+b)^{-1} (v^T-\overline{v}^{'T}) ) }\Big]k_{z_0}(v',v')f'(v')d_{z_0}(v').
\end{align*}
By Example \ref{example232}(2), we have:
\begin{align*}
&\pi_{ F_{z_0}, \psi}(g)(f')(v)\\
&= \gamma(\widehat{z_0}, \widehat{z_0'})\int_{\C^m} \kappa(\widehat{z_0},v; \widehat{z_0'},v')^{-1}f'(v')\kappa_{z_0}(v',v') dz_0(v')\\
&={\det}^{-1/2}\big(\tfrac{2i 1_m+b}{2i}\big) \int_{\C^m} e^{-\pi i(v-\overline{v'})(2i 1_m+b)^{-1} (v-\overline{v'})^T}f'(v')\kappa_{z_0}(v',v') dz_0(v')\\
&={\det}^{-1/2}\big(\tfrac{2i 1_m+b}{2i}\big)2^{\tfrac{m}{4}}  \int_{\R^m} \int_{\C^m} e^{-\pi i(v-\overline{v'})(2i 1_m+b)^{-1} (v-\overline{v'})^T}\kappa_{z_0}(v',v') e^{2\pi i\big( \tfrac{1}{2} x i x^T +xv^{'T}\big)}f(x)dxdz_0(v')\\
&={\det}^{-1/2}\big(\tfrac{2i 1_m+b}{2i}\big)2^{\tfrac{m}{4}}  \int_{\R^m} e^{2\pi i\big( \tfrac{1}{2} x i x^T \big)}f(x)dx \int_{\C^m} e^{-\pi i(v-\overline{v'})(2i 1_m+b)^{-1} (v-\overline{v'})^T}e^{-2\pi v'_y(v'_y)^T}e^{2\pi i(xv^{'T})} dz_0(v')\\
&=\theta_{ F_{z_0}, X^{\ast}}\big(\pi_{X^{\ast}, \psi}(g)(f)\big)(v).
\end{align*}
2) Let  $g=h(a)$.  
\begin{align*}
&\theta_{ F_{z_0}, X^{\ast}}\big(\pi_{X^{\ast}, \psi}(g)(f)\big)(v)\\
&=\det(2 \cdot 1_m)^{\tfrac{1}{4}} \int_{\R^m} e^{2\pi i\big( \tfrac{1}{2} x i x^T +xv^T\big)}\pi_{X^{\ast}, \psi}(g)f(x)dx\\
&=2^{\tfrac{m}{4}} \int_{\R^m} e^{2\pi i\big( \tfrac{1}{2} x i x^T +xv^T\big)}|\det(a)|^{1/2} f(xa)dx\\
&=2^{\tfrac{m}{4}}\det(2\Im(z))^{\tfrac{1}{4}} \int_{\C^m} \int_{\R^m} e^{2\pi i\big( \tfrac{1}{2} x i x^T +xv^T\big)}|\det(a)|^{1/2}  e^{2\pi i\big[\tfrac{1}{2}xaz_0a^Tx^T-xa\overline{v}^{'T} \big]} k_{z_0}(v',v')f'(v')d_{z_0}(v')dx\\
&=2^{\tfrac{m}{2}}|\det(a)|^{1/2}\int_{\C^m} \Big[\int_{\R^m} e^{-\pi (  x (1_m+aa^T) x^T - xx^T)}   e^{2\pi i\big(x(v^T-a\overline{v}^{'T}) \big)} dx\Big] k_{z_0}(v',v')f'(v')d_{z_0}(v') \\
&=2^{\tfrac{m}{2}}|\det(a)|^{1/2}  \int_{\C^m}  \Big[{\det}^{-1/2}(1_m+aa^T) e^{-\pi\big((v-\overline{v}'a^T)( 1_m+aa^T)^{-1} (v^T-(\overline{v}'a^T)^T)\big)}\Big]k_{z_0}(v',v')f'(v')d_{z_0}(v')\\
& ={\det}^{-1/2}(\tfrac{1_m+aa^T}{2})  |\det(a)|^{1/2}\int_{\C^m} \Big[e^{-\pi\big((v-\overline{v}'a^T)( 1_m+aa^T)^{-1} (v^T-(\overline{v}'a^T)^{T}) \big) }\Big]k_{z_0}(v',v')f'(v')d_{z_0}(v')\\
&\xlongequal{\textrm{Ex. \ref{example232}(1) }}\pi_{ F_{z_0}, \psi}(g)(f')(v)\\
&=\pi_{ F_{z_0}, \psi}(g)\big(\theta_{ F_{z_0}, X^{\ast}}(f)\big)(v).
\end{align*}
3) Let  $g=\omega$. 

\begin{align*}
&\theta_{ F_{z_0}, X^{\ast}}\big(\pi_{X^{\ast}, \psi}(g)(f)\big)(v)\\
&=\det(2\cdot 1_m)^{\tfrac{1}{4}} \int_{\R^m} e^{2\pi i\big( \tfrac{1}{2} x i x^T +xv^T\big)}\pi_{X^{\ast}, \psi}(g)f(x)dx\\
&=\det(2\cdot 1_m)^{\tfrac{1}{4}} \int_{\R^m} e^{2\pi i\big( \tfrac{1}{2} x i x^T +xv^T\big)}dx\int_{\R^m} e^{2\pi i x y^T}  f(-y) dy\\
&=\det(2\cdot 1_m)^{\tfrac{1}{4}} \int_{\R^m}  f(-y) dy \Big[\int_{\R^m} e^{2\pi i\big( \tfrac{1}{2} x i x^T +x(v^T+y^T)\big)}dx\Big] \\
&\xlongequal{\textrm{ \cite[App.A, Thm.1]{Fo} }}\det(2\cdot 1_m)^{\tfrac{1}{4}} \int_{\R^m}  f(-y) dy \Big[e^{-\pi\big((v+y)(v^T+y^T)\big)}\Big]; 
\end{align*}
By Example \ref{exm}, we have:
\begin{align*}
&\pi_{ F_{z_0}, \psi}(g)\big(\theta_{ F_{z_0}, X^{\ast}}(f)\big)(v)\\
&=\theta_{ F_{z_0}, X^{\ast}}(f)(v(-1_mi))\\
&=\det(2\Im(z))^{\tfrac{1}{4}} \int_{\R^m} e^{-\pi i v(-1_mi) v^T} e^{2\pi i\big( \tfrac{1}{2} xz_0 x^T -xiv^T\big)}f(x)dx\\
&=\det(2\Im(z))^{\tfrac{1}{4}} \int_{\R^m} e^{-\pi  v v^T} e^{\pi \big( - x x^T +2xv^T\big)}f(x)dx\\
&=\det(2\Im(z))^{\tfrac{1}{4}} \int_{\R^m} e^{-\pi  v v^T} e^{\pi \big( - yy^T -2yv^T\big)}f(-y)dy\\
&=\theta_{ F_{z_0}, X^{\ast}}\big(\pi_{X^{\ast}, \psi}(g)(f)\big)(v).
\end{align*}
\begin{lemma}\label{twosame}
 \begin{itemize}
 \item[(1)] $\beta_{z_0}(g_1, g_2)^{-1}=\widetilde{c}_{X^{\ast}}(g_1, g_2)$, for $g_i\in \Sp_{2m}(\R)$.
 \item[(2)] $\beta_{z_0}(g_1, g_2)^{-1}\beta_{z_0}(g_1^{h_{-1}}, g_2^{h_{-1}})^{-1}=1$, for $g_i\in \Sp_{2m}(\R)$.
 \end{itemize}
\end{lemma}
\begin{proof}
Part (1) is a consequence of the above proposition.  Part (2) comes from Lemma \ref{h1h1h1h1eq}.
\end{proof}
Retain the notations from  Example \ref{exm}.
\begin{lemma}
Assume that $m=1$. Let $k_t(v)=e^{\pi t v-\tfrac{\pi}{2}t^2}$, for $t\in \R$. Then $[\mathcal{B}^{-1}(k_t)](x)= 2^{\tfrac{1}{4}} e^{-\pi x^2+2\pi x t -\pi t^2}$.
\end{lemma}
\begin{proof}
See \cite[Prop. 6.10]{Zh}.
\end{proof}
\begin{lemma}
Assume that $m=1$. Let $f(v)=v$. Then $[\mathcal{B}^{-1}(f)](x)= 2^{\tfrac{5}{4}} x e^{-\pi x^2}$.
\end{lemma}
\begin{proof}
$[\mathcal{B}^{-1}(f)](x)=\tfrac{1}{\pi}\tfrac{d}{dt}[\mathcal{B}^{-1}(k_t)]\mid_{t=0}(x)=\bigg\{\tfrac{1}{\pi} 2^{\tfrac{1}{4}} e^{-\pi x^2+2\pi x t -\pi t^2}[2\pi x-2\pi t]\bigg\}\mid_{t=0}=2^{\tfrac{5}{4}} x e^{-\pi x^2}$.
\end{proof}
\begin{itemize}
\item[1)]  If we take $f(v)=e^{-\tfrac{\pi}{2}vv^T}$, then 
\begin{align*}
\theta_{  X^{\ast},F_{z_0}}(f)(x)&= \int_{\C^m} \det(2\Im(z))^{\tfrac{1}{4}} e^{2\pi i\big[\tfrac{1}{2}xzx^T-x\overline{v}^T \big]} k_{z_0}(v,v)e^{-\tfrac{\pi}{2}vv^T}d_{z_0}(v)\\
&=\mathcal{B}^{-1}(1)\\
&\stackrel{\textrm{ \cite[p.43]{Fo} }}{=}2^{m/4}e^{-\pi x x^T}.
\end{align*}
\item[2)] Take $f_k(v)=v_k ie^{-\tfrac{\pi}{2}vv^T}$, for $v=(v_1, \cdots, v_m)$. We denote $\widehat{v}_k=(v_1, \cdots,v_{k-1}, v_{k+1}, \cdots, v_m)$, and $\widehat{x}_k=(x_1, \cdots,x_{k-1}, x_{k+1}, \cdots, x_m)$.
\begin{align*}
&\theta_{  X^{\ast},F_{z_0}}(f_k)(x)\\
&= \int_{\C^m} \det(2\Im(z))^{\tfrac{1}{4}} v_ki e^{2\pi i\big[\tfrac{1}{2}xzx^T-x\overline{v}^T \big]} k_{z_0}(v,v)e^{-\tfrac{\pi}{2}vv^T}d_{z_0}(v)\\
&=\int_{\C} 2^{\tfrac{1}{4}} v_ki e^{2\pi i\big[\tfrac{1}{2}x_kx_k^T-x_k\overline{v_k}^T \big]} k_{z_0}(v_k,v_k)e^{-\tfrac{\pi}{2}v_kv_k^T}d_{z_0}(v_k) \bigg[\int_{\C^{m-1}} 2^{\tfrac{m-1}{4}}  e^{2\pi i\big[\tfrac{1}{2}\widehat{x}_kz\widehat{x}_k^T-\widehat{x}_k\overline{\widehat{v}_k}^T \big]} k_{z_0}(\widehat{v}_k,\widehat{v}_k)e^{-\tfrac{\pi}{2}\widehat{v}_k\widehat{v}_k^T}d_{z_0}(\widehat{v}_k)\bigg]\\
&=2^{(m-1)/4}e^{-\pi \widehat{x}_k \widehat{x}_k^T}\int_{\C} 2^{\tfrac{1}{4}} v_ki e^{2\pi i\big[\tfrac{1}{2}x_kx_k^T-x_k\overline{v_k}^T \big]} k_{z_0}(v_k,v_k)e^{-\tfrac{\pi}{2}v_kv_k^T}d_{z_0}(v_k)\\
&=2^{(m-1)/4}e^{-\pi \widehat{x}_k \widehat{x}_k^T} \mathcal{B}^{-1}(v_k)\\
&=2\cdot 2^{m/4}x_ke^{-\pi x x^T}.
\end{align*}
\end{itemize}
Let us denote $A(x)=e^{-\pi x x^T}$,  $B_i(x)=x_ie^{-\pi x x^T}$, and $B(x)=\begin{pmatrix} B_1(x)\\ \vdots\\B_m(x)\end{pmatrix}$, for $x=(x_1,\cdots, x_m)\in X$.
\begin{lemma}\label{ABM}
Let  $g=\begin{pmatrix}a & -b\\ b& a\end{pmatrix}\in \U_{m}(\C)$. 
\begin{itemize}
\item[(1)] $\pi_{X^{\ast},\psi }(g) A(x)=A(x)$.
\item[(2)]  $\pi_{X^{\ast},\psi }(g) B(x)=(-bi+a)^TB(x)$.
\end{itemize}
\end{lemma}
\begin{proof}
Let $F(v)=\begin{pmatrix} f_1(v)\\ \vdots\\f_m(v)\end{pmatrix}$. By the above 1) and 2), $\theta_{  X^{\ast},F_{z_0}}(2^{-m/4} f)=A$ and $\theta_{  X^{\ast},F_{z_0}}(2^{-1}\cdot 2^{-m/4} F)=B$. Hence: 
\begin{align*}
\pi_{X^{\ast},\psi }(g) A&=\pi_{X^{\ast},\psi }(g)\theta_{  X^{\ast},F_{z_0}}(2^{-m/4} f)\\
&=\theta_{  X^{\ast},F_{z_0}}\Big(2^{-m/4}\pi_{F_{z_0},\psi}(g)f\Big)\\
&\stackrel{Ex. \ref{exm}}{=}\theta_{  X^{\ast},F_{z_0}}(2^{-m/4} f)=A;\\
& \\
\pi_{X^{\ast},\psi }(g) B&=\pi_{X^{\ast},\psi }(g)\theta_{  X^{\ast},F_{z_0}}(2^{-1}\cdot 2^{-m/4} F)\\
&=\theta_{  X^{\ast},F_{z_0}}\Big(2^{-1}\cdot 2^{-m/4}\pi_{F_{z_0},\psi}(g)F\Big)\\
&\stackrel{Ex. \ref{exm}}{=}\theta_{  X^{\ast},F_{z_0}}\Big(2^{-1}\cdot 2^{-m/4}(-bi+a)^TF\Big)\\
&=(-bi+a)^TB.
\end{align*}
\end{proof}
\section{Siegel modular forms associated to Weil representations}
\subsection{ Iwasawa decomposition}\label{Iwa}
 Let us define: $$P^{>0}_{X^{\ast}}(\R)=\{ p=\begin{pmatrix} a & b\\ 0& (a^T)^{-1}  \end{pmatrix} \in P_{X^{\ast}}(\R)\mid a \textrm{ is a positive definite symmetric matrix}\}.$$
Then there exists the positive Iwasawa decomposition:
$$\Sp_{2m}(\R)=P^{>0}_{X^{\ast}}(\R) \cdot \U_m(\C), \quad \quad P^{>0}_{X^{\ast}}(\R)  \cap \U_m(\C)=1_{2m} .$$ 
For   $g=\begin{pmatrix}a & b\\ c& d\end{pmatrix}\in \Sp_{2m}(\R)$,  we  write $z_g=g(i1_m)=x_g+iy_g\in \mathbb{H}_m$. Then there exists a positive definite symmetric matrix  $\sqrt{y_g}$ such that $(\sqrt{y_g})^2=y_g$. Let:
\begin{itemize}
\item $p_g=\begin{pmatrix} \sqrt{y_g} & x_g(\sqrt{y_g})^{-1}\\ 0& (\sqrt{y_g})^{-1} \end{pmatrix} \in P^{>0}_{X^{\ast}}(\R)$;
\item $k_g=p_g^{-1}g=\begin{pmatrix} \sqrt{y_g}^{-1}  & -(\sqrt{y_g})^{-1}x_g\\ 0& \sqrt{y_g} \end{pmatrix} \cdot \begin{pmatrix}a & b\\ c& d\end{pmatrix}  = \begin{pmatrix}\sqrt{y_g}^{-1}a -(\sqrt{y_g})^{-1}x_g c & \sqrt{y_g}^{-1}b-(\sqrt{y_g})^{-1}x_g d\\ \sqrt{y_g} c &  \sqrt{y_g}d\end{pmatrix} \in  \U_m(\C)$.
\end{itemize}  
\begin{lemma}
$(ci+d)(-ci+d)^Ty_g^T=1_m$.
\end{lemma}
\begin{proof}
$g(i1_m)=(ai+b)(ci+d)^{-1}=x_g+iy_g$. Then $(ai+b)=(x_g+iy_g)(ci+d)$.
\[(ci+d)(-ci+d)^Ty_g^T=([cd^T-dc^T]i +[dd^T+cc^T])y_g^T=[dd^T+cc^T]y_g^T;\]
 \[(ci+d) \overline{(ai+b)}^T=(ci+d)(-ci+d)^T (x_g-iy_g)^T=[-(ci+d)(-ci+d)^Ty_g^T]i+[(ci+d)(-ci+d)^T x_g];\]
 \[(ci+d) \overline{(ai+b)}^T=(ci+d)(-a^Ti+b^T)=(cb^T-da^T)i+ (db^T+ca^T)=-1_mi +(db^T+ca^T).\]
\end{proof}
\begin{corollary}\label{identi}
$\det (ci+d) \det (-ci+d) \det y_g=1$ and $|\det (ci+d)| |\det (-ci+d)| \det y_g=1$.
\end{corollary}

\subsection{Automorphic factor I}\label{autoI}
Let  $z\in \mathbb{H}_m$, and $g=\begin{pmatrix} a & b\\ c& d\end{pmatrix}\in \Sp_{2m}(\R)$. Let us define a cocycle $\beta_z'$ such that  $\beta_z'(g_1,g_2)= \beta_z(g^{h_{-1}}_1, g^{h_{-1}}_2)$, for $g_1, g_2\in \Sp_{2m}(\R)$.   Let $\Mp^{z}_{2m}(\R)$ denote  the Metaplectic group associated  to the   $2$-cocycle $\beta'_z$ of center $T$.   Following \cite{Ta1}, define:
\begin{itemize}
\item[(1)] $J(g, z)= cz+d$;
\item[(2)] $J_{1/2}(g, z)=   \epsilon(g; z,z_0) \cdot |\det J(g, z)|^{1/2}$;
\item[(3)] $J_{3/2}(g, z)=  J_{1/2}(g, z)J(g, z)$;
\end{itemize}
\begin{itemize}
\item[(a)] $\alpha_{z}(g)=\frac{\det J(g, z)}{|\det J(g, z)|}$;
\item[(b)] $\widetilde{\Sp}^{z}_{2m}(\R)=\{ (g, t)\in \Mp^{z}_{2m}(\R) \mid  t^2\alpha_{z}(g)=1\}$.
\end{itemize}
\begin{lemma}\label{aleql}
$[\beta_z(g_1^{h_{-1}},g_2^{h_{-1}})]^2= \alpha_{ z}(g_1)\alpha_z(g_2)\alpha_{ z}(g_1g_2)^{-1}$, for $g_1,g_2\in \Sp_{2m}(\R)$.
\end{lemma}
\begin{proof}
See \cite[p.129]{Ta1}. 
\end{proof} 
\begin{lemma}
If $z=z_0$, then:
\begin{itemize}
\item[(1)] $\alpha_{z_0}(g^{h_{-1}}) \alpha_{z_0}(g) =1$.
\item[(2)] $[\beta_{z_0}(g_1,g_2)]^{-2}= \alpha_{ z_0}(g_1)\alpha_{z_0}(g_2)\alpha_{ z_0}(g_1g_2)^{-1}=[\beta_{z_0}(g_1^{h_{-1}},g_2^{h_{-1}})]^2$, for $g_1,g_2\in \Sp_{2m}(\R)$.
\end{itemize}
\end{lemma}
\begin{proof}
1) For $g=\begin{pmatrix} a& b\\ c& d\end{pmatrix}\in \Sp_{2m}(\R)$,  we have $\alpha_{z_0}(g^{h_{-1}})=\frac{\det (-ci+d)}{|\det (-ci+d)|}$. By Coro.\ref{identi},  $\alpha_{z_0}(g^{h_{-1}})^{-1}=\frac{\det (ci+d)}{|\det (ci+d)|}=\alpha_{z_0}(g)$. \\
2) By Lemma \ref{twosame}, $\beta_{z_0}(g_1,g_2)^{-1}=\beta_{z_0}(g_1^{h_{-1}},g_2^{h_{-1}})$. Then the result follows directly from the above lemma \ref{aleql}.
\end{proof}
\begin{lemma}\label{J121}
  $J_{1/2}(g, z)^2= \alpha_{z_0}(g)^{-1} \det J(g, z)$ and  $J_{1/2}(g_1g_2, z)=J_{1/2}(g_1, g_2(z))J_{1/2}(g_2, z)\beta_{z_0}(g_1^{h_{-1}},g_2^{h_{-1}})$, for $g, g_i\in \Sp_{2m}(\R)$.
\end{lemma}
\begin{proof}
It follows from \cite[pp.131-132]{Ta1}.
\end{proof}
\begin{remark}
$J_{1/2}(g, z_0)=|\det J(g, z_0)|^{1/2}=1$, for $g\in \U_m(\C)$.
\end{remark}
 Recall the notations  $\overline{\Sp}_{2m}(\R)$, $m_{X^{\ast}}$, and   $\overline{c}_{X^{\ast}}$  from Section \ref{schod}.
\begin{lemma}
$ \alpha_{z_0}(g)= [m_{X^{\ast}}(g)]^{-2}$, for  $g\in \Sp_{2m}(\R)$.
\end{lemma}
\begin{proof}
If $z=z_0$, then $\widetilde{c}_{X^{\ast}}(g_1,g_2)=\beta_{z_0}(g_1,g_2)^{-1}=\beta_{z_0}(g^{h_{-1}}_1,g^{h_{-1}}_2)$. By (\ref{chap228inter}), we have:
$$m_{X^{\ast}}(g_1g_2)^{-2} m_{X^{\ast}}(g_1)^2 m_{X^{\ast}}(g_2)^2  \alpha_{z_0}(g_1)\alpha_{z_0}(g_2)\alpha_{z_0}(g_1g_2)^{-1}=1.$$
Since $\Sp_{2m}(\R)$ is a perfect group, $ \alpha_{z_0}(g)= [m_{X^{\ast}}(g)]^{-2}$.
\end{proof}
As a  consequence, $\widetilde{\Sp}^{z_0}_{2m}(\R)=\{ (g, t)\mid g\in \Sp_{2m}(\R), t=\pm m_{X^{\ast}}(g)=\pm m_{X^{\ast}}(g^{h_{-1}})^{-1}\}$.
\begin{lemma}\label{imbed}
There exists a group isomorphism $\iota: \overline{\Sp}_{2m}(\R) \to \widetilde{\Sp}^{z_0}_{2m}(\R); [g, t]\longmapsto [g, m_{X^{\ast}}(g)t]$.
\end{lemma} 
\begin{proof}
\begin{align*}
&\iota([g_1, t_1])\iota([g_2, t_2])\\
&= [g_1, m_{X^{\ast}}(g_1)t_1] [g_2, m_{X^{\ast}}(g_2)t_2]\\
&=[g_1g_2, \beta_{z_0}(g_1,g_2)^{-1} m_{X^{\ast}}(g_1)m_{X^{\ast}}(g_2)t_1 t_2]\\
&=[g_1g_2, \widetilde{c}_{X^{\ast}}(g_1,g_2) m_{X^{\ast}}(g_1)m_{X^{\ast}}(g_2)t_1 t_2]\\
&=[g_1g_2, \overline{c}_{X^{\ast}}(g_1, g_2) m_{X^{\ast}}(g_1g_2)t_1 t_2]\\
&=\iota([g_1g_2, t_1t_2 \overline{c}_{X^{\ast}}(g_1, g_2)])\\
&=\iota([g_1, t_1]\cdot [g_2, t_2]).
\end{align*}
\end{proof}
\begin{lemma}
$[J_{1/2}(g, z)m_{X^{\ast}}(g)^{-1}]^2=\det(cz+d)$, for $g=\begin{pmatrix} a & b\\ c& d  \end{pmatrix} \in \Sp_{2m}(\R)$.
\end{lemma}
\begin{proof}
By Lemma \ref{J121}, $[J_{1/2}(g, z)m_{X^{\ast}}(g)^{-1}]^2=\det J(g, z)\alpha_{z_0}(g)^{-1}m_{X^{\ast}}(g)^{-2}=\det J(g, z).$
\end{proof}

\begin{definition}
  $\sqrt{\det(cz+d)}\stackrel{Def.}{=} J_{1/2}(g, z)m_{X^{\ast}}(g)^{-1}$.  
  \end{definition}
  Let $g_3=g_1g_2$ and $g_i=\begin{pmatrix} a_i& b_i\\ c_i &d_i \end{pmatrix}$. 
    \begin{lemma}
    $\sqrt{\det(c_3z+d_3)}=[\sqrt{\det(c_1(g_2 z)+d_1)}]\cdot [\sqrt{\det(c_2 z+d_2)}]\cdot \overline{c}_{X^{\ast}}(g_1,g_2)$.
  \end{lemma}
  \begin{proof}
 \begin{align*}
 \text{Left side}
 &=J_{1/2}(g_1g_2, z)m_{X^{\ast}}(g_1g_2)^{-1}\\
 &\stackrel{\textrm{Lem.} \ref{J121}}{=}J_{1/2}(g_1, g_2(z))J_{1/2}(g_2, z)\beta_{z_0}(g_1^{h_{-1}},g_2^{h_{-1}})m_{X^{\ast}}(g_1g_2)^{-1}\\
 &=J_{1/2}(g_1, g_2(z))J_{1/2}(g_2, z)\widetilde{c}_{X^{\ast}}(g_1,g_2)m_{X^{\ast}}(g_1g_2)^{-1}\\
 &\stackrel{(\ref{chap228inter})}{=}J_{1/2}(g_1, g_2(z))J_{1/2}(g_2, z) m_{X^{\ast}}(g_1)^{-1} m_{X^{\ast}}(g_2)^{-1}\overline{c}_{X^{\ast}}(g_1, g_2)\\
 &= \text{Right side}. 
 \end{align*}
  \end{proof}
 For later use, let us consider $r\in \Gamma_m(1,2)$ and define
   \begin{align}\label{lambda}
   \lambda(r)= m_{X^{\ast}}(r)\widetilde{\beta}^{-1}(r).
    \end{align}
  \begin{lemma}\label{mutirr12}
  $ \lambda(r_1)\lambda(r_2)=\lambda(r_1r_2)\overline{c}_{X^{\ast}}(r_1,r_2)$, for $r_i\in \Gamma_m(1,2)$.
\end{lemma}
\begin{proof}
By Lemma \ref{widetildec}, $\widetilde{\beta}(r_1)^{-1}\widetilde{\beta}(r_2)^{-1}\widetilde{\beta}(r_1r_2)=\widetilde{c}_{X^{\ast}}(r_1,r_2)$. Consequently,
\begin{align*}
& m_{X^{\ast}}(r_1)\widetilde{\beta}(r_1)^{-1} m_{X^{\ast}}(r_2)\widetilde{\beta}(r_2)^{-1}\\
&= m_{X^{\ast}}(r_1)m_{X^{\ast}}(r_2)\widetilde{\beta}(r_1r_2)^{-1}\widetilde{c}_{X^{\ast}}(r_1,r_2)\\
&= \widetilde{\beta}(r_1r_2)^{-1} m_{X^{\ast}}(r_1r_2)\overline{c}_{X^{\ast}}(r_1, r_2).
\end{align*}
\end{proof}
\subsection{Automorphic factor II}
For $\widetilde{g}=(g, t)\in \widetilde{\Sp}^{z_0}_{2m}(\R)$, with $g=\begin{pmatrix} a & b\\ c& d\end{pmatrix}\in \Sp_{2m}(\R)$, define:
\begin{itemize}
\item[(1)] $J_{1/2}(\widetilde{g}, z)= t^{-1} \cdot \epsilon(g; z,z_0) \cdot |\det J(g, z)|^{1/2}$;
\item[(2)] $J_{3/2}(\widetilde{g}, z)=  J_{1/2}(\widetilde{g}, z)J(g, z)$.
\end{itemize}

\begin{lemma}\label{J122}
Let  $\widetilde{g}, \widetilde{g}_i=(g_i, t_i)\in \widetilde{\Sp}^{z_0}_{2m}(\R)$.
\begin{itemize}
\item[(1)] $J_{1/2}(\widetilde{g}, z)^2=\det J(g, z)$.
\item[(2)] $J_{1/2}(\widetilde{g}_1\widetilde{g}_2, z)=J_{1/2}(\widetilde{g}_1, g_2(z))J_{1/2}(\widetilde{g}_2, z)$.
\item[(3)]$J_{3/2}(\widetilde{g}_1\widetilde{g}_2, z)=J_{3/2}(\widetilde{g}_1, g_2(z))J_{3/2}(\widetilde{g}_2, z)$.
\end{itemize}    
\end{lemma}
\begin{proof}
1) $J_{1/2}(g, z)^2=J_{1/2}(\widetilde{g}, z)^2 t^{2}=\det J(g, z)\alpha_{z_0}(g)^{-1}$.\\
2)  Note that $\widetilde{g}_1\widetilde{g}_2=[g_1g_2, \beta_{z_0}(g_1^{h_{-1}},g_2^{h_{-1}})t_1t_2]$.
\begin{align*}
&J_{1/2}(\widetilde{g}_1\widetilde{g}_2, z)\\
&=J_{1/2}(g_1g_2, z)\beta_{z_0}(g_1^{h_{-1}},g_2^{h_{-1}})^{-1}t_1^{-1}t_2^{-1}\\
&= J_{1/2}(g_1, g_2(z))J_{1/2}(g_2, z)\beta_{z_0}(g_1^{h_{-1}},g_2^{h_{-1}}) \beta_{z_0}(g_1^{h_{-1}},g_2^{h_{-1}})^{-1}t_1^{-1}t_2^{-1}\\
&=J_{1/2}(\widetilde{g}_1, g_2(z))J_{1/2}(\widetilde{g}_2, z).
\end{align*}
3) It is a consequence of (2).
\end{proof}

  \begin{corollary}
  For $\widetilde{g}=(g, t)\in \widetilde{\Sp}^{z_0}_{2m}(\R)$, we have:
  \begin{itemize}
  \item  $J_{1/2}(\widetilde{g}, z) t m_{X^{\ast}}(g)^{-1}=\sqrt{\det(cz+d)}$.
\item $J_{1/2}(\widetilde{g}, z) = \sqrt{\det(cz+d)}$ or $-\sqrt{\det(cz+d)}$.
\end{itemize}
  \end{corollary}
  Let $\widetilde{P^{>0}_{X^{\ast}}}^{z_0} (\R)$, $\widetilde{\U}_m^{z_0}(\C)$ denote the inverse images of $P^{>0}_{X^{\ast}}(\R)$ and $\U_m(\C)$ in $\widetilde{\Sp}^{z_0}_{2m}(\R)$ respectively. Note that for $p\in P^{>0}_{X^{\ast}}(\R)$, $m_{X^{\ast}}(p)=1$. Moreover, for $g\in \Sp_{2m}(\R)$, if write $g=p_gk_g$ for $p_g\in  P^{>0}_{X^{\ast}}(\R)$ and $k_g\in \U_m(\C)$, then $m_{X^{\ast}}(g)=m_{X^{\ast}}(k_g)$. It is known that 
$$\widetilde{c}_{X^{\ast}}(g,p)=1=\widetilde{c}_{X^{\ast}}(p,g), \quad\quad g\in \Sp_{2m}(\R), p\in P_{X^{\ast}}(\R).$$
Hence $\widetilde{P^{>0}_{X^{\ast}}}^{z_0} (\R) \simeq P^{>0}_{X^{\ast}} (\R) \times \mu_8$, and 
\[P^{>0}_{X^{\ast}} (\R)  \hookrightarrow \widetilde{\Sp}^{z_0}_{2m}(\R); p \longmapsto [p,1].\] Then  there exists the positive Iwasawa decomposition:
$$\widetilde{\Sp}^{z_0}_{2m}(\R)=P^{>0}_{X^{\ast}}(\R) \cdot \widetilde{\U}_m^{z_0}(\C), \quad \quad P^{>0}_{X^{\ast}}(\R)  \cap \widetilde{\U}_m^{z_0}(\C)=1_{2m}.$$ 
\[\widetilde{g}=(g, t_g)\in \widetilde{\Sp}^{z_0}_{2m}(\R), \quad \widetilde{g}=p_g\widetilde{k_g}, \quad z_{\widetilde{g}}=z_{p_g}.\]
\[ J_{1/2}(\widetilde{g}, z_0)=t_g^{-1}J_{1/2}(g, z_0)=t_g^{-1}J_{1/2}(p_g, z_0).\]
\[ J_{3/2}(\widetilde{g}, z_0)=J_{1/2}(\widetilde{g}, z_0)(cz_0+d)=t_g^{-1}J_{3/2}(g, z_0).\]
\subsection{Siegel modular forms: $ P_{X^{\ast}}(\R)$}
Let  $p=\begin{pmatrix} a & b\\ 0& (a^T)^{-1}  \end{pmatrix} \in P_{X^{\ast}}(\R) $. We write $z_p= p(i1_m)=ba^T + iaa^T \in \mathbb{H}_m$. Recall the notations $A$, $B_k$, and $B$ from Lemma \ref{ABM}.    By (\ref{chap2representationsp2})(\ref{chap2representationsp3}), we have:
\begin{align*}
\pi_{X^{\ast}, \psi} (p)A( x)&= \pi_{X^{\ast}, \psi} (h(a) u(a^{-1}b))A( x)\\
&=|\det(a)|^{1/2}\pi_{X^{\ast}, \psi} (u(a^{-1}b))A( xa)\\
&=|\det(a)|^{1/2}\psi(\tfrac{1}{2} xa a^{-1}b a^T x^T) A( xa)\\
&= |\det(a)|^{1/2}  e^{\pi i  xba^T x^T}\cdot e^{- \pi xa a^Tx^T}\\
&=  |\det(a)|^{1/2} e^{i   \pi xz_p x^T }.
\end{align*}
\begin{align*}
\pi_{X^{\ast}, \psi} (p)B_k( x)&= \pi_{X^{\ast}, \psi} (h(a) u(a^{-1}b))B_k( x)\\
&=|\det(a)|^{1/2}\pi_{X^{\ast}, \psi} (u(a^{-1}b))B_k( xa)\\
&=|\det(a)|^{1/2}\psi(\tfrac{1}{2} xa a^{-1}b a^T x^T) B_k( xa)\\
&=  |\det(a)|^{1/2} (xa)_k e^{i   \pi xz_p x^T }; 
\end{align*}
\begin{align*}
\pi_{X^{\ast}, \psi} (p)B( x)&= \pi_{X^{\ast}, \psi} (p)\begin{pmatrix} B_1(x)\\ \vdots\\B_m(x)\end{pmatrix}\\
&=  |\det(a)|^{1/2}e^{i   \pi xz_p x^T } (xa)^T .
\end{align*}
\begin{align*}
 J_{1/2}(p, z_0)=   \epsilon(p; z_0,z_0) \cdot |\det J(p, z_0)|^{1/2}=|\det (a)|^{-1/2}.
 \end{align*}
\begin{align*}
 J_{3/2}(p, z_0)=   J_{1/2}(p, z_0) J(p, z_0)=|\det (a)|^{-1/2} (a^T)^{-1}.
 \end{align*}
 \begin{align*}
J_{1/2}(p, z_0)\pi_{X^{\ast}, \psi} (p)A( x)= e^{i   \pi xz_p x^T }.
\end{align*}
\begin{align*}
J_{3/2}(p, z_0)\pi_{X^{\ast}, \psi} (p)B( x)=x^T e^{i   \pi xz_p x^T }.
\end{align*}
 \subsection{Siegel modular forms: $ \Sp_{2m}(\R)$}
\[g=p_gk_g  (\textrm{ cf. Section }\ref{Iwa}), \quad\quad z_g=z_{p_g}.\]
\[ J_{1/2}(p_gk_g, z_0)=J_{1/2}(p_g, k_g(z_0))J_{1/2}(k_g, z_0)\beta_{z_0}(p_g^{h_{-1}},k_g^{h_{-1}})=J_{1/2}(p_g, z_0).\]
\[ J_{3/2}(p_gk_g, z_0)= J_{1/2}(p_gk_g, z_0)(cz_0+d)=(cz_0+d)J_{1/2}(p_g, z_0).\]
\begin{align*}
&J_{1/2}(g, z_0)\pi_{X^{\ast}, \psi} (g)A( x)\\
&= J_{1/2}(p_g, z_0)\widetilde{c}_{X^{\ast}}(p_g, k_g)^{-1}\pi_{X^{\ast}, \psi} (p_g)\pi_{X^{\ast}, \psi} (k_g)A( x)\\
&= J_{1/2}(p_g, z_0)\pi_{X^{\ast}, \psi} (p_g)\pi_{X^{\ast}, \psi} (k_g)A( x)\\
&\stackrel{\textrm{Lem.} \ref{ABM}}{=}J_{1/2}(p_g, z_0)\pi_{X^{\ast}, \psi} (p_g)A( x)\\
&=e^{i   \pi xz_g x^T }.
\end{align*}
\begin{align*}
&J_{3/2}(g, z_0)\pi_{X^{\ast}, \psi} (g)B( x)\\
&=(cz_0+d) J_{1/2}(p_g, z_0)\widetilde{c}_{X^{\ast}}(p_g, k_g)^{-1}\pi_{X^{\ast}, \psi} (p_g)\pi_{X^{\ast}, \psi} (k_g)B( x)\\
&=(cz_0+d) J_{1/2}(p_g, z_0)\pi_{X^{\ast}, \psi} (p_g)[\pi_{X^{\ast}, \psi} (k_g)B]( x)\\
&\stackrel{\textrm{Lem.} \ref{ABM}}{=}(cz_0+d) J_{1/2}(p_g, z_0)\pi_{X^{\ast}, \psi} (p_g)[(-\sqrt{y_g} c z_0+ \sqrt{y_g}d)^TB]( x)\\
&=(cz_0+d) e^{i   \pi xz_g x^T }(-\sqrt{y_g} c z_0+ \sqrt{y_g}d)^T  \sqrt{y_g}^Tx^T\\
&=e^{i   \pi xz_g x^T }(cz_0+d) (-cz_0+d)^Ty_g^Tx^T\\
&=e^{i   \pi xz_g x^T }x^T\\
&=J_{3/2}(p_g, z_0)\pi_{X^{\ast}, \psi} (p_g)B( x).
\end{align*}
\subsection{Siegel modular forms: $\widetilde{\Sp}^{z_0}_{2m}(\R)$} 
\[\widetilde{g}=(g, t_g)\in \widetilde{\Sp}^{z_0}_{2m}(\R), \quad \widetilde{g}=p_g\widetilde{k_g}, \quad z_{\widetilde{g}}=z_{p_g}.\]
\begin{align*}
& J_{1/2}(\widetilde{g}, z_0)\pi_{X^{\ast}, \psi} (\widetilde{g})A( x)\\
&= t_g^{-1}J_{1/2}(g, z_0)t_g\pi_{X^{\ast}, \psi} (g)A( x)\\
&= e^{i   \pi xz_g x^T }.
\end{align*}
\begin{align*}
& J_{3/2}(\widetilde{g}, z_0)\pi_{X^{\ast}, \psi} (\widetilde{g})B( x)\\
&=J_{3/2}(g, z_0)\pi_{X^{\ast}, \psi} (g)B( x)\\
&=e^{i   \pi xz_g x^T }x^T\\
&=J_{3/2}(p_g, z_0)\pi_{X^{\ast}, \psi} (p_g)B( x).
\end{align*}
\subsection{Siegel modular forms: $\Gamma_{m}(1,2)$}
\begin{align*}
 \theta_{L, X^{\ast}}(f)(0)=\sum_{l\in L\cap X}f(l)=\sum_{n\in \Z^m}f( n).
 \end{align*}
For $g\in \Sp_{2m}(\R)$, let us define:
 \begin{align*}
 \theta_{1/2}(g)&= \theta_{L, X^{\ast}}\Big(J_{1/2}(g, z_0)\pi_{X^{\ast}, \psi} (g)A\Big)(0)\\
&=\sum_{n\in \Z^m}J_{1/2}(g, z_0)\pi_{X^{\ast}, \psi} (g)A(  n)\\
&=\sum_{n\in \Z^m} e^{i   \pi nz_g n^T }.
 \end{align*}
 \begin{align*}
 \theta_{3/2}(g)&=\theta_{L, X^{\ast}}\Big(J_{3/2}(g, z_0)\pi_{X^{\ast}, \psi} (g)B\Big)(0)\\
 &=\sum_{n\in \Z^m}J_{3/2}(g, z_0)\pi_{X^{\ast}, \psi} (g)B(  n)\\
&=\sum_{n\in \Z^m} e^{i   \pi nz_g n^T }n^T.
  \end{align*}
 Note that $\theta_{1/2}(g)$ and $\theta_{3/2}(g)$ only depend on $z_g=g(z_0)$.  For $z\in \mathbb{H}_m$,  we define:
\begin{align*}
\theta_{1/2}(z)&=\sum_{n\in \Z^m} e^{i  \pi nz n^T},
\end{align*}
\begin{align*}
\theta_{3/2}(z)&=\sum_{n\in \Z^m} n^T e^{i   \pi nz n^T }.
\end{align*}
Let us consider the explicit action of $\Gamma_m(1,2)$ on $\theta_{1/2}(z)$ and  $\theta_{3/2}(z)$. For $r=\begin{pmatrix}   a & b \\  c& d\end{pmatrix} \in \Gamma_{m}(1,2)$, we have:
\begin{align*}
\theta_{1/2}(r g)&= \theta_{1/2}(r p_g)\\
&=\theta_{L, X^{\ast}}\Big(J_{1/2}(r p_g, z_0)\pi_{X^{\ast}, \psi} (r p_g)A\Big)(0)\\
&\stackrel{\widetilde{c}_{X^{\ast}}(r, p_g)=1}{=} J_{1/2}(r ,  p_g(z_0))\pi_{L, \psi}(r)\theta_{L, X^{\ast}}\Big(J_{1/2}(p_g,z_0)\pi_{X^{\ast}, \psi} ( p_g)A\Big)(0) \\
&\stackrel{(\ref{chapeq8})}{=}J_{1/2}(r , z_g)\epsilon_{r} \theta_{1/2}( g)\\
&=\sqrt{\det(cz_g+d)} m_{X^{\ast}}(r)\epsilon_{r} \theta_{1/2}( g)\\
&= m_{X^{\ast}}(r)\widetilde{\beta}^{-1}(r)\sqrt{\det(cz_g+d)}\theta_{1/2}( g)\\
&=\lambda(r)\sqrt{\det(cz_g+d)}\theta_{1/2}( g).
\end{align*} 
\begin{align*}
\theta_{3/2}(r g)&= \theta_{3/2}(r p_g)\\
&=\theta_{L, X^{\ast}}\Big(J_{3/2}(r p_g, z_0)\pi_{X^{\ast}, \psi} (r p_g)B\Big)(0)\\
&= J_{3/2}(r ,  p_g(z_0))\pi_{L, \psi}(r)\theta_{L, X^{\ast}}\Big(J_{1/2}(p_g,z_0)\pi_{X^{\ast}, \psi} ( p_g)B\Big)(0) \\
&\stackrel{(\ref{chapeq8})}{=} J_{3/2}(r , z_g)\epsilon_{r}\theta_{L, X^{\ast}}\Big(J_{3/2}(p_g,z_0)\pi_{X^{\ast}, \psi} ( p_g)B\Big)(0)\\
&=J_{3/2}(r , z_g)\epsilon_{r} \theta_{3/2}( g)\\
&= m_{X^{\ast}}(r)\widetilde{\beta}^{-1}(r)\sqrt{\det(cz_g+d)}(cz_g+d)\theta_{3/2}( g)\\
&= \lambda(r)\sqrt{\det(cz_g+d)}(cz_g+d)\theta_{3/2}( g).
\end{align*} 
 We conclude: 
 \begin{theorem}\label{mainthm1}
 For $r=\begin{pmatrix} a&b\\ c& d\end{pmatrix}\in \Gamma_{m}(1,2)$, we have:
 \begin{itemize}
 \item[(1)] $\theta_{1/2}(r z)=\lambda(r) \sqrt{\det(cz+d)}\theta_{1/2}( z)$.
 \item[(2)] $\theta_{3/2}(r z)=\lambda(r) \sqrt{\det(cz+d)}(cz+d)\theta_{3/2}( z)$. 
 \end{itemize}
\end{theorem}
\subsection{Siegel modular forms: $\Sp_{2m}(\Z)$}
Let   $\widetilde{\Gamma}^{z_0}_m(1,2)$ denote the inverse image of $ \Gamma_m(1,2)$ in $\widetilde{\Sp}^{z_0}_{2m}(\R)$. Let $\widetilde{g}=(g, t_g)\in  \widetilde{\Sp}^{z_0}_{2m}(\R)$ and  $\widetilde{r}=(r, t_r)\in  \widetilde{\Gamma}^{z_0}_m(1,2)$. Let us define:
 \begin{align*}
 \theta_{1/2}(\widetilde{g})&=\theta_{L, X^{\ast}}\Big(  J_{1/2}(\widetilde{g}, z_0)\pi_{X^{\ast}, \psi} (\widetilde{g})A\Big)(0)\\
& =\theta_{L, X^{\ast}}\Big( J_{1/2}(g, z_0)\pi_{X^{\ast}, \psi} (g)A\Big)(0)\\
&=\sum_{n\in \Z^m} e^{i   \pi nz_g n^T };
 \end{align*}
 \begin{align*}
 \theta_{3/2}(\widetilde{g})&=\theta_{L, X^{\ast}}\Big(  J_{3/2}(\widetilde{g}, z_0)\pi_{X^{\ast}, \psi} (\widetilde{g})B\Big)(0)\\
 & =\theta_{L, X^{\ast}}\Big( J_{3/2}(g, z_0)\pi_{X^{\ast}, \psi} (g)B\Big)(0)\\
&=\sum_{n\in \Z^m} e^{i   \pi nz_g n^T }n^T.
  \end{align*}
  Then:
  \begin{align*}
 \theta_{1/2}(\widetilde{r}\widetilde{g})& =\theta_{L, X^{\ast}}\Big( J_{1/2}(rg, z_0)\pi_{X^{\ast}, \psi} (rg)A\Big)(0)\\
&=J_{1/2}(r , z_g)\epsilon_{r} \theta_{1/2}( g)\\
&=J_{1/2}(\widetilde{r} , z_g) t_r \epsilon_{r} \theta_{1/2}( \widetilde{g});
 \end{align*}
   \begin{align*}
 \theta_{3/2}(\widetilde{r}\widetilde{g})&=\theta_{L, X^{\ast}}\Big(  J_{3/2}(rg, z_0)\pi_{X^{\ast}, \psi} (rg)B\Big)(0)\\
 &=J_{3/2}(r , z_g)\epsilon_{r} \theta_{3/2}( g)\\
&=J_{3/2}(\widetilde{r} , z_g) t_r \epsilon_{r} \theta_{3/2}( \widetilde{g}).
  \end{align*}

  Recall notations $\mathcal{M}'_{2}$, $\mathcal{N}'_{2}$, $M_q $,  $M^i_q$, $M^{(j,k)}_q$ from Sections \ref{setminus}--\ref{modification11}. Let us write:  
$$\iota_i(\begin{pmatrix} 1& 0\\ 0& 1\end{pmatrix})=1_{2m}, \quad  u_{i}= \iota_i(\begin{pmatrix} 1& 1\\ 0& 1\end{pmatrix}), \quad  u^-_i=\iota_i(\begin{pmatrix} 1& 0\\-1& 1\end{pmatrix}), \quad  \omega_{i}= \iota_i(\begin{pmatrix} 0& 1\\-1&0\end{pmatrix}), \quad n=(n_1,\cdots,n_m)\in \Z^m.$$ 
Note that $ u_{i}$,$u_{i}^-$ are  just the previous  $u_{ii}$ and $u^-_{ii}$. Let us write $\epsilon_1=(1,0,\cdots, 0), \cdots, \epsilon_m=(0,\cdots, 0, 1)$ for the natural basis of $\R^m$.
\subsubsection{$\widetilde{u}_{i}=(u_i,t_{u_i})$}
\begin{align*}
 \theta_{L, X^{\ast}}(A)(\tfrac{1}{2} e_i^{\ast})=\sum_{l=(l_1, \cdots, l_m)\in L\cap X}\psi(\tfrac{1}{2} l_i)A(l)=\sum_{n=(n_1,\cdots,n_m)\in \Z^m}\psi(\tfrac{1}{2} n_i)A( n).
 \end{align*}
 \begin{align*}
 &\theta_{L, X^{\ast}}\Big(J_{1/2}(p_g,z_0)\pi_{X^{\ast}, \psi} ( p_g)A\Big)(\tfrac{1}{2} e_i^{\ast})\\
 &=\sum_{n\in \Z^m}\psi(\tfrac{1}{2} n_i)\Big(J_{1/2}(p_g,z_0)\pi_{X^{\ast}, \psi} ( p_g)A\Big)( n)\\
 &=\sum_{n\in \Z^m} (-1)^{n_i}e^{i   \pi nz_g n^T }.
 \end{align*}
 \begin{align*}
 &\theta_{L, X^{\ast}}\Big(J_{32}(p_g,z_0)\pi_{X^{\ast}, \psi} ( p_g)B\Big)(\tfrac{1}{2} e_i^{\ast})\\
 &=\sum_{n\in \Z^m}\psi(\tfrac{1}{2} n_i)\Big(J_{3/2}(p_g,z_0)\pi_{X^{\ast}, \psi} ( p_g)B\Big)( n)\\
 &=\sum_{n\in \Z^m} (-1)^{n_i}n^Te^{i   \pi nz_g n^T }.
 \end{align*}
 \begin{align*}
&\theta_{1/2}(\widetilde{u}_{i} \widetilde{g})\\
&=\theta_{1/2}(u_i  p_g)\\
&=\theta_{L, X^{\ast}}\Big(J_{1/2}(u_i p_g, z_0)\pi_{X^{\ast}, \psi} (u_i p_g)A\Big)(0)\\
&= J_{1/2}(u_i ,  p_g(z_0))\pi_{L, \psi}(u_i)\theta_{L, X^{\ast}}\Big(J_{1/2}(p_g,z_0)\pi_{X^{\ast}, \psi} ( p_g)A\Big)( 0) \\
&\stackrel{\textrm{Case  6 }(\ref{ui})}{=} J_{1/2}(u_i , z_g)\theta_{L, X^{\ast}}\Big(J_{1/2}(p_g,z_0)\pi_{X^{\ast}, \psi} ( p_g)A\Big)(\tfrac{1}{2} e_i^{\ast})\\
&=J_{1/2}(u_i , z_g) \sum_{n\in \Z^m} (-1)^{n_i}e^{i   \pi nz_g n^T }\\
&= J_{1/2}(\widetilde{u}_{i} , z_g) t_{u_i}\sum_{n\in \Z^m} (-1)^{n_i}e^{i   \pi nz_g n^T }.
\end{align*}   
\begin{align*}
&\theta_{3/2}(\widetilde{u}_{i} \widetilde{g})\\
&=\theta_{3/2}(u_i  p_g)\\
&=\theta_{L, X^{\ast}}\Big(J_{3/2}(u_i p_g, z_0)\pi_{X^{\ast}, \psi} (u_i p_g)B\Big)(0)\\
&= J_{3/2}(u_i ,  p_g(z_0))\pi_{L, \psi}(u_i)\theta_{L, X^{\ast}}\Big(J_{3/2}(p_g,z_0)\pi_{X^{\ast}, \psi} ( p_g)B\Big)( 0) \\
&\stackrel{\textrm{Case  6 }(\ref{ui})}{=}  J_{3/2}(u_i , z_g)\theta_{L, X^{\ast}}\Big(J_{3/2}(p_g,z_0)\pi_{X^{\ast}, \psi} ( p_g)B\Big)(\tfrac{1}{2} e_i^{\ast})\\
&=J_{3/2}(u_i , z_g) \sum_{n\in \Z^m} (-1)^{n_i}n^Te^{i   \pi nz_g n^T }\\
&= J_{3/2}(\widetilde{u}_{i} , z_g) t_{u_i}\sum_{n\in \Z^m} (-1)^{n_i}n^Te^{i   \pi nz_g n^T }.
\end{align*} 
\subsubsection{$\widetilde{u}_i^-=(u_i^-, t_{u_i^-})$} Write  $u_i^-=\begin{pmatrix} 1_m & 0 \\ c & 1_m \end{pmatrix}$, for $c=\epsilon_{ii}(-1)$.
\begin{align*}
 \theta_{L, X^{\ast}}(f)(\tfrac{1}{2} e_i)=\sum_{l\in L\cap X}f(l+\tfrac{1}{2} e_i)=\sum_{n=(n_1,\cdots,n_m)\in \Z^m}f( n+\tfrac{1}{2}\epsilon_i).
 \end{align*}
 \begin{align*}
 &\theta_{L, X^{\ast}}\Big(J_{1/2}(p_g,z_0)\pi_{X^{\ast}, \psi} ( p_g)A\Big)(\tfrac{1}{2} e_i)\\
 &=\sum_{n\in \Z^m}\Big(J_{1/2}(p_g,z_0)\pi_{X^{\ast}, \psi} ( p_g)A\Big)( n+\tfrac{1}{2}\epsilon_i)\\
 &=\sum_{n\in \Z^m} e^{i   \pi ( n+\tfrac{1}{2}\epsilon_i)z_g ( n+\tfrac{1}{2}\epsilon_i)^T }.
 \end{align*}
 \begin{align*}
 &\theta_{L, X^{\ast}}\Big(J_{3/2}(p_g,z_0)\pi_{X^{\ast}, \psi} ( p_g)B\Big)(\tfrac{1}{2} e_i)\\
 &=\sum_{n\in \Z^m}\Big(J_{3/2}(p_g,z_0)\pi_{X^{\ast}, \psi} ( p_g)B\Big)( n+\tfrac{1}{2}\epsilon_i)\\
 &=\sum_{n\in \Z^m} ( n+\tfrac{1}{2}\epsilon_i)^T  e^{i   \pi ( n+\tfrac{1}{2}\epsilon_i)z_g ( n+\tfrac{1}{2}\epsilon_i)^T }.
 \end{align*}
 \begin{align*}
&\theta_{1/2}(\widetilde{u}_i^- \widetilde{g})\\
&=\theta_{1/2}(u^-_i  p_g)\\
&=\theta_{L, X^{\ast}}\Big(J_{1/2}(u^-_i p_g, z_0)\pi_{X^{\ast}, \psi} (u^-_i p_g)A\Big)(0)\\
&= J_{1/2}(u^-_i ,  p_g(z_0))\pi_{L, \psi}(u^-_i)\theta_{L, X^{\ast}}\Big(J_{1/2}(p_g,z_0)\pi_{X^{\ast}, \psi} ( p_g)A\Big)( 0) \\
&\stackrel{\textrm{Case  7 }(\ref{ui-})}{=} J_{1/2}(u^-_i , z_g)m_{X^{\ast}}(u^-_i)^{-1} \theta_{L, X^{\ast}}\Big(J_{1/2}(p_g,z_0)\pi_{X^{\ast}, \psi} ( p_g)A\Big)(\tfrac{1}{2} e_i)\\
&=\sum_{n\in \Z^m}J_{1/2}(\widetilde{u}_i^- , z_g) t_{u_i^-}m_{X^{\ast}}(u^-_i)^{-1} e^{i   \pi ( n+\tfrac{1}{2}\epsilon_i)z_g ( n+\tfrac{1}{2}\epsilon_i)^T }.
\end{align*}
\begin{align*}
&\theta_{3/2}(\widetilde{u}_i^- \widetilde{g})\\
&=\theta_{3/2}(u^-_i  p_g)\\
&=\theta_{L, X^{\ast}}\Big(J_{3/2}(u^-_i p_g, z_0)\pi_{X^{\ast}, \psi} (u^-_i p_g)B\Big)(0)\\
&= J_{3/2}(u^-_i ,  p_g(z_0))\pi_{L, \psi}(u^-_i)\theta_{L, X^{\ast}}\Big(J_{3/2}(p_g,z_0)\pi_{X^{\ast}, \psi} ( p_g)B\Big)( 0) \\
&\stackrel{\textrm{Case  7 }(\ref{ui-})}{=} J_{3/2}(u^-_i , z_g)m_{X^{\ast}}(u^-_i)^{-1} \theta_{L, X^{\ast}}\Big(J_{3/2}(p_g,z_0)\pi_{X^{\ast}, \psi} ( p_g)B\Big)(\tfrac{1}{2} e_i)\\
&=\sum_{n\in \Z^m}J_{3/2}(\widetilde{u}_i^- , z_g) t_{u_i^-}m_{X^{\ast}}(u^-_i)^{-1}( n+\tfrac{1}{2}\epsilon_i)^T  e^{i   \pi ( n+\tfrac{1}{2}\epsilon_i)z_g ( n+\tfrac{1}{2}\epsilon_i)^T }.
\end{align*}
\subsubsection{$\widetilde{M_q^{(j,k)}}$} Let $u=M_q^{(j,k)}$,  $t_u=t_{M_q^{(j,k)}}$ and $\widetilde{u}=(u,t_u)$.
\begin{align*}
&\theta_{L, X^{\ast}}(f)(-\tfrac{1}{2} e_j^{\ast}-\tfrac{1}{2} e_k^{\ast}-\tfrac{1}{2} e_j-\tfrac{1}{2} e_k)\\
&=\sum_{l\in L\cap X}f(l-\tfrac{1}{2} e_j-\tfrac{1}{2} e_k)\psi(\langle l,-\tfrac{1}{2} e_j^{\ast}-\tfrac{1}{2} e_k^{\ast} \rangle)\psi(\tfrac{\langle -\tfrac{1}{2} e_j-\tfrac{1}{2} e_k, -\tfrac{1}{2} e_j^{\ast}-\tfrac{1}{2} e_k^{\ast}\rangle}{2})  \\
&=\sum_{n=(n_1,\cdots,n_m)\in \Z^m}f( n-\tfrac{1}{2}(\epsilon_j+\epsilon_k))\psi(-\tfrac{1}{2} n_j-\tfrac{1}{2} n_k) \psi(\tfrac{1}{4}).
 \end{align*}
 \begin{align*}
 &\theta_{L, X^{\ast}}\Big(J_{1/2}(p_g,z_0)\pi_{X^{\ast}, \psi} ( p_g)A\Big)(-\tfrac{1}{2} e_j^{\ast}-\tfrac{1}{2} e_k^{\ast}-\tfrac{1}{2} e_j-\tfrac{1}{2} e_k)\\
 &=\sum_{n\in \Z^m}\Big(J_{1/2}(p_g,z_0)\pi_{X^{\ast}, \psi} ( p_g)A\Big)(  n-\tfrac{1}{2}(\epsilon_j+\epsilon_k)) \psi(-\tfrac{1}{2} n_j-\tfrac{1}{2} n_k) \psi(\tfrac{1}{4})\\
 &=\sum_{n\in \Z^m} \psi(\tfrac{1}{4}) (-1)^{n_j+n_k} e^{i   \pi ( n-\tfrac{1}{2}\epsilon_j-\tfrac{1}{2}\epsilon_k)z_g ( n-\tfrac{1}{2}\epsilon_j-\tfrac{1}{2}\epsilon_k)^T }.
 \end{align*}
 \begin{align*}
 &\theta_{L, X^{\ast}}\Big(J_{3/2}(p_g,z_0)\pi_{X^{\ast}, \psi} ( p_g)B\Big)(-\tfrac{1}{2} e_j^{\ast}-\tfrac{1}{2} e_k^{\ast}-\tfrac{1}{2} e_j-\tfrac{1}{2} e_k)\\
 &=\sum_{n\in \Z^m}\Big(J_{3/2}(p_g,z_0)\pi_{X^{\ast}, \psi} ( p_g)B\Big)(  n-\tfrac{1}{2}(\epsilon_j+\epsilon_k)) \psi(-\tfrac{1}{2} n_j-\tfrac{1}{2} n_k) \psi(\tfrac{1}{4})\\
 &=\sum_{n\in \Z^m} \psi(\tfrac{1}{4}) (-1)^{n_j+n_k}( n-\tfrac{1}{2}\epsilon_j-\tfrac{1}{2}\epsilon_k)^T e^{i   \pi ( n-\tfrac{1}{2}\epsilon_j-\tfrac{1}{2}\epsilon_k)z_g ( n-\tfrac{1}{2}\epsilon_j-\tfrac{1}{2}\epsilon_k)^T }.
 \end{align*}
 \begin{align*}
\theta_{1/2}(\widetilde{u} \widetilde{g})&=\theta_{1/2}(u p_g)\\
&=\theta_{L, X^{\ast}}\Big(J_{1/2}(u p_g, z_0)\pi_{X^{\ast}, \psi} (u p_g)A\Big)(0)\\
&= J_{1/2}(u ,  p_g(z_0))\pi_{L, \psi}(u)\theta_{L, X^{\ast}}\Big(J_{1/2}(p_g,z_0)\pi_{X^{\ast}, \psi} ( p_g)A\Big)( 0) \\
&\stackrel{\textrm{Case  9 }(\ref{ui-})}{=} \psi(-\tfrac{1}{8})J_{1/2}(u , z_g) \theta_{L, X^{\ast}}\Big(J_{1/2}(p_g,z_0)\pi_{X^{\ast}, \psi} ( p_g)A\Big)(-\tfrac{1}{2} e_j^{\ast}-\tfrac{1}{2} e_k^{\ast}-\tfrac{1}{2} e_j-\tfrac{1}{2} e_k)\\
&=\sum_{n\in \Z^m}J_{1/2}(\widetilde{u} , z_g) t_{u}m_{X^{\ast}}(u)^{-1} (-1)^{n_j+n_k} e^{i   \pi ( n-\tfrac{1}{2}\epsilon_j-\tfrac{1}{2}\epsilon_k)z_g ( n-\tfrac{1}{2}\epsilon_j-\tfrac{1}{2}\epsilon_k)^T }.
\end{align*}
\begin{align*}
\theta_{3/2}(\widetilde{u} \widetilde{g})&=\theta_{3/2}(u p_g)\\
&=\theta_{L, X^{\ast}}\Big(J_{3/2}(u p_g, z_0)\pi_{X^{\ast}, \psi} (u p_g)B\Big)(0)\\
&= J_{3/2}(u ,  p_g(z_0))\pi_{L, \psi}(u)\theta_{L, X^{\ast}}\Big(J_{3/2}(p_g,z_0)\pi_{X^{\ast}, \psi} ( p_g)B\Big)( 0) \\
&\stackrel{\textrm{Case  9 }(\ref{ui-})}{=} \psi(-\tfrac{1}{8})J_{3/2}(u , z_g) \theta_{L, X^{\ast}}\Big(J_{3/2}(p_g,z_0)\pi_{X^{\ast}, \psi} ( p_g)B\Big)(-\tfrac{1}{2} e_j^{\ast}-\tfrac{1}{2} e_k^{\ast}-\tfrac{1}{2} e_j-\tfrac{1}{2} e_k)\\
&=\sum_{n\in \Z^m}J_{3/2}(\widetilde{u} , z_g) t_{u}m_{X^{\ast}}(u)^{-1} (-1)^{n_j+n_k} ( n-\tfrac{1}{2}\epsilon_j-\tfrac{1}{2}\epsilon_k)^Te^{i   \pi ( n-\tfrac{1}{2}\epsilon_j-\tfrac{1}{2}\epsilon_k)z_g ( n-\tfrac{1}{2}\epsilon_j-\tfrac{1}{2}\epsilon_k)^T }.
\end{align*}
\subsubsection{} Recall the notations $M_q$, $ M_q^{i}$, $M_q^{(j,k)}$, $\mathcal{S}^q_0=\{ i_1, \cdots, i_r\}$, $\mathcal{S}^q_1=\{(j_1,k_1), \cdots, (j_s, k_s)\}$ from Sections \ref{setminus}--\ref{modification11}. Write $M_q=(\prod_{i\in \mathcal{S}^q_0} M_q^{i}) \cdot (\prod_{ (j,k)\in\mathcal{S}^q_1}M_q^{(j,k)})$, with $M_q^{i}=1_{2m}$,  $u_i$ or $u_i^-$, $M_q^{(j,k)}=\iota_{(j,k)}(u_{1111}^1u_{1111}^2)$. 
For each $M_q$, let us define:
\begin{itemize}
\item $\underline{e}=(e_1, \cdots, e_m)$ and $\underline{e}^{\ast}=(e_1^{\ast}, \cdots, e_m^{\ast})$.
\item $m_q=(k_1, \cdots, k_m)$, for $k_i=\left\{ \begin{array}{ll} 1 & \textrm{ if }  i\in \mathcal{S}^q_0, M_q^i= u_i,\\
-1& \textrm{ if } i=j \textrm{ or } k, (j,k)\in \mathcal{S}^q_1, \\  0 &  \textrm{ if } i\in \mathcal{S}^q_0,   M_q^{i}=1_{2m}  \textrm{ or }u_i^- .\end{array}\right.$
\item $\epsilon_{q}=(a_1, \cdots, a_m)$, for $a_i=\left\{ \begin{array}{ll} 1 & \textrm{ if } i\in \mathcal{S}^q_0, M_q^i=u_i^-,\\ -1& \textrm{ if } i=j \textrm{ or } k, (j,k)\in \mathcal{S}^q_1, \\ 0 &  \textrm{ if } i\in \mathcal{S}^q_0, M_q^{i}=1_{2m}, \textrm{ or }u_i. \end{array}\right.$
\item $m_{X^{\ast}}(q)= m_{X^{\ast}}(M_q)=\big(\prod_{i\in \mathcal{S}^q_0} m_{X^{\ast}}(M_q^{i}) \big)\cdot \big(\prod_{ (j,k)\in\mathcal{S}^q_1}m_{X^{\ast}}(M_q^{(j,k)})\big)$.
\end{itemize}

For $\widetilde{M_q}=(M_q, t_{M_q})\in  \widetilde{\Sp}^{z_0}_{2m}(\R)$, decompose it as: 
$$\widetilde{M}_q=(\prod_{i\in \mathcal{S}^q_0} \widetilde{M}_q^{i}) \cdot (\prod_{ (j,k)\in\mathcal{S}^q_1}\widetilde{M}_q^{(j,k)}),$$ 
where $\widetilde{M}_q^{i}=(M_q^{i},t_{M_q^{i}})$ and $\widetilde{M}_q^{(j,k)}=(M_q^{(j,k)},t_{M_q^{(j,k)}})$. Given the cocycle values between any two distinct elements of  $M_q^{i}, M_q^{(j,k)}$ are all $1$, we have:
 \begin{align*}
  t_{M_q}=(\prod_{i\in \mathcal{S}^q_0}t_{M_q^{i}}) \cdot( \prod_{ (j,k)\in\mathcal{S}^q_1}t_{M_q^{(j,k)}}).
  \end{align*}
 \begin{align*}
&\theta_{1/2}(\widetilde{M_q} \widetilde{g})\\
&=\theta_{L, X^{\ast}}\Big(  J_{1/2}(\widetilde{M_q}p_g, z_0)\pi_{X^{\ast}, \psi} (\widetilde{M_q}p_g)A\Big)(0)\\
&= J_{1/2}(\widetilde{M_q},z_g)\pi_{L, \psi} (\widetilde{M_q})\theta_{L, X^{\ast}}\Big(  J_{1/2}(p_g, z_0)\pi_{X^{\ast}, \psi} (p_g)A\Big)(0)\\
&= J_{1/2}(\widetilde{M_q},z_g)\prod_{i\in \mathcal{S}^q_0}\prod_{ (j,k)\in\mathcal{S}^q_1} \pi_{L, \psi} (\widetilde{M}^i_q) \pi_{L, \psi} (\widetilde{M}_q^{(j,k)})\theta_{L, X^{\ast}}\Big(  J_{1/2}(p_g, z_0)\pi_{X^{\ast}, \psi} (p_g)A\Big)(0)\\
&= J_{1/2}(\widetilde{M_q},z_g)m_{X^{\ast}}(q)^{-1}t_{M_q}\theta_{L, X^{\ast}}\Big(  J_{1/2}(p_g, z_0)\pi_{X^{\ast}, \psi} (p_g)A\Big)(\tfrac{1}{2} \epsilon_{q} \cdot \underline{e}^T+\tfrac{1}{2} m_{q} \cdot \underline{e}^{\ast T})\\
&= J_{1/2}(\widetilde{M_q},z_g)m_{X^{\ast}}(q)^{-1}t_{M_q}\sum_{n\in \Z^m}(-1)^{m_q \cdot n^T}e^{i  \pi (n+\tfrac{1}{2}\epsilon_{q}) z_g (n+\tfrac{1}{2}\epsilon_{q})^T}.
\end{align*}
\begin{align*}
&\theta_{3/2}(\widetilde{M_q} \widetilde{g})\\
&=\theta_{L, X^{\ast}}\Big(  J_{3/2}(\widetilde{M_q}p_g, z_0)\pi_{X^{\ast}, \psi} (\widetilde{M_q}p_g)B\Big)(0)\\
&= J_{3/2}(\widetilde{M_q},z_g)\pi_{L, \psi} (\widetilde{M_q})\theta_{L, X^{\ast}}\Big(  J_{3/2}(p_g, z_0)\pi_{X^{\ast}, \psi} (p_g)B\Big)(0)\\
&= J_{3/2}(\widetilde{M_q},z_g)\prod_{i=1}^m \pi_{L, \psi} (\widetilde{M}^i_q)\theta_{L, X^{\ast}}\Big(  J_{3/2}(p_g, z_0)\pi_{X^{\ast}, \psi} (p_g)B\Big)(0)\\
&= J_{3/2}(\widetilde{M_q},z_g)m_{X^{\ast}}(q)^{-1}t_{M_q}\theta_{L, X^{\ast}}\Big(  J_{3/2}(p_g, z_0)\pi_{X^{\ast}, \psi} (p_g)B\Big)(\tfrac{1}{2} \epsilon_{q} \cdot \underline{e}^T+\tfrac{1}{2} m_{q} \cdot \underline{e}^{\ast T})\\
&= J_{3/2}(\widetilde{M_q},z_g)m_{X^{\ast}}(q)^{-1}t_{M_q}\sum_{n\in \Z^m}(-1)^{m_q \cdot n^T}(n+\tfrac{1}{2}\epsilon_{q})^Te^{i  \pi (n+\tfrac{1}{2}\epsilon_{q})^T z_g (n+\tfrac{1}{2}\epsilon_{q})^T}.
\end{align*}
For $z\in \mathbb{H}_m$,  we define:
\begin{align*}
\theta_{1/2}^{\widetilde{M_q}}(z)
&=m_{X^{\ast}}(q)^{-1}t_{M_q}\sum_{n=(n_1,\cdots,n_m)\in \Z^m} (-1)^{m_q \cdot n^T}e^{i  \pi (n+\tfrac{1}{2}\epsilon_{q}) z (n+\tfrac{1}{2}\epsilon_{q})^T}.\\
& \\
\theta_{3/2}^{\widetilde{M_q}}(z)&=m_{X^{\ast}}(q)^{-1}t_{M_q}\sum_{n=(n_1,\cdots,n_m)\in \Z^m} (-1)^{m_q \cdot n^T}(n+\tfrac{1}{2}\epsilon_{q})^Te^{i  \pi (n+\tfrac{1}{2}\epsilon_{q}) z (n+\tfrac{1}{2}\epsilon_{q})^T}.
\end{align*}
Then:
 \begin{itemize}
 \item $\theta_{1/2}(\widetilde{M_q}z)=J_{1/2}(\widetilde{M_q}, z) \theta_{1/2}^{\widetilde{M_q}}(z)$.
 \item $\theta_{3/2}(\widetilde{M_q}z)=J_{3/2}(\widetilde{M_q}, z)\theta_{3/2}^{\widetilde{M_q}}(z)$.
 \end{itemize}
 Let $\widetilde{r}=(r,t_r)\in \widetilde{\Gamma}^{z_0}_m(1,2)$.  Let us define:
\[\widetilde{\lambda}: \widetilde{\Gamma}^{z_0}_m(1,2) \longrightarrow T; \widetilde{r}\longmapsto t_r \widetilde{\beta}^{-1}(r).\]
\begin{lemma}
$\widetilde{\lambda}$ is a character of $\widetilde{\Gamma}^{z_0}_m(1,2)$.
\end{lemma}
\begin{proof}
Let $\widetilde{s}=(s,t_s)\in \widetilde{\Gamma}^{z_0}_m(1,2)$. Then:
\[\widetilde{\lambda}(\widetilde{r})\widetilde{\lambda}(\widetilde{s})=t_r \widetilde{\beta}^{-1}(r)t_s \widetilde{\beta}^{-1}(s);\]
\[\widetilde{r}\widetilde{s}=[rs, t_rt_s \beta_{z_0}(r^{h_{-1}},s^{h_{-1}})];\]
\begin{align*}
\widetilde{\lambda}(\widetilde{r}\widetilde{s})&=t_rt_s \beta_{z_0}(r^{h_{-1}},s^{h_{-1}})\widetilde{\beta}^{-1}(rs)\\
&=t_rt_s \widetilde{c}_{X^{\ast}}(r,s)\widetilde{\beta}^{-1}(rs)\\
&\stackrel{\textrm{Lem.} \ref{widetildec}}{=}t_rt_s \widetilde{\beta}^{-1}(r)\widetilde{\beta}^{-1}(s).
\end{align*}
\end{proof}

\subsubsection{}\label{gamma12433}
Let $\widetilde{r}=(r,t_r)\in \widetilde{\Gamma}^{z_0}_m(1,2)$. Let $z=z_g\in \mathbb{H}_m$.  Write $\widetilde{M_q}\widetilde{r}=\widetilde{s}\widetilde{M_p}$, for some $\widetilde{M_p}\in \mathcal{M}_{2m}$ , $\widetilde{s}\in \widetilde{\Gamma}^{z_0}_m(1,2)$.  Then:
\begin{align*}
&\theta_{1/2}^{\widetilde{M_q}}(\widetilde{r}z)\\
&=J_{1/2}(\widetilde{M_q},\widetilde{r}z)^{-1}\theta_{1/2}(\widetilde{M_q}\widetilde{r}z)\\
&=J_{1/2}(\widetilde{M_q},\widetilde{r}z)^{-1}\theta_{1/2}(\widetilde{s}\widetilde{M_p}z)\\
&=J_{1/2}(\widetilde{M_q},\widetilde{r}z)^{-1}J_{1/2}(\widetilde{s} ,\widetilde{M_p}(z)) t_s \widetilde{\beta}(s)^{-1} \theta_{1/2}( \widetilde{M_p}z)\\
&=J_{1/2}(\widetilde{M_q},\widetilde{r}z)^{-1}J_{1/2}(\widetilde{s} ,\widetilde{M_p}(z))J_{1/2}(\widetilde{M_p}, z) t_s \widetilde{\beta}(s)^{-1} \theta_{1/2}^{M_p}(z)\\
&=J_{1/2}(\widetilde{r},z)\widetilde{\lambda}(\widetilde{s}) \theta_{1/2}^{\widetilde{M_p}}(z).
\end{align*}
\begin{align*}
&\theta_{3/2}^{\widetilde{M_q}}(\widetilde{r}z)\\
&=J_{3/2}(\widetilde{M_q},\widetilde{r}z)^{-1}\theta_{3/2}(\widetilde{M_q}\widetilde{r}z)\\
&=J_{3/2}(\widetilde{M_q},\widetilde{r}z)^{-1}\theta_{3/2}(\widetilde{s}\widetilde{M_p}z)\\
&=J_{3/2}(\widetilde{M_q},\widetilde{r}z)^{-1}J_{3/2}(\widetilde{s} ,\widetilde{M_p}(z)) t_s \widetilde{\beta}(s)^{-1} \theta_{3/2}( \widetilde{M_p}z)\\
&=J_{3/2}(\widetilde{M_q},\widetilde{r}z)^{-1}J_{3/2}(\widetilde{s} ,\widetilde{M_p}(z))J_{3/2}(\widetilde{M_p}, z) t_s \widetilde{\beta}(s)^{-1} \theta_{3/2}^{\widetilde{M_p}}(z)\\
&=J_{3/2}(\widetilde{r},z)\widetilde{\lambda}(\widetilde{s}) \theta_{3/2}^{\widetilde{M_p}}(z).
\end{align*}
\subsubsection{}\label{Mqpr} Let $\widetilde{M_q}, \widetilde{M_p}\in \mathcal{M}_{2m}$. Write $\widetilde{M_q}\widetilde{M_p}=\widetilde{s}\widetilde{M_r}$, for some $M_r\in \mathcal{M}_{2m}$ , $\widetilde{s}\in \widetilde{\Gamma}^{z_0}_m(1,2)$.  
\begin{align*}
&\theta_{1/2}^{\widetilde{M_q}}(\widetilde{M_p}z)\\
&=J_{1/2}(\widetilde{M_q},\widetilde{M_p}z)^{-1}\theta_{1/2}(\widetilde{M_q}\widetilde{M_p}z)\\
&=J_{1/2}(\widetilde{M_q},\widetilde{M_p}z)^{-1}\theta_{1/2}(\widetilde{s}\widetilde{M_r}z)\\
&=J_{1/2}(\widetilde{M_q},\widetilde{M_p}z)^{-1}J_{1/2}(\widetilde{s} ,\widetilde{M_r}(z)) t_s \widetilde{\beta}(s)^{-1} \theta_{1/2}( \widetilde{M_r}z)\\
&=J_{1/2}(\widetilde{M_q},\widetilde{M_p}z)^{-1}J_{1/2}(\widetilde{s} ,\widetilde{M_r}(z))J_{1/2}(\widetilde{M_r}, z) t_s \widetilde{\beta}(s)^{-1} \theta_{1/2}^{\widetilde{M_r}}(z)\\
&=J_{1/2}(\widetilde{M_p},z)\lambda(\widetilde{s}) \theta_{1/2}^{\widetilde{M_r}}(z).
\end{align*}
\begin{align*}
&\theta_{3/2}^{\widetilde{M_q}}(\widetilde{M_p}z)\\
&=J_{3/2}(\widetilde{M_q},\widetilde{M_p}z)^{-1}\theta_{3/2}(\widetilde{M_q}\widetilde{M_p}z)\\
&=J_{3/2}(\widetilde{M_q},\widetilde{M_p}z)^{-1}\theta_{3/2}(\widetilde{s}\widetilde{M_r}z)\\
&=J_{3/2}(\widetilde{M_q},\widetilde{M_p}z)^{-1}J_{3/2}(\widetilde{s} ,\widetilde{M_r}(z)) t_s \widetilde{\beta}(s)^{-1} \theta_{3/2}( \widetilde{M_r}z)\\
&=J_{3/2}(\widetilde{M_q},\widetilde{M_p}z)^{-1}J_{3/2}(\widetilde{s} ,\widetilde{M_r}(z))J_{3/2}(\widetilde{M_r}, z) t_s \widetilde{\beta}(s)^{-1} \theta_{3/2}^{\widetilde{M_r}}(z)\\
&=J_{3/2}(\widetilde{M_p},z)\lambda(\widetilde{s}) \theta_{3/2}^{\widetilde{M_r}}(z).
\end{align*}

\subsubsection{}
Recall notations from Sections \ref{setminus}--\ref{modification11}. Let $M_{q_1}, \cdots, M_{q_{n}}$ be  the sets of $\mathcal{M}_{2m}$ corresponding to each element  $q_1, \cdots, q_n$ in $Q_0^{-1}(0)$.  Let $n=|Q_0^{-1}(0)|=2^{2m-1} + 2^{m-1}$. For each $q_i$, we choose an element $ \widetilde{M_{q_i}}\in  \widetilde{\Sp}^{z_0}_{2m}(\Z)$. Let us define:
$$\widetilde{\gamma}=\Ind_{\widetilde{\Gamma}^{z_0}_m(1,2)}^{\widetilde{\Sp}^{z_0}_{2m}(\Z)} \widetilde{\lambda}^{-1}, \widetilde{M}=\Ind_{\widetilde{\Gamma}^{z_0}_m(1,2)}^{\widetilde{\Sp}^{z_0}_{2m}(\Z)} \C.$$
Then  $\widetilde{M}$ consists of the functions $f: \widetilde{\Sp}^{z_0}_{2m}(\Z) \longrightarrow \C$ such that $f(\widetilde{r}\widetilde{g})=\widetilde{\lambda}(\widetilde{r})^{-1}f(\widetilde{g})$, for $\widetilde{r}\in\widetilde{\Gamma}^{z_0}_m(1,2)$ and $\widetilde{g}\in \widetilde{\Sp}^{z_0}_{2m}(\Z)$.  Let $\widetilde{e}_i$ denote the function of $\widetilde{M}$, supported on $\widetilde{\Gamma}^{z_0}_m(1,2) \widetilde{M_{q_i}}$ and $\widetilde{e}_i(\widetilde{M_{q_i}})=1$. Then $\widetilde{M}=\oplus_{i=1}^n \C \widetilde{e}_i$.  
\begin{itemize}
 \item \begin{itemize}
\item[(a)]  $\widetilde{M_{q_i}} \widetilde{r}=\widetilde{s}\widetilde{M_{q_j}}$, for some $\widetilde{s}\in \widetilde{\Gamma}^{z_0}_m(1,2)$ iff $\widetilde{\Gamma}^{z_0}_m(1,2) \widetilde{M_{q_i}} \widetilde{r}=\widetilde{\Gamma}^{z_0}_m(1,2) \widetilde{M_{q_j}}$.
\item[(b)] If the above condition holds, then $ [\widetilde{\gamma}(\widetilde{r})(\widetilde{e}_j)](\widetilde{M_{q_i}})=\widetilde{\lambda}(\widetilde{s})^{-1}$, and $ \supp [\widetilde{\gamma}(\widetilde{r})](\widetilde{e}_j)=\widetilde{\Gamma}^{z_0}_m(1,2) \widetilde{M_{q_i}}$. So $[\widetilde{\gamma}(\widetilde{r})](\widetilde{e}_j)=\widetilde{\lambda}(\widetilde{s})^{-1}\widetilde{e}_i$, and $[\widetilde{\gamma}(\widetilde{r}^{-1})](\widetilde{e}_i)=\widetilde{\lambda}(\widetilde{s})\widetilde{e}_j$.
\end{itemize}
 \item   
 \begin{itemize}
\item[(c)] $\widetilde{M_{q_i}}\widetilde{M_{q_j}}=\widetilde{s}\widetilde{M_{q_r}}$, for some $\widetilde{s}\in \widetilde{\Gamma}^{z_0}_m(1,2)$ iff $[\widetilde{\Gamma}^{z_0}_m(1,2)\widetilde{M_{q_i}}]\widetilde{M_{q_j}}=\widetilde{\Gamma}^{z_0}_m(1,2)\widetilde{M_{q_r}}$.
    \item [(d)] If the above condition holds, then $ [\widetilde{\gamma}(\widetilde{M_{q_j}})(\widetilde{e}_r)](\widetilde{M_{q_i}})=\widetilde{\lambda}(\widetilde{s})^{-1}$,  and $ \supp \widetilde{\gamma}(\widetilde{M_{q_j}})(\widetilde{e}_r)=\widetilde{\Gamma}^{z_0}_m(1,2) \widetilde{M_{q_i}}$.  So $[\widetilde{\gamma}(\widetilde{M_{q_j}})](\widetilde{e}_r)=\widetilde{\lambda}(\widetilde{s})^{-1}\widetilde{e}_i$, and  $[\widetilde{\gamma}(\widetilde{M_{q_j}}^{-1})](\widetilde{e}_i)=\widetilde{\lambda}(\widetilde{s})\widetilde{e}_r$.
    \end{itemize} 
    \item Under the basis $\{\widetilde{e}_1, \cdots, \widetilde{e}_n\}$, we identity $\widetilde{M}$ with $M_{1n}(\C)$, and write 
    $\widetilde{\gamma}(\widetilde{g})(\widetilde{e}_1, \cdots, \widetilde{e}_n)=(\widetilde{e}_1, \cdots, \widetilde{e}_n)M(\widetilde{g})$, for some $n\times n$-matrix $M(\widetilde{g})$. Then 
    \[\begin{array}{rcll}
    \widetilde{\gamma}: & \widetilde{\Sp}^{z_0}_{2m}(\Z) & \longrightarrow &  M_n(\C);\\
                        &\widetilde{g}                    & \longmapsto     & M(\widetilde{g}),
    \end{array}\] gives a matrix representation. 
\end{itemize}

We now define:
$$\Theta_{1/2}:  \mathbb{H}_m \longrightarrow M_{1n}(\C); z\longmapsto (\theta_{1/2}^{\widetilde{M_{q_1}}}(z), \cdots, \theta_{1/2}^{\widetilde{M_{q_n}}}(z))$$
$$\Theta_{3/2}:  \mathbb{H}_m \longrightarrow M_{mn}(\C); z\longmapsto (\theta_{3/2}^{\widetilde{M_{q_1}}}(z), \cdots, \theta_{3/2}^{\widetilde{M_{q_n}}}(z))$$
\begin{theorem}
Let $\widetilde{r}=(r,t_r)\in \widetilde{\Sp}^{z_0}_{2m}(\Z)$ with  $r=\begin{pmatrix}a& b\\ c&d\end{pmatrix}$. Then:
\begin{itemize}
\item $\Theta_{1/2}(\widetilde{r}z)=J_{1/2}(\widetilde{r},z)\Theta_{1/2}(z)\widetilde{\gamma}(\widetilde{r}^{-1})=\pm\sqrt{\det(cz+d)}\Theta_{1/2}(z)\widetilde{\gamma}(\widetilde{r}^{-1})$.
\item $\Theta_{3/2}(\widetilde{r}z)=J_{3/2}(\widetilde{r},z)\Theta_{3/2}(z)\widetilde{\gamma}(\widetilde{r}^{-1})=\pm \sqrt{\det(cz+d)}(cz+d)\Theta_{3/2}(z)\widetilde{\gamma}(\widetilde{r}^{-1})$.
\end{itemize}
\end{theorem}
\begin{proof}
1a) If $\widetilde{r}\in \widetilde{\Gamma}^{z_0}_{m}(1,2)$,  then by the results of Section \ref{gamma12433} and the above (a)(b), the equality holds.\\
1b)  If  $\widetilde{r}=\widetilde{M_{q}}$, by the results of Section \ref{Mqpr} and the above (c)(d), the equality holds.\\
1c) If $\widetilde{r}=\widetilde{s}\widetilde{M_{q}}$, for some $\widetilde{s}\in \widetilde{\Gamma}^{z_0}_{m}(1,2)$, then 
\begin{align*}
\Theta_{1/2}(\widetilde{r}z)&=\Theta_{1/2}(\widetilde{s}\widetilde{M_{q}}z)\\
&=J_{1/2}(\widetilde{s},\widetilde{M_{q}}z)\Theta_{1/2}(\widetilde{M_{q}}z)\widetilde{\gamma}(\widetilde{s}^{-1})\\
&=J_{1/2}(\widetilde{s},\widetilde{M_{q}}z)J_{1/2}(\widetilde{M_{q}},z)\Theta_{1/2}(z)\widetilde{\gamma}(\widetilde{M_{q}}^{-1})\widetilde{\gamma}(\widetilde{s}^{-1})\\
&=J_{1/2}(\widetilde{r},z)\Theta_{1/2}(z)\widetilde{\gamma}(\widetilde{r}^{-1}).
\end{align*}
The proof of Part (2)  is similar.
\end{proof}
\subsection{The finite-dimensional representation I}\label{finitI}
\subsubsection{} Let's recall the notation for $\widetilde{\Sp}_{2m}(\R)$ from Section \ref{schod} and for  $\Mp^{z_0}_{2m}(\R)$ from Section \ref{autoI}.   By Lemma \ref{twosame}, these two groups have the same associated cocycles but differ only in their centers. Therefore,   $\widetilde{\Sp}_{2m}(\R)$ is a subgroup of $\Mp^{z_0}_{2m}(\R)$.  Note that $\widetilde{\Sp}_{2m}^{z_0}(\R)$ is another subgroup of $\Mp^{z_0}_{2m}(\R)$. It consists of elements $(g, t)$ with $t=\pm m_{X^{\ast}}(g)$, and $g\in \Sp_{2m}(\R)$. Since $\pm m_{X^{\ast}}(g)\in \mu_8$, we have:
$$\widetilde{\Sp}_{2m}^{z_0}(\R) \subseteq \widetilde{\Sp}_{2m}(\R) \subseteq \Mp^{z_0}_{2m}(\R).$$
For $\widetilde{g}\in  \widetilde{\Sp}_{2m}(\R) $, let us write $ \widetilde{g}=\widetilde{g}^{z_0} \epsilon$, for some $\widetilde{g}^{z_0}=(g,t_g)\in \widetilde{\Sp}_{2m}^{z_0}(\R)$ and $\epsilon \in \mu_8$. We now define the canonical action of $ \widetilde{\Sp}_{2m}(\R)$ on $\mathbb{H}_m$ via $\Sp_{2m}(\R)$. 
Define:
\begin{itemize}
\item $J_{1/2}(\widetilde{g},z)= \epsilon^{-1} J_{1/2}(\widetilde{g}^{z_0},z)= \epsilon^{-1} t_g^{-1} \cdot \epsilon(g; z,z_0) \cdot |\det J(g, z)|^{1/2}$.
\item  $J_{3/2}(\widetilde{g},z)=\epsilon^{-1} J_{3/2}(\widetilde{g}^{z_0},z) =\epsilon^{-1} J_{1/2}(\widetilde{g}^{z_0},z) J(g,z)$.
\end{itemize}
For $\widetilde{g}_1, \widetilde{g}_2\in \widetilde{\Sp}_{2m}(\R) $, we have:
\begin{itemize}
\item[(1)] $J_{1/2}(\widetilde{g}_1\widetilde{g}_2, z)=J_{1/2}(\widetilde{g}_1, g_2(z))J_{1/2}(\widetilde{g}_2, z)$.
\item[(2)]$J_{3/2}(\widetilde{g}_1\widetilde{g}_2, z)=J_{3/2}(\widetilde{g}_1, g_2(z))J_{3/2}(\widetilde{g}_2, z)$.
\end{itemize}    
Let $\widetilde{\Gamma}_m(1,2)$ and $\widetilde{\Sp}_{2m}(\Z)$ denote the inverse images of $\Gamma_m(1,2)$ and $\Sp_{2m}(\Z)$ in $\widetilde{\Sp}_{2m}(\R)$, respectively. 
We extend  $\widetilde{\lambda}$ to $\widetilde{\Gamma}_m(1,2)$ by defining 
\[\widetilde{\lambda}: \widetilde{\Gamma}_m(1,2) \longrightarrow T; \quad \widetilde{r}=(r, t_r\epsilon)\longmapsto t_r \epsilon \widetilde{\beta}^{-1}(r).\]
We also extend the above $\widetilde{\gamma}$ to $\widetilde{\Sp}_{2m}(\Z)$ by defining
$$\widetilde{\gamma}=\Ind_{\widetilde{\Gamma}_m(1,2)}^{\widetilde{\Sp}_{2m}(\Z)} \widetilde{\lambda}^{-1}, \quad \widetilde{M}=\Ind_{\widetilde{\Gamma}_m(1,2)}^{\widetilde{\Sp}_{2m}(\Z)} \C.$$
For  $\widetilde{r}=(r, t)\in \widetilde{\Sp}_{2m}(\Z)$, where $r=\begin{pmatrix}a& b\\ c&d\end{pmatrix}$ and  $t\in \mu_8$, the following hold:
\begin{itemize}
\item $\Theta_{1/2}(\widetilde{r}z)=J_{1/2}(\widetilde{r},z)\Theta_{1/2}(z)\widetilde{\gamma}(\widetilde{r}^{-1})$,
\item $\Theta_{3/2}(\widetilde{r}z)=J_{3/2}(\widetilde{r},z)\Theta_{3/2}(z)\widetilde{\gamma}(\widetilde{r}^{-1})$.
\end{itemize}
We now  reduce $\widetilde{\gamma}$ to a representation of a finite group.
\subsubsection{$m\geq 2$}  
For  $g\in \Sp_{2m}(\Z)$, write $g=r M_q$, for  unique  $r\in \Gamma_m(1,2)$ and $M_q\in \mathcal{M}_{2m}$. Define the function:
$$f: \Sp_{2m}(\Z) \longmapsto \mu_8; \quad g \longmapsto \widetilde{\beta}(r) \widetilde{c}_{X^{\ast}}(r, M_q).$$
We  modify the cocycle $\widetilde{c}_{X^{\ast}}$ by $f$  to obtain a new cocycle: 
$$\widetilde{c}'_{X^{\ast}}(g_1,g_2)=\widetilde{c}_{X^{\ast}}(g_1,g_2) f(g_1)f(g_2) f(g_1g_2)^{-1}.$$
\begin{lemma}\label{crg}
Let $r, r_1,r_2\in \Gamma_m(1,2)$ and $g\in \Sp_{2m}(\Z)$. Then $\widetilde{c}'_{X^{\ast}}(r,g)=1=\widetilde{c}'_{X^{\ast}}(r_1,r_2)$.
\end{lemma}
\begin{proof}
1) By (\ref{trivgamma}), $\widetilde{c}'_{X^{\ast}}(r_1,r_2)=1$.\\
2) If $g=M_q$, then $$\widetilde{c}'_{X^{\ast}}(r,g)=\widetilde{c}_{X^{\ast}}(r,M_q) f(r)f(M_q) f(rM_q)^{-1}=\widetilde{c}_{X^{\ast}}(r,M_q)\widetilde{\beta}(r)1 \widetilde{\beta}(r)^{-1} \widetilde{c}_{X^{\ast}}(r, M_q)^{-1}=1.$$
3) For  general $g$,  write $g=r_gM_q$. Then: 
\begin{equation*}
\widetilde{c}'_{X^{\ast}}(r,g)=\widetilde{c}'_{X^{\ast}}(r,r_gM_q)=\widetilde{c}'_{X^{\ast}}(r,r_gM_q)\widetilde{c}'_{X^{\ast}}(r_g,M_q)
=\widetilde{c}'_{X^{\ast}}(r,r_g)\widetilde{c}'_{X^{\ast}}(rr_g,M_q)=1.
\end{equation*}
\end{proof}
\begin{lemma}\label{ep}
$\widetilde{c}'_{X^{\ast}}(g,-)$ defines a character of $\Gamma_m(2)$.  In particular, $\widetilde{c}'_{X^{\ast}}(g,r)=1$, for any $r\in \Gamma_m(4,8)$.
\end{lemma}
\begin{proof}
Let $g_1,g_2,g \in \Sp_{2m}(\Z)$ and $r,r_1,r_2\in \Gamma_m(2)$.
\begin{align*}
\widetilde{c}'_{X^{\ast}}(rg_1,g_2)
&=\widetilde{c}'_{X^{\ast}}(r,g_1)\widetilde{c}'_{X^{\ast}}(rg_1,g_2)\\
&=\widetilde{c}'_{X^{\ast}}(r,g_1g_2)\widetilde{c}'_{X^{\ast}}(g_1,g_2)\\
&=\widetilde{c}'_{X^{\ast}}(g_1,g_2).
\end{align*}
\begin{align*}
\widetilde{c}'_{X^{\ast}}(g,r_1r_2)
&=\widetilde{c}'_{X^{\ast}}(g,r_1r_2)\widetilde{c}'_{X^{\ast}}(r_1,r_2)\\
&=\widetilde{c}'_{X^{\ast}}(g,r_1)\widetilde{c}'_{X^{\ast}}(gr_1, r_2)\\
&=\widetilde{c}'_{X^{\ast}}(g,r_1)\widetilde{c}'_{X^{\ast}}(gr_1g^{-1} g, r_2)\\
&=\widetilde{c}'_{X^{\ast}}(g,r_1)\widetilde{c}'_{X^{\ast}}( g, r_2).
\end{align*}
By Prop.\ref{gamma2q}, $[\Gamma_m(2), \Gamma_m(2)]=\Gamma_m(4,8)$. So the last statement holds.
\end{proof}
\begin{lemma}
 $\widetilde{c}'_{X^{\ast}}(g,r)=-1$  for some $g\in  \Sp_{2m}(\Z)$ and  $r\in \Gamma_m(4) $.
\end{lemma}
\begin{proof}
By Prop.\ref{gamma2q}, the images of $u_i(4)$, $u_i^-(4)$ generate $\Gamma_m(4)/\Gamma_m(4,8)$. Let us consider $r=u_i^-(4)=\iota_i(u_1^-(4))$ and $g=M_q=\iota_i(u_1)$,  for $u_1= \begin{pmatrix} 1& 1\\ 0& 1\end{pmatrix}$, $u_1^-(4)=  \begin{pmatrix} 1&0 \\ -4& 1\end{pmatrix}$. Then:
\[\begin{pmatrix} 1& 1\\ 0& 1\end{pmatrix}\begin{pmatrix} 1& 0\\-4& 1\end{pmatrix}=\begin{pmatrix} -3& 1\\-4& 1\end{pmatrix}=\begin{pmatrix} -3 & 4\\ -4& 5\end{pmatrix}\cdot \begin{pmatrix}1 & 1\\ 0& 1\end{pmatrix};\]
\[f(M_q)=1, \quad\quad f(r)= \widetilde{\beta}(r)=  \widetilde{\beta}(u_1^-(4))=e^{\tfrac{\pi i}{4}},\]
 \begin{align*}
 f(M_qr)&=f(\iota_i(\begin{pmatrix} -3& 1\\-4& 1\end{pmatrix}))=f(\iota_i(\begin{pmatrix} -3 & 4\\ -4& 5\end{pmatrix})\cdot \iota_i( \begin{pmatrix}1 & 1\\ 0& 1\end{pmatrix}))\\
 &=\widetilde{\beta}(\iota_i(\begin{pmatrix} -3 & 4\\ -4& 5\end{pmatrix})) \widetilde{c}_{X^{\ast}}(\iota_i(\begin{pmatrix} -3 & 4\\ -4& 5\end{pmatrix}),\iota_i( \begin{pmatrix}1 & 1\\ 0& 1\end{pmatrix}))\\
 &=\widetilde{\beta}(\iota_i(\begin{pmatrix} -3 & 4\\ -4& 5\end{pmatrix}))=\widetilde{\beta}(\begin{pmatrix} -3 & 4\\ -4& 5\end{pmatrix})=\big(\tfrac{-2}{5}\big)e^{\tfrac{\pi i}{4}}=-e^{\tfrac{\pi i}{4}}.
 \end{align*}
 \begin{align*}
\widetilde{c}'_{X^{\ast}}(g,r)
&=\widetilde{c}_{X^{\ast}}(M_q,r) f(M_q)f(r)f(M_qr)^{-1}\\
&=f(M_q)f(r)f(M_qr)^{-1}\\
&=-1.
\end{align*}
\end{proof}

$\widetilde{c}'_{X^{\ast}}$ is indeed a cocycle on $\tfrac{\Sp_{2m}(\Z)}{\Gamma_m(4,8)}$.  Let $\widetilde{\Sp}'_{2m}(\Z)$, $\widetilde{\Gamma}'_m(1,2)$, $\widetilde{\big(\tfrac{\Sp_{2m}(\Z)}{\Gamma_m(4,8)}\big)}$,  $\widetilde{\big(\tfrac{\Gamma_m(1,2)}{\Gamma_m(4,8)}\big)}$ be the corresponding covering groups associated to $\widetilde{c}'_{X^{\ast}}$. Note that there exists an isomorphism: 
$$\widetilde{\kappa}: \widetilde{\Sp}'_{2m}(\Z) \longrightarrow \widetilde{\Sp}_{2m}(\Z); [g,t] \longmapsto [g, f(g)t].$$
\begin{lemma}
\begin{itemize}
\item[(a)] $\widetilde{\lambda}\circ \widetilde{\kappa}([r,t])= t$, for $[r,t]\in \widetilde{\Gamma}'_m(1,2)$.
\item[(b)] $\widetilde{\gamma}\circ \widetilde{\kappa} \simeq \Ind_{\widetilde{\Gamma}'_m(1,2)}^{\widetilde{\Sp}'_{2m}(\Z)} [\widetilde{\lambda}\circ \widetilde{\kappa}]^{-1}\simeq \Ind_{\widetilde{\Gamma}'_m(1,2)}^{\widetilde{\Sp}'_{2m}(\Z)} [1_{\Gamma_m(1,2)} \cdot \Id_{\mu_8}]^{-1}\simeq  \Ind^{\widetilde{\big(\tfrac{\Sp_{2m}(\Z)}{\Gamma_m(4,8)}\big)}}_{\widetilde{\big(\tfrac{\Gamma_m(1,2)}{\Gamma_m(4,8)}\big)}} [\widetilde{\lambda}\circ \widetilde{\kappa}]^{-1}$.
    \end{itemize}
\end{lemma}
\begin{proof}
1) $\widetilde{\lambda}\circ \widetilde{\kappa}([r,t])=\widetilde{\lambda}([r, f(r)t])=\widetilde{\lambda}([r, \widetilde{\beta}(r) t])=t$.\\
2) It is a consequence of part (1).
\end{proof}

\subsection{The finite-dimensional representation II}\label{finitII}
Recall the isomorphism: $\overline{\Sp}_{2m}(\R) \stackrel{\iota}{\to} \widetilde{\Sp}_{2m}^{z_0}(\R)$ from Lemma \ref{imbed}.  Let $\overline{\Sp}_{2m}(\Z)$, $\overline{\Gamma}_m(1,2)$ and  $\overline{\Gamma}_m(2)$ denote the inverse images of 
$ \widetilde{\Sp}^{z_0}_{2m}(\Z)$,    $\widetilde{\Gamma}^{z_0}_m(1,2)$ and   $\widetilde{\Gamma}_m(2)$ in $\overline{\Sp}_{2m}(\R)$, respectively. Let $\overline{\lambda}=\widetilde{\lambda}\circ\iota$. Recall the notation $\lambda$ from  (\ref{lambda}). Then:
\[\widetilde{\kappa}^{-1}\circ \iota: \overline{\Sp}_{2m}(\Z)  \stackrel{\iota}{\to} \widetilde{\Sp}_{2m}^{z_0}(\Z) \hookrightarrow \widetilde{\Sp}_{2m}(\Z) \stackrel{\widetilde{\kappa}^{-1}}{\longrightarrow}\widetilde{\Sp}'_{2m}(\Z).\]
\[(\widetilde{\lambda}\circ \widetilde{\kappa})\circ (\widetilde{\kappa}^{-1}\circ \iota)= \widetilde{\lambda}\circ \iota=\overline{\lambda}.\]
\begin{lemma}\label{lambdaoverlin}
$\overline{\lambda}(r,\epsilon)=\lambda(r) \epsilon$, for $(r,\epsilon)\in \overline{\Gamma}_m(1,2)$.
\end{lemma}
\begin{proof}
$\overline{\lambda}(r,\epsilon)=m_{X^{\ast}}(r)\epsilon \widetilde{\beta}^{-1}(r) =\epsilon \lambda(r)$.
\end{proof}
\begin{lemma}
$\overline{\lambda}(\overline{r})\in \mu_2$, for $\overline{r}=(r,t)\in \overline{\Gamma}_m(2)$.
\end{lemma}
\begin{proof}
By Lemma \ref{lambdaoverlin}, $\overline{\lambda}(\overline{r})=\lambda(r)t$.  According to Lemma \ref{mutirr12},  for $r_1,r_2\in \Gamma_m(2)$, we have:
\[ \lambda(r_1r_2)=\lambda(r_1)\lambda(r_2) \overline{c}_{X^{\ast}}(r_1,r_2).\]
Hence, it suffices to show the result holds for the generators $u_{ij}(2)$,  $u^-_{ij}(2)$, $v_{ij}(2)$   of $\Gamma_m(2)$.  
\begin{itemize}
\item[1)] If $r=u_{ij}(2)$ or $v_{ij}(2)$, by Example \ref{exampleP}, $\widetilde{\beta}(r)=1$. In this case, $m_{X^{\ast}}(r)=1$.
\item[2)] If $r=u^-_{ij}(2)$, with $i\neq j$, then by Example \ref{exampleineqj}, $\widetilde{\beta}(r)=1$. In this case, $j(r)=2$, $x(r)=-4$, 
$$m_{X^{\ast}}(r)=\gamma(x(r), \psi^{\tfrac{1}{2}})^{-1} \gamma(\psi^{\tfrac{1}{2}})^{-j(r)}=\gamma(-1, \psi^{\tfrac{1}{2}})^{-1} \gamma(\psi^{\tfrac{1}{2}})^{-2}=e^{\tfrac{\pi i}{2}} e^{-\tfrac{\pi i}{2}}=1.$$ 
\item[3)] If $r=u^-_{ii}(2)$, then $r=\iota_i(u^-(2))$,  for $u^-(2)=\begin{pmatrix}1 & 0\\-2& 1\end{pmatrix}\in \SL_2(\Z)$. Hence $$m_{X^{\ast}}(r) \widetilde{\beta}^{-1}(r)=m_{X^{\ast}}(u^-(2)) \widetilde{\beta}^{-1}(u^-(2))= e^{\tfrac{i \pi}{4}[\sgn(2)]}e^{\tfrac{i \pi}{4}[- \sgn(2)]}=1.$$
\end{itemize}
\end{proof}
\begin{definition}
$\overline{\gamma}\stackrel{Def.}{=} \Ind_{\overline{\Gamma}_m(1,2)}^{\overline{\Sp}_{2m}(\Z)} \overline{\lambda}^{-1} \simeq \widetilde{\gamma}\circ \iota$.
\end{definition}
 For $\overline{g}=(g,\epsilon)\in  \overline{\Sp}_{2m}(\R) $, $\iota(\overline{g})=[g, m_{X^{\ast}}(g)\epsilon]\in  \widetilde{\Sp}^{z_0}_{2m}(\R)$.  Let us define the canonical action of $ \overline{\Sp}_{2m}(\R)$ on $\mathbb{H}_m$ via $\Sp_{2m}(\R)$. 
Define:
\begin{itemize}
\item $J_{1/2}(\overline{g},z)=\epsilon^{-1}\cdot \epsilon(g; z,z_0) \cdot |\det J(g, z)|^{1/2}$.
\item $J_{3/2}(\overline{g},z)=J_{1/2}(\overline{g},z)J(g,z)$.
\end{itemize}
\begin{lemma}
For $\overline{g}_1, \overline{g}_2\in \overline{\Sp}_{2m}(\R) $, we have:
\begin{itemize}
\item[(1)] $J_{1/2}(\overline{g}_1\overline{g}_2, z)=J_{1/2}(\overline{g}_1, g_2(z))J_{1/2}(\overline{g}_2, z)m_{X^{\ast}}(g_1g_2)m_{X^{\ast}}(g_1)^{-1}m_{X^{\ast}}(g_2)^{-1}$.
\item[(2)]$J_{3/2}(\overline{g}_1\overline{g}_2, z)=J_{3/2}(\overline{g}_1, g_2(z))J_{3/2}(\overline{g}_2, z)m_{X^{\ast}}(g_1g_2)m_{X^{\ast}}(g_1)^{-1}m_{X^{\ast}}(g_2)^{-1}$.
\end{itemize} 
\end{lemma}  
\begin{proof}
$\overline{g}_1\overline{g}_2=(g_1,\epsilon_1)(g_2,\epsilon_2)=(g_1g_2, \epsilon_1\epsilon_2 \overline{c}_{X^{\ast}}(g_1,g_2))=\overline{(g_1g_2)}$.\\
1) \begin{align*}
&J_{1/2}(\overline{g}_1\overline{g}_2, z)\\
&=J_{1/2}(\iota(\overline{(g_1g_2)}), z) m_{X^{\ast}}(g_1g_2)\\
&=J_{1/2}(\iota(\overline{g_1})\iota(\overline{g_2}), z) m_{X^{\ast}}(g_1g_2)\\
&=J_{1/2}(\iota(\overline{g}_1), g_2(z))J_{1/2}(\iota(\overline{g}_2), z)m_{X^{\ast}}(g_1g_2)\\
&=J_{1/2}(\overline{g}_1, g_2(z))J_{1/2}(\overline{g}_2, z)m_{X^{\ast}}(g_1g_2)m_{X^{\ast}}(g_1)^{-1}m_{X^{\ast}}(g_2)^{-1}.
\end{align*}
2) It follows form part (1).
\end{proof} 
\begin{theorem}\label{mainthm112}
Let $\overline{r}=(r,\epsilon)\in \overline{\Sp}_{2m}(\Z)$ with  $r=\begin{pmatrix}a& b\\ c&d\end{pmatrix}$, $\epsilon\in \mu_2$. Then:
\begin{itemize}
\item $\Theta_{1/2}(\overline{r}z)=J_{1/2}(\overline{r},z)m_{X^{\ast}}(r)^{-1}\Theta_{1/2}(z)\overline{\gamma}(\overline{r}^{-1})=\epsilon\sqrt{\det(cz+d)}\Theta_{1/2}(z)\overline{\gamma}(\overline{r}^{-1})$.
\item $\Theta_{3/2}(\overline{r}z)=J_{3/2}(\overline{r},z)m_{X^{\ast}}(r)^{-1}\Theta_{3/2}(z)\overline{\gamma}(\overline{r}^{-1})=\epsilon \sqrt{\det(cz+d)}(cz+d)\Theta_{3/2}(z)\overline{\gamma}(\overline{r}^{-1})$.
\end{itemize}
\end{theorem}
\begin{proof}
\begin{align*}
\Theta_{1/2}(\overline{r}z)&=\Theta_{1/2}(\iota(\overline{r})z)\\
&=J_{1/2}(\iota(\overline{r}),z)\Theta_{1/2}(z)\widetilde{\gamma}(\iota(\overline{r})^{-1})\\
&=J_{1/2}(\overline{r},z)m_{X^{\ast}}(r)^{-1}\Theta_{1/2}(z)\overline{\gamma}(\overline{r})^{-1}\\
&=\epsilon^{-1} J_{1/2}(r,z)m_{X^{\ast}}(r)^{-1}\Theta_{1/2}(z)\overline{\gamma}(\overline{r})^{-1}\\
&=\epsilon\sqrt{\det(cz+d)}\Theta_{1/2}(z)\overline{\gamma}(\overline{r}^{-1}).
\end{align*}
\begin{align*}
\Theta_{3/2}(\overline{r}z)&=\Theta_{3/2}(\iota(\overline{r})z)\\
&=J_{3/2}(\iota(\overline{r}),z)\Theta_{3/2}(z)\widetilde{\gamma}(\iota(\overline{r})^{-1})\\
&=J_{3/2}(\overline{r},z)m_{X^{\ast}}(r)^{-1}\Theta_{3/2}(z)\overline{\gamma}(\overline{r})^{-1}\\
&=\epsilon^{-1}J_{3/2}(r,z)m_{X^{\ast}}(r)^{-1}\Theta_{3/2}(z)\overline{\gamma}(\overline{r})^{-1}\\
&=\epsilon \sqrt{\det(cz+d)}(cz+d)\Theta_{3/2}(z)\overline{\gamma}(\overline{r}^{-1}).
\end{align*}
\end{proof}
\labelwidth=4em
\setlength\labelsep{0pt}
\addtolength\parskip{\smallskipamount}

\end{document}